\documentclass[12pt]{amsart}
\usepackage{graphicx,color,enumerate,amssymb}
\usepackage{amsfonts,amsthm,amsmath,amssymb,graphicx,tabularx,epsf}%booktabs,
\usepackage{listings}
\usepackage{url}
\usepackage{pgf,tikz,pgfplots}
\usepackage{mathrsfs}
\usetikzlibrary{arrows}
\usepackage{graphics}
\usepackage[latin1]{inputenc}
\usepackage[T1]{fontenc}
\usepackage[english]{babel}
\usepackage{float}
\usepackage{hyperref}%make links to references
\theoremstyle{plain}% Theorem-like structures provided by amsthm.sty
\newtheorem{theorem}{Theorem}[section]
\newtheorem{lemma}[theorem]{Lemma}

\newtheorem{corollary}[theorem]{Corollary}
\newtheorem{proposition}[theorem]{Proposition}

\theoremstyle{definition}
\newtheorem{definition}[theorem]{Definition}
\newtheorem{example}[theorem]{Example}

\theoremstyle{remark}

\usepackage[body={21.6cm,27.9cm},top=3cm,left=3cm,right=3cm,bottom=3cm]{geometry}
\usepackage{amsfonts}
\usepackage{amssymb}
\usepackage{amsmath}
\usepackage{tikz}
   \usetikzlibrary{positioning,arrows,calc}
   \tikzset{
   modal/.style={>=stealth,shorten >=1pt,shorten <=1pt,auto,node distance=1.5cm,
   semithick},
   world/.style={circle,draw,minimum size=0.5cm,fill=gray!15},
   point/.style={circle,draw,inner sep=0.5mm,fill=black},
   reflexive above/.style={->,loop,looseness=7,in=120,out=60},
   reflexive below/.style={->,loop,looseness=7,in=240,out=300},
   reflexive left/.style={->,loop,looseness=7,in=150,out=210},
   reflexive right/.style={->,loop,looseness=7,in=30,out=330}
   }
\usetikzlibrary{decorations.markings}

\definecolor{david}{rgb}{0.0, 0.0, 1.0}

\newcommand{\gpp}{gpp}
\newcommand{\gpps}{gpps}
\newcommand{\ps}{(P_q, S_r)}

\newcommand{\ceilof}{\lceil\frac{n}{2} \rceil}
\newcommand{\avgdeg}[1]{2 - \frac{2}{#1} }
\newcommand{\minorder}{8}

\begin{document}
%\definecolor{rvwvcq}{rgb}{0.08235294117647059,0.396078431372549,0.7529411764705882}
\pagestyle{myheadings}

\title{Most Laplacian eigenvalues of a tree are small}
\author[D. Jacobs]{David P. Jacobs}
\address{School of Computing, Clemson University, Clemson, USA} \email{\tt
dpj@clemson.edu}
\author[E. Oliveira]{Elismar Oliveira}
\address{Instituto de Matem\'atica e Estat\'{\i}stica, UFRGS, Porto Alegre, Brazil}
\email{\tt elismar.oliveira@ufrgs.br}
\author[V. Trevisan]{Vilmar Trevisan}
\address{Instituto de Matem\'atica e Estat\'{\i}stica, UFRGS, Porto Alegre, Brazil and Dipartimento di Matematica e Applicazioni, Universit\`a `Federico II', Napoli, Italia}
\email{\tt trevisan@mat.ufrgs.br}

\begin{abstract}
We show that the number of Laplacian eigenvalues less than the average degree $\avgdeg{n}$ of a tree
having $n$ vertices is at least $\ceilof$.
\end{abstract}

\maketitle

\section{Introduction}
\label{sec:intro}

For a graph $G$ of order $n$, the Laplacian matrix of $G$ is $L=D-A$, where
$A$ is the {\em adjacency} matrix and $D$ is the diagonal
{\em degree} matrix.  The
eigenvalues of $L$, which lie in the interval $[0,n]$, are called
\emph{Laplacian eigenvalues of $G$}.
Studying the distribution of Laplacian eigenvalues of graphs is a natural and
relevant problem. It is relevant due to the many applications related  to
Laplacian matrices (see, for example \cite{Moh92}). We believe it is also a
hard problem because little is known about how the Laplacian eigenvalues are
distributed in the interval $[0,n]$.
Even though there exist results that bound the number of
Laplacian eigenvalues in
{\em subintervals} of $[0, n]$
(see for example \cite{Gro94,Gro90,Guo2011,HJT2015,Mer91,ZhouZhouDu2015}
 and the references therein)
there lacks a finer understanding of the {\em distribution}. For
instance, it is not known whether smaller eigenvalues outnumber the larger
ones,
%or whether the eigenvalues may be clustered around a point and so on.
and little known about how eigenvalues are clustered around a point \cite{Guo-2008}.

We first consider the question {\em What is a large Laplacian eigenvalue?} A
reasonable measure is to compare this eigenvalue with the average of all
eigenvalues.
Since the average of the Laplacian eigenvalues equals the average
degree $d_n = \frac{2m}{n}$ of $G$, we say that a Laplacian eigenvalue is
\emph{large} if it is greater than or equal to the average degree,
and {\em small} otherwise.
Inspired by this idea, the paper~\cite{Das16} introduces the spectral parameter
$\sigma(G)$ which counts the number of Laplacian eigenvalues greater than or
equal to $d_n$.

There is evidence that $\sigma(G)$ plays an important role in defining
structural properties of a graph $G$. For example, it is related to the clique
number $\omega$ of $G$ (the number of vertices of the largest induced  complete
subgraph of $G$) and it also gives insight about the Laplacian energy of a graph
\cite{Das16,Pir15}.  Moreover, several structural properties of a graph are
related to  $\sigma$ (see, for example \cite{Das15,Das16}).

In this paper, we are interested in the distribution of Laplacian eigenvalues
of trees. More precisely, we want to investigate $\sigma(T)$. If $I$ is a
real interval, we let $m_GI$ denote the number of Laplacian eigenvalues of
$G$ in $I$, counting multiplicities. Given a tree of order $n$, its {\em
average vertex degree} is $d_n=\avgdeg{n}$. In 2011 it was conjectured that
in any tree, at least half of the Laplacian eigenvalues were less than this
average \cite{Tre2011}. The purpose of this paper is to prove the conjecture,
which we state as the following result.
\begin{theorem}
\label{thr:main}
For any tree $T$ of order $n$,
$m_T [0, d_n ) \geq \left \lceil\frac{n}{2} \right \rceil.$
\end{theorem}
Since there are $n$ Laplacian eigenvalues bounded by $n$, this is equivalent to
$$
\sigma(T)=m_T[d_n , n]  \leq ~\left \lfloor\frac{n}{2} \right \rfloor .
$$
This property does not hold for general graphs, and complete
graphs $K_n$ provide a counter-example.
It is known that rational Laplacian eigenvalues of graphs are integers, and
when $n \geq 3$, $d_n$ is not integer so it suffices to show $m_T(d_n , n]
\leq \lfloor\frac{n}{2} \rfloor$ for all trees $T$.

The proof of this conjecture is, perhaps surprisingly, difficult.
There have been a few attempts to solve it and we summarize now some partial results.
First, computation has verified that Theorem~\ref{thr:main} holds for all
trees of order $n \leq 20$ (see the experiments by J. Carvalho \cite{Carv}.)
In the paper \cite{Tre2011}, where the conjecture was set forward, it was
proved that diameter 3 trees and caterpillars satisfy the conjecture and in
\cite{Jacobs2017} the theorem was shown to hold for all trees of diameter four.
It is known \cite{Guo2007,Braga2013} that  $m_T [0,2) \geq  \lceil\frac{n}{2} \rceil$ for all trees $T$.
If there were never eigenvalues in
$\left(2-\frac{2}{n}, 2\right)$, then Theorem~\ref{thr:main} would easily follow from this result.
However, there {\em can} exist eigenvalues in this interval.

By \cite{Carv} we assume $n \geq 8$.
For each $n \geq 8$, our proof makes use of four
{\em prototype trees} which satisfy the theorem.
If any tree can be transformed to a prototype tree
in a way that does not
decrease the number of eigenvalues above the average degree,
then Theorem~\ref{thr:main} holds.

The paper's remainder is organized as follows. In the next section, we
present a notation that is crucial for our strategy to prove our result.
To simplify the representation of trees
we introduce a concatenation operator that is executed on
{\em suns} and {\em paths} to form
{\em generalized pendant paths.}
We also define a summation operator
and {\em starlike vertices} which are based on these operators.
The main tool to prove
Theorem~\ref{thr:main} is the algorithm \texttt{Diagonalize} that counts the
number of Laplacian eigenvalues of a tree in any interval \cite{Braga2013}.
In Section~\ref{sec:stra} this algorithm is described
along with
a procedure called \texttt{Transform}, a high-level proof strategy.
In Section~\ref{sec:prop},
we show how to use the \texttt{Diagonalize} algorithm to transform a tree
in a way
that does not decrease
the number of eigenvalues above average degree, which we call a {\em proper} transformation.
In  Section~\ref{sec:prot}, we define prototype trees,
which are close to stars and depend only on the congruence $n\equiv \alpha \pmod 4$.
We prove that these prototype trees satisfy the conjecture,
however they are extreme examples in which equality is achieved in Theorem~\ref{thr:main}.
In Section~\ref{sec:red} describe a procedure
\texttt{ReduceStarVertex}
that does much of the structural transformation.
This is based on the proper transformations and since it is not obvious that
\texttt{Transform} halts we also prove its correctness in
Section~\ref{sec:red}.
All trees get reduced to some small cases that are
eventually transformed into prototype trees in Section~\ref{sec:small}.
A complete example showing a tree properly transformed into a prototype tree
is given in Section~\ref{sec:exa}.

\section{Notation} \label{sec:not}
Consider a tree $T$ with $n$ vertices, whose average degree is
$d_n=2-\frac{2}{n}$.
Recall that $m_T[0, d_n)$ is the number
Laplacian eigenvalues of $T$ which are smaller than $d_n$,
and
$\sigma(T) =
m_T(d_n , n]$, denotes the number of eigenvalues which are larger than
$d_n$. We will use the fact that $m_T[0, d_n) + \sigma(T) = n$.

The concept of a {\em pendant path} is well known, but important here. Let
$u$ be a vertex of a tree $T$ with degree $\geq 3$. Suppose that $P_q = u\,
u_1 \ldots u_q~~ (q \geq 1)$ is a path in $T$ whose internal vertices
$u_1,\ldots, u_{q-1}$ all have degree 2 in $T$, and where $u_q$ is a leaf.
Then we say that $P_q$ is a pendant path of length $q$ attached at $u$.
The following is known \cite{MOHAR2007}.
\begin{lemma}
\label{lemma-starlike} Any tree that is not a path has at least one vertex
$u$, with $\deg(u) \geq 3$, having (at least) two pendant paths.
\end{lemma}

Let $T$ be a tree with $n$ vertices, and let $u$ be a vertex of
degree at least $\ell$ of $T$ having $\ell \geq 1$ pendant paths attached at
$u$. We denote the \emph{sum} of pendant paths attached at $u$ by
$P(u)=P_{q_1}\oplus\cdots \oplus P_{q_\ell}$, as illustrated in Figure
\ref{fig:ex1}. The number of edges in each path is denoted by $\sharp
P_{q}=q$.
\begin{figure}[!ht]
  \centering
\definecolor{rvwvcq}{rgb}{0.08235294117647059,0.396078431372549,0.7529411764705882}
\begin{tikzpicture}[line cap=round,line join=round,>=triangle 45,x=1cm,y=1cm, scale=0.6, every node/.style={scale=0.7}]
\clip(-4,-4) rectangle (3,1);
\draw [line width=1pt,dash pattern=on 1pt off 1pt] (2,0.47)-- (1.34,-0.27);
\draw [line width=1pt] (1.34,-0.27)-- (0.86,-0.81);
\draw [line width=1pt] (0.86,-0.81)-- (0.08,-0.99);
\draw [line width=1pt] (0.08,-0.99)-- (-0.58,-1.17);
\draw [line width=1pt] (0.86,-0.81)-- (0.54,-1.27);
\draw [line width=1pt] (0.54,-1.27)-- (0.18,-1.75);
\draw [line width=1pt] (0.18,-1.75)-- (-0.2,-2.23);
\draw [line width=1pt] (-0.2,-2.23)-- (-0.56,-2.71);
\draw [line width=1pt] (0.86,-0.81)-- (1.22,-1.47);
\draw [line width=1pt] (0.86,-0.81)-- (0.7,-1.69);
\draw [line width=1pt] (0.7,-1.69)-- (0.58,-2.29);
\draw [line width=1pt] (0.58,-2.29)-- (0.48,-2.79);
\draw [line width=1pt] (0.48,-2.79)-- (0.38,-3.37);
\draw [line width=1pt] (0.86,-0.81)-- (0.22,-0.39);
\draw [line width=1pt] (0.22,-0.39)-- (-0.38,-0.01);
\draw (0.42,0.19) node[anchor=north west] {\tiny{$u$}};
\draw (-1.6,0.49) node[anchor=north west] {\tiny{$P_2$}};
\draw (-1.7,-0.85) node[anchor=north west] {\tiny{$P_2$}};
\draw (0.6,-2.83) node[anchor=north west] {\tiny{$P_4$}};
\draw (-1.8,-2.25) node[anchor=north west] {\tiny{$P_5$}};
\draw (1.3,-1.47) node[anchor=north west] {\tiny{$P_1$}};
\draw [line width=1pt] (-0.56,-2.71)-- (-0.9,-3.21);
\begin{scriptsize}
\draw [fill=black] (1.34,-0.27) circle (2.5pt);
\draw [fill=black] (0.86,-0.81) circle (2.5pt);
\draw [fill=black] (0.08,-0.99) circle (2.5pt);
\draw [fill=black] (-0.58,-1.17) circle (2.5pt);
\draw [fill=black] (0.54,-1.27) circle (2.5pt);
\draw [fill=black] (0.18,-1.75) circle (2.5pt);
\draw [fill=black] (-0.2,-2.23) circle (2.5pt);
\draw [fill=black] (-0.56,-2.71) circle (2.5pt);
\draw [fill=black] (1.22,-1.47) circle (2.5pt);
\draw [fill=black] (0.7,-1.69) circle (2.5pt);
\draw [fill=black] (0.58,-2.29) circle (2.5pt);
\draw [fill=black] (0.48,-2.79) circle (2.5pt);
\draw [fill=black] (0.38,-3.37) circle (2.5pt);
\draw [fill=black] (0.22,-0.39) circle (2.5pt);
\draw [fill=black] (-0.38,-0.01) circle (2.5pt);
\draw [fill=black] (-0.9,-3.21) circle (2.5pt);
\end{scriptsize}
\end{tikzpicture}
\caption{Vertex $u$ with $P(u)=P_{1}\oplus 2P_{2}\oplus P_{4}\oplus P_{5}$}\label{fig:ex1}
\end{figure}
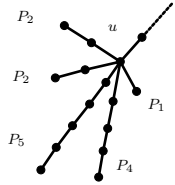
Pendant paths of length 2 are key to our strategy. Essentially, we transform
any tree $T$ to a tree $T^\prime$ having  only $P_2$'s as pendant paths.
A subgraph obtained by a vertex $u$ attached to $r\geq 1$ paths of length 2,
is called a {\em sun with $r$ rays} and denoted by $S_r$.
Hence, if a vertex $u$
of degree at least $r$
has $r\geq 1$ pendant
paths of length 2, say $P(u) = r P_2$, we will write
$P(u)=S_r$.

Consider a path between $u$ and $v$, where
$v$ has $r$ pendant $P_2$ paths,
that is $P(v) = S_r$.
To simplify the representation,
we use the concatenation symbol and write
$P_q \ast S_r$.

Given a tree $T$, we recall that a \emph{branch at a vertex $u$} is a maximal
subtree where $u$ is a leaf.
Let $v$ be a vertex where $P(v) = S_r$.
If a vertex $u$ of degree $\geq 3$ is connected
to $v$ by a path of length
$q \geq 0$, we observe that
$P(u)=P_q\ast S_r$ is a branch at $u$, which we call a \emph{generalized
pendant path at $u$},
abbreviated by {\em \gpp}.

To further simplify the graphic part of the representation, we will use a black
square {\tiny{$\blacksquare$}} to represent a pendant sun $S_r$ attached to a vertex,
and a single edge to represent the entire path $P_q$, omitting the $r$
pendant $P_2$'s and the $q$ vertices.
We will refer to this as the
$(P_q,S_r)$ {\em representation} of this generalized pendant path $P_q \ast S_r$,
as shown in Figure~\ref{fig:ex2}.
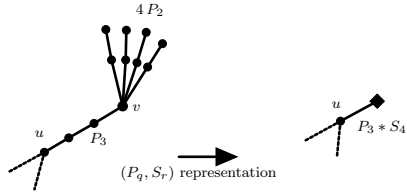
\begin{figure}[!ht]%\clip(-12.2,-1.06) rectangle (10.2,2.94);
  \centering
\begin{tikzpicture}[line cap=round,line join=round,>=triangle 45,x=1cm,y=1cm,scale=0.6, every node/.style={scale=0.7}]
\clip(-8.86,-0.64) rectangle (2.58,4);
\draw [line width=1pt,dash pattern=on 1pt off 1pt] (-6.12,0.44)-- (-6.92,0);
\draw [line width=1pt] (-4.38526,1.46298)-- (-3.8329483896370125,2.3307297959890714);
\draw [line width=1pt] (-3.8329483896370125,2.3307297959890714)-- (-3.5034939745135114,2.8484438768974316);
\draw [line width=1pt] (-4.38526,1.46298)-- (-4.068272971868086,2.4248596288815003);
\draw [line width=1pt] (-4.068272971868086,2.4248596288815003)-- (-3.8643250006011565,3.0994567646105757);
\draw [line width=1pt] (-4.38526,1.46298)-- (-4.31928585958123,2.4876128508097866);
\draw [line width=1pt] (-4.31928585958123,2.4876128508097866)-- (-4.3035975540991585,3.1465216810567904);
\draw [line width=1pt] (-4.38526,1.46298)-- (-4.63305196922266,2.4876128508097866);
\draw [line width=1pt] (-4.63305196922266,2.4876128508097866)-- (-4.758558413079232,3.162209986538862);
\draw (-5.28,1.04) node[anchor=north west] {{\tiny $P_3$}};
\draw [line width=1pt,dash pattern=on 1pt off 1pt] (0.44,1.13)-- (-0.36,0.69);
\draw [line width=1pt] (0.44,1.13)-- (1.25474,1.57298);
\draw (0.64,1.30) node[anchor=north west] {{\tiny $P_3*S_4$}};
\draw [->,line width=1pt] (-3.14,0.34) -- (-1.82,0.34);
\draw (-4.26,3.9) node[anchor=north west] {{\tiny $4\, P_2$}};
\draw (-4.34,1.66) node[anchor=north west] {{\tiny $v$}};
\draw (-4.6,0.3) node[anchor=north west] {{\tiny $(P_q,S_r) \text{ representation}$}};
\draw (0.08,1.74) node[anchor=north west] {{\tiny $u$}};
\draw [line width=1pt] (-6.12,0.44)-- (-5.58,0.76);
\draw [line width=1pt] (-5.58,0.76)-- (-5.02,1.08);
\draw [line width=1pt] (-5.02,1.08)-- (-4.38526,1.46298);
\draw (-6.5,1.08) node[anchor=north west] {{\tiny $u$}};
\draw [line width=1pt,dash pattern=on 1pt off 1pt] (-6.12,0.44)-- (-6.34,-0.38);
\draw [line width=1pt,dash pattern=on 1pt off 1pt] (0.44,1.13)-- (0.36,0.32);
\begin{scriptsize}
\draw [fill=black] (-6.12,0.44) circle (2.5pt);
\draw [fill=black] (-4.38526,1.46298) circle (3pt);
\draw [fill=black] (-3.8329483896370125,2.3307297959890714) circle (2.5pt);
\draw [fill=black] (-3.5034939745135114,2.8484438768974316) circle (2.5pt);
\draw [fill=black] (-4.068272971868086,2.4248596288815003) circle (2.5pt);
\draw [fill=black] (-3.8643250006011565,3.0994567646105757) circle (2.5pt);
\draw [fill=black] (-4.31928585958123,2.4876128508097866) circle (2.5pt);
\draw [fill=black] (-4.3035975540991585,3.1465216810567904) circle (2.5pt);
\draw [fill=black] (-4.63305196922266,2.4876128508097866) circle (2.5pt);
\draw [fill=black] (-4.758558413079232,3.162209986538862) circle (2.5pt);
\draw [fill=black] (0.44,1.13) circle (2.5pt);
\draw [fill=black] (1.25474,1.57298) ++(-4.5pt,0 pt) -- ++(4.5pt,4.5pt)--++(4.5pt,-4.5pt)--++(-4.5pt,-4.5pt)--++(-4.5pt,4.5pt);
\draw [fill=black] (-5.58,0.76) circle (2.5pt);
\draw [fill=black] (-5.02,1.08) circle (2.5pt);
\end{scriptsize}
\end{tikzpicture} 
\caption{$(P_q, S_r)$ representation of a generalized pendant path}\label{fig:ex2}
\end{figure}

We can consider $q=0$ for paths
$P_q$ of length 0, as well as $r=0$ for no pendant $S_r$.
However, we do {\em not}
allow both $r=q=0$ simultaneously. As we will see in Algorithm
\texttt{InitiateRepresentation}, it is always possible to use the $(P_q,
S_r)$ representation to represent every pendant path in a given tree. We call
this a $(P_q, S_r)$ {\em representation of $T$}.
We adopt the following convention for representing paths.
\begin{eqnarray*}
   P_1 & = & P_1 \ast S_0 \\
   P_2 & = & P_0 \ast S_1 \\
   P_q & = & P_{q-2} \ast S_1  \text{ for } q \geq 3.
\end{eqnarray*}

This convention provides a particular $(P_q, S_r)$ representation
of paths in a given tree
in which $r$ is always equal to 0 or 1.
We observe that there are many other feasible choices,
including the one where, if a vertex $u$ has $r \geq 2$ pendant $P_2$
attached and no other pendant paths, we write $P(u)=r\, P_2=P_0 \ast S_r$.

Consider the fairly large
tree $\bar{T}$ in Figure~\ref{fig:ex3:1}
that will be used throughout the paper.
Using the special symbol to represent pendant $S_r$, we obtain, in
Figure~\ref{fig:an_example:0001} (left), its $(P_q, S_r)$ representation.
{\em All trees are considered to be in $(P_q, S_r)$ representation.}

\begin{figure}[!ht]
  \centering
  \begin{tikzpicture}[line cap=round,line join=round,>=triangle 45,x=1cm,y=1cm, scale=0.6]
\clip(-10.65,-3.2) rectangle (10.65,3.2);
\draw [line width=1pt] (0.57,-0.58)-- (1.1342730908889087,-1.1324377442818057);
\draw [line width=1pt] (-2,-1)-- (-2.8348681960751843,-1.838411490975024);
\draw [line width=1pt] (-2.8348681960751843,-1.838411490975024)-- (-3.62,-2.6);
\draw [line width=1pt] (-3.62,-2.6)-- (-4.2,-2.36);
\draw [line width=1pt] (-4.2,-2.36)-- (-4.64,-2.2);
\draw [line width=1pt] (0,0)-- (0,1);
\draw [line width=1pt] (0,1)-- (0.31,1.76);
\draw [line width=1pt] (0,1)-- (0.68,1);
\draw [line width=1pt] (0.68,1)-- (1.32,0.98);
\draw [line width=1pt] (1.32,0.98)-- (1.84,0.98);
\draw [line width=1pt] (1.84,0.98)-- (2.185389558187702,0.5462109422998472);
\draw [line width=1pt] (2.185389558187702,0.5462109422998472)-- (2.483467362347061,0.15400330524805916);
\draw [line width=1pt] (1.84,0.98)-- (2.6246621116857045,0.750158913566777);
\draw [line width=1pt] (2.6246621116857045,0.750158913566777)-- (3.063934665183707,0.6246524697102048);
\draw [line width=1pt] (1.84,0.98)-- (2.35474,1.26298);
\draw [line width=1pt] (1.1342730908889087,-1.1324377442818057)-- (1.5735456443869111,-1.5403336868156652);
\draw [line width=1pt] (-2,-1)-- (-2.447367050668015,-0.8139651429957524);
\draw [line width=1pt] (-5.2226282904464725,0.36579542925602465)-- (-5.661900843944475,0.5383667895588113);
\draw [line width=1pt] (-5.661900843944475,0.5383667895588113)-- (-6.0227318700321195,0.6952498443795265);
\draw [line width=1pt] (2.35474,1.26298)-- (2.90705161036299,2.1307297959890708);
\draw [line width=1pt] (2.90705161036299,2.1307297959890708)-- (3.236506025486492,2.648443876897431);
\draw [line width=1pt] (2.35474,1.26298)-- (2.6717270281319174,2.2248596288815);
\draw [line width=1pt] (2.6717270281319174,2.2248596288815)-- (2.8756749993988473,2.8994567646105756);
\draw [line width=1pt] (2.35474,1.26298)-- (2.4207141404187733,2.287612850809786);
\draw [line width=1pt] (2.4207141404187733,2.287612850809786)-- (2.4364024459008444,2.94652168105679);
\draw [line width=1pt] (2.35474,1.26298)-- (2.106948030777343,2.287612850809786);
\draw [line width=1pt] (2.106948030777343,2.287612850809786)-- (1.9814415869207707,2.9622099865388614);
\draw [line width=1pt] (-3.62,-2.6)-- (-3.96,-2.12);
\draw [line width=1pt] (-3.96,-2.12)-- (-4.32,-1.7);
\draw [line width=1pt] (-3.62,-2.6)-- (-3.697724997589117,-2.0109828512778107);
\draw [line width=1pt] (-3.697724997589117,-2.0109828512778107)-- (-3.8075431359636176,-1.446203853923236);
\draw [line width=1pt] (2.483467362347061,0.15400330524805916)-- (2.7501685555422766,-0.1911394153575143);
\draw [line width=1pt] (2.185389558187702,0.5462109422998472)-- (2.05988311433113,0.05987347235563002);
\draw [line width=1pt] (2.05988311433113,0.05987347235563002)-- (1.9657532814387009,-0.3166458592140865);
\draw [line width=1pt] (0.57,-0.58)-- (0,0);
\draw [line width=1pt] (-3.344738124242506,-0.43430815032962167)-- (-3.7761665249994727,-0.261736790026835);
\draw [line width=1pt] (-3.7761665249994727,-0.261736790026835)-- (-4.242109197816997,-0.054651157663490914);
\draw [line width=1pt] (-4.242109197816997,-0.054651157663490914)-- (-4.690794734604243,0.13517733866957446);
\draw [line width=1pt] (-4.690794734604243,0.13517733866957446)-- (-5.2226282904464725,0.36579542925602465);
\draw [line width=1pt] (-3.344738124242506,-0.43430815032962167)-- (-2.913309723485539,-0.6241366466626871);
\draw [line width=1pt] (-2.913309723485539,-0.6241366466626871)-- (-2.447367050668015,-0.8139651429957524);
\draw [line width=1pt] (-2,-1)-- (-0.96,-0.48);
\draw [line width=1pt] (-0.96,-0.48)-- (0,0);
\draw [line width=1pt] (-0.96,-0.48)-- (-0.9,-1.4);
\draw [line width=1pt] (-0.9,-1.4)-- (-0.88,-2);
\draw [line width=1pt] (-0.88,-2)-- (-0.44,-2.56);
\draw [line width=1pt] (-0.9,-1.4)-- (-1.44,-1.8);
\draw [line width=1pt] (-0.88,-2)-- (-0.98,-2.66);
\draw [line width=1pt] (-0.96,-0.48)-- (-1.34,0.32);
\draw [line width=1pt] (-1.34,0.32)-- (-1.34,1.06);
\draw [line width=1pt] (-1.34,0.32)-- (-1.96,0.44);
\draw [line width=1pt] (-1.96,0.44)-- (-2.56,0.54);
\begin{scriptsize}
\draw [fill=black] (0,0) circle (2.5pt);
\draw [fill=black] (0.57,-0.58) circle (2.5pt);
\draw [fill=black] (1.1342730908889087,-1.1324377442818057) circle (2.5pt);
\draw [fill=black] (-2,-1) circle (2.5pt);
\draw [fill=black] (-2.8348681960751843,-1.838411490975024) circle (2.5pt);
\draw [fill=black] (-3.62,-2.6) circle (2.5pt);
\draw [fill=black] (-4.2,-2.36) circle (2.5pt);
\draw [fill=black] (-4.64,-2.2) circle (2.5pt);
\draw [fill=black] (0,1) circle (2.5pt);
\draw [fill=black] (0.31,1.76) circle (2.5pt);
\draw [fill=black] (0.68,1) circle (2.5pt);
\draw [fill=black] (1.32,0.98) circle (2.5pt);
\draw [fill=black] (1.84,0.98) circle (2.5pt);
\draw [fill=black] (2.185389558187702,0.5462109422998472) circle (2.5pt);
\draw [fill=black] (2.483467362347061,0.15400330524805916) circle (2.5pt);
\draw [fill=black] (2.6246621116857045,0.750158913566777) circle (2.5pt);
\draw [fill=black] (3.063934665183707,0.6246524697102048) circle (2.5pt);
\draw [fill=black] (2.35474,1.26298) circle (2.5pt);
\draw [fill=black] (1.5735456443869111,-1.5403336868156652) circle (2.5pt);
\draw [fill=black] (-2.447367050668015,-0.8139651429957524) circle (2.5pt);
\draw [fill=black] (-5.2226282904464725,0.36579542925602465) circle (2.5pt);
\draw [fill=black] (-5.661900843944475,0.5383667895588113) circle (2.5pt);
\draw [fill=black] (-6.0227318700321195,0.6952498443795265) circle (2.5pt);
\draw [fill=black] (2.90705161036299,2.1307297959890708) circle (2.5pt);
\draw [fill=black] (3.236506025486492,2.648443876897431) circle (2.5pt);
\draw [fill=black] (2.6717270281319174,2.2248596288815) circle (2.5pt);
\draw [fill=black] (2.8756749993988473,2.8994567646105756) circle (2.5pt);
\draw [fill=black] (2.4207141404187733,2.287612850809786) circle (2.5pt);
\draw [fill=black] (2.4364024459008444,2.94652168105679) circle (2.5pt);
\draw [fill=black] (2.106948030777343,2.287612850809786) circle (2.5pt);
\draw [fill=black] (1.9814415869207707,2.9622099865388614) circle (2.5pt);
\draw [fill=black] (-3.96,-2.12) circle (2.5pt);
\draw [fill=black] (-4.32,-1.7) circle (2.5pt);
\draw [fill=black] (-3.697724997589117,-2.0109828512778107) circle (2.5pt);
\draw [fill=black] (-3.8075431359636176,-1.446203853923236) circle (2.5pt);
\draw [fill=black] (2.7501685555422766,-0.1911394153575143) circle (2.5pt);
\draw [fill=black] (2.05988311433113,0.05987347235563002) circle (2.5pt);
\draw [fill=black] (1.9657532814387009,-0.3166458592140865) circle (2.5pt);
\draw [fill=black] (-3.344738124242506,-0.43430815032962167) circle (2.5pt);
\draw [fill=black] (-3.7761665249994727,-0.261736790026835) circle (2.5pt);
\draw [fill=black] (-4.242109197816997,-0.054651157663490914) circle (2.5pt);
\draw [fill=black] (-4.690794734604243,0.13517733866957446) circle (2.5pt);
\draw [fill=black] (-2.913309723485539,-0.6241366466626871) circle (2.5pt);
\draw [fill=black] (-0.96,-0.48) circle (2.5pt);
\draw [fill=black] (-0.9,-1.4) circle (2.5pt);
\draw [fill=black] (-0.88,-2) circle (2.5pt);
\draw [fill=black] (-0.44,-2.56) circle (2.5pt);
\draw [fill=black] (-1.44,-1.8) circle (2.5pt);
\draw [fill=black] (-0.98,-2.66) circle (2.5pt);
\draw [fill=black] (-1.34,0.32) circle (2.5pt);
\draw [fill=black] (-1.34,1.06) circle (2.5pt);
\draw [fill=black] (-1.96,0.44) circle (2.5pt);
\draw [fill=black] (-2.56,0.54) circle (2.5pt);
\end{scriptsize}
\end{tikzpicture} 
\caption{Tree before the $(P_q, S_r)$ representation}\label{fig:ex3:1}
\end{figure}
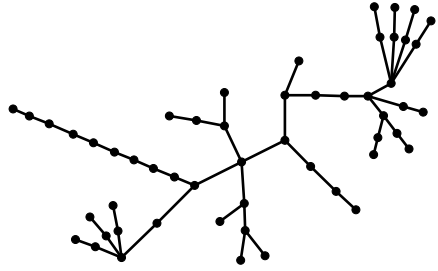

We say a vertex $u$ is a \emph{starlike vertex} if it has degree $\geq 3$ and
has at least two generalized pendant paths attached to it.
This definition depends on the particular $(P_q, S_r)$
representation, as well as the graph itself. For example,
even though they represent the same subgraph,
a vertex $u$ with $P(u)=3 (P_0 * S_1)$ has three
generalized pendant paths (hence a starlike vertex),
whereas  $P(u)= P_0 *S_3$ has a single generalized pendant path and therefore
is not starlike. Let $T$ be a tree and let $u$ be a starlike
vertex of $T$ having $\ell \geq 2$ generalized pendant paths attached at $u$.
According to our $(P_q,S_r)$ representation, we have  $P(u)=P_{q_1} \ast
S_{r_1}\oplus\cdots \oplus P_{q_\ell} \ast S_{r_\ell}$. Let us call the
\emph{weight} of $u$ and denote by $w(u)$, the sum of the number of vertices
of the generalized pendant paths, that is
$$w(u) =\sum_{i=1}^{l} \sharp (P_{q_i}\ast S_{r_i})
=\sum_{i=1}^{l} (q_i+2r_i).$$
For example
in ascending order of weight,
the vertices $u_1, u_2, u_3, u_4$
and $u_5$  in the tree of Figure \ref{fig:an_example:0001}
are all its starlike vertices
and have weights respectively, 2,3,4, 6 and 8.

\begin{center}
\begin{figure}[t]
  \centering
\input{initiate}
\caption{Procedure \texttt{InitiateRepresentation}.}\label{init-rep}
\end{figure}
\end{center}
\section{Strategy and algorithmic tools}\label{sec:stra}
A \emph{proper transformation} is defined  as an operation on a tree $T$
which gives a new tree $T^\prime$, with the same number of vertices, that
does not decrease the number of eigenvalues above the average degree, that is
$$\sigma(T) \leq \sigma(T^\prime).$$
A proper transformation requires $n \geq \minorder$, however as explained
earlier, trees of smaller size have been checked by computation.

Our strategy to prove that, for any tree $T$ of order $n$, $\sigma(T) \leq
\lfloor\frac{n}{2}\rfloor$ is to make successive transformations on $T$ to
obtain a prototype tree $T^\prime$. From the fact that we use proper
transformations, and that $\sigma(T^\prime) \leq \lfloor\frac{n}{2}\rfloor$,
it follows that $\sigma(T) \leq \sigma(T^\prime)$, proving
Theorem~\ref{thr:main}.
More precisely, for any tree with a given order $n$, we properly reduce it to
a prototype tree that depends only on the congruence $n \equiv \alpha
\pmod{4}$. We refer to Section \ref{sec:prot} for the definition of the prototype
trees and for the proof they satisfy Theorem~\ref{thr:main}.
\begin{center}
\begin{figure}[b]
  \centering
\input{transform}
\caption{Transforming $T$.}\label{treealgo}
\end{figure}
\end{center}
We describe a high level \texttt{algorithm Transform} to do the
transformation, shown in Figure~\ref{treealgo}. The initialization procedure
\texttt{InitiateRepresentation$(T)$} puts the tree $T$ into a $(P_q, S_r)$
representation, as illustrated in Figure~\ref{fig:an_example:0001}.
It may be formally described by the pseudo code of Figure~\ref{init-rep}.

The next step is the identification and ordering of all $k$ starlike vertices
of $T$. Recall that the weight of a starlike vertex $u$ is the total number
of vertices hanging at $u$. The main parameters of our transformation
algorithm is the number $k$ of starlike vertices and their weights.

The heart of our algorithm is the procedure \texttt{ReduceStarVertex}
$(T,u_1)$. It takes the tree $T$ and its starlike vertex of minimum weight
$u_1$ as arguments, and properly transforms the generalized pendant paths at
$u_1$ into a single generalized pendant path. More precisely,
$P(u)=P_{q_1} \ast S_{r_1}\oplus\cdots \oplus P_{q_{\ell}} \ast S_{r_\ell}$
is replaced by
$P_q\ast S_r$, for certain values of $q$ and $r$. We will prove in Section \ref{sec:red}, Lemma \ref{lem:nu-star-vertices} that if  the number of starlike vertices is greater than one,  then there is always a starlike vertex $u$ with $w(u) \leq 2\lfloor \frac{n}{4} \rfloor$.

Moreover, we will see that this does not increase the number of starlike
vertices, but increases the minimum weight.
This guarantees the algorithm stops.

While the $\ps$ representation of a tree changes during the transformation,
the following invariants are preserved by \texttt{ReduceStarVertex}. Leaves
and only leaves are square representing $S_r$, and pendant paths appear
to have length one, representing $P_q$, $q \geq 0$, and where $r + q \geq 1$.
That is, {\gpps} are on the tree's extremities.

For trees with a small number of starlike vertices, we have a different
strategy. In fact if a tree $T$ of order $n \geq 8$ has $k=0$ or $1$ starlike
vertices, we prove in Section~\ref{sec:small}, that $T$ can be properly transformed into $T_\alpha$, where $n \equiv \alpha \pmod 4$.
These results, along with the main procedure
\texttt{ReduceStarVertex} $(T,u_1)$, are going to be described later (in
Sections~\ref{sec:red}~and~\ref{sec:small}), as they need some technical analysis.

We finally review our main tool to prove the conjecture. It is the algorithm
reproduced in Figure \ref{alg-lap}. For any tree $T$ with $n$ vertices, it
produces a diagonal matrix $D$ that is congruent to the matrix $L + x I_n$,
where $L$ is the Laplacian matrix of $T$.

This algorithm, presented first in \cite{FHRT11}, is the Laplacian matrix
version of the adjacency matrix algorithm \cite{JT2011} that has been useful
in many applications of spectral graph theory (see, for example, the recent
ordering by the spectral radius \cite{OST2018,BELARDO2019} of certain trees).
It is shown (see \cite{JT2011}) that

\begin{lemma}
\label{lem:negzeropos} The number of eigenvalues of $T$ less (greater) than
$x$ is exactly the number of negative (positive) diagonal elements produced
by Diagonalize$( T,-x$ ).
\end{lemma}

\begin{center}
\begin{figure}
  \centering
\input{diagonalize}
\caption{\label{alg-lap} $L + x I_n$ diagonalization}
\end{figure}
\end{center}

\begin{example}
We illustrate here how the algorithm may be executed on the tree itself.
Considering the star with $n >2$ vertices $K_{1,n-1}$ with $n-1$ leaves and
the center vertex of degree $n-1$.
The tree can be rooted anywhere, but
we choose the center vertex as the root.
A vertex $v$ is initialized with $\deg(v) + x$.
When $x= -2 + \frac{2}{n}$,
the initial values at the leaves are $-1 + \frac{2}{n} <0$  and
the initial value at the root is $n-3 + \frac{2}{n}$. The values at the leaves are
kept, while the value at the root changes to
$$
{\displaystyle n-3 + \frac{2}{n}-\frac{n-1}{-1+\frac{2}{n}}>0}.
$$
By Lemma \ref{lem:negzeropos}, it
implies there are $n-1$ Laplacian eigenvalues smaller than the average degree
and a single eigenvalue above it.
\end{example}

The technique we use to prove that
for any tree $T$ with $n$ vertices,
the number of Laplacian eigenvalues
greater than $2-\frac{2}{n}$ is at most
$\lfloor\frac{n}{2} \rfloor$, is the analysis of the signs on $T$ after
applying \texttt{Algorithm Diagonalize}$(T,-2+\frac{2}{n})$.

\section{Proper transformations}\label{sec:prop}

In this section we present a few local transformations which are performed on
a tree $T$ that preserve the number of vertices and do not decrease the
number of eigenvalues above the average degree.
Such proper transformations are local and for this reason we can translate this
property in terms of elementary rational recursions.

We start by analyzing the signs of the vertices after applying
\texttt{Diagonalize$\left(T,-2+\frac{2}{n}\right)$} on a tree having $r$ pendant $P_2$'s
attached to a path as in
Figure \ref{newseq_r}.
We assume $n \geq 8$ is the
number of vertices in $T$ and $0 \leq r \leq \lfloor\frac{n}{2}\rfloor$ is
the number of pendant paths $P_{2}$. In fact, using our notation established
above, this is a generalized pendant path $P_q \ast S_r$.

We use a white dots to indicate vertices where
\texttt{Diagonalize$\left(T,-2+\frac{2}{n}\right)$}
produces a negative value, black dots
indicate a positive value and light
gray where we do not know the precise sign.

\begin{figure}[!ht]
  \centering
\definecolor{ffffff}{rgb}{1,1,1}
\begin{tikzpicture}[line cap=round,line join=round,>=triangle 45,x=1cm,y=1cm, scale=0.6, every node/.style={scale=0.8}]
\clip(-12.2,-2.0) rectangle (12.2,3.4);
\draw [line width=1pt] (-2.02,0.97)-- (-3,0);
\draw [line width=1pt] (-3,0)-- (-4,-1);
\draw [line width=1pt] (-2.02,0.97)-- (-3.3387230793767304,0.5850247785033855);
\draw [line width=1pt] (-3.3387230793767304,0.5850247785033855)-- (-4.709949526960204,0.24783794713040147);
\draw [line width=1pt] (-2.02,0.97)-- (-3.3612022014682625,1.3268358075239504);
\draw [line width=1pt] (-3.3612022014682625,1.3268358075239504)-- (-4.799866015326334,1.6415435168054022);
\draw [line width=1pt] (-2.02,0.97)-- (-3,2);
\draw [line width=1pt] (-3,2)-- (-4,3);
\draw (-5.92,1.23) node[anchor=north west] {{\tiny}$S_{r}$};
\draw [line width=1pt] (-0.011813009829940168,0.9896489761509664)-- (-2.02,0.97);
\draw (-4.86,0.27) node[anchor=north west] {{\tiny $x_1$}};
\draw (-2.8,0.21) node[anchor=north west] {{\tiny $x_2$}};
\draw (-5.1,1.55) node[anchor=north west] {{\tiny $x_1$}};
\draw (-4.56,3.05) node[anchor=north west] {{\tiny $x_1$}};
\draw (-3.86,-0.85) node[anchor=north west] {{\tiny $x_1$}};
\draw (-3.84,0.55) node[anchor=north west] {{\tiny $x_2$}};
\draw (-4.04,1.39) node[anchor=north west] {{\tiny $x_2$}};
\draw (-3.72,2.27) node[anchor=north west] {{\tiny $x_2$}};
\draw (-2,1.85) node[anchor=north west] {{\tiny $b_1(r)$}};
\draw (-0.48,0.87) node[anchor=north west] {{\tiny $b_2(r)$}};
\draw [line width=1pt] (-0.011813009829940168,0.9896489761509664)-- (1.966349734224908,0.9896489761509664);
\draw (1.02,1.9) node[anchor=north west] {{\tiny}$P_q$};
\begin{scriptsize}
\draw [fill=ffffff] (-2.02,0.97) circle (2.5pt);
\draw [fill=black] (-3,0) circle (2.5pt);
\draw [fill=ffffff] (-4,-1) circle (2.5pt);
\draw [fill=black] (-3.3387230793767304,0.5850247785033855) circle (2.5pt);
\draw [fill=ffffff] (-4.709949526960204,0.24783794713040147) circle (2.5pt);
\draw [fill=black] (-3.3612022014682625,1.3268358075239504) circle (2.5pt);
\draw [fill=ffffff] (-4.799866015326334,1.6415435168054022) circle (2.5pt);
\draw [fill=black] (-3,2) circle (2.5pt);
\draw [fill=ffffff] (-4,3) circle (2.5pt);
\draw [fill=black] (-0.011813009829940168,0.9896489761509664) circle (2.5pt);
\end{scriptsize}
\end{tikzpicture}
  \caption{A tree end with a generalized pendant path $P_q\ast S_r$.} \label{newseq_r}
\end{figure}
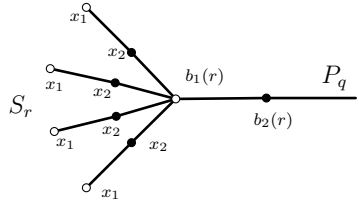

Applying the algorithm to the Laplacian matrix  to locate
$x = d_n$  we
obtain, in each extremal vertex of the pendant path $P_{2}$ the value
$$x_{1}=1-d_n=-1+\frac{2}{n}<0$$
 and the next value is
$$x_{2}=2-d_n -\frac{1}{x_{1}}=\frac{2}{n}-\frac{1}{x_{1}}>1.$$
For completeness, we may consider the recurrence relation
\begin{equation}\label{dependparamx}
   \begin{cases}
     x_{1}=-1+\frac{2}{n} \\
     x_{j+1}= \frac{2}{n} -\frac{1}{x_{j}}
   \end{cases}
\end{equation}
From these values we proceed processing the vertices of the path
$P_q$ in Figure \ref{newseq_r}.  obtaining
$$b_{1}=r+1-d_n-\frac{r}{x_{2}}= x_{1} + r\left(1- \frac{1}{x_{2}}\right),$$
and the rest of the values on the path are given by the recursion
\begin{equation}
\label{dependparamcase2}
   \begin{cases}
     b_{1}=x_{1} + r\left(1- \frac{1}{x_{2}}\right)  \\
     b_{j+1}= \frac{2}{n} -\frac{1}{b_{j}}
   \end{cases}
\end{equation}
for $n \in[8,\; \infty)$.
We observe that $b_j\not =0$. To see this, notice that $b_j$ is a rational function of $d_n=2 -\frac{2}{n}$, and if this was 0, then $d_n$ would be a root of a polynomial with integral coefficients, which is absurd since $d_n$ is non integral rational number. In what follows, we are going to analyze  the signs of the values $b_1$ and $b_2$. Denoting by $v_1$ and $v_2$, respectively, the corresponding vertices, we can always choose the root of the tree away from these vertices, provided there are other vertices. In case there are no other vertices in the tree, then this tree is already one of the prototype trees.

As the sequence $b$ depends on $r$ we denote
$b_{j}:=b_{j}(r)$ for $r \geq 0$. We observe that if $r=0$, then $b_j=x_j$ and if $r=1$, then $b_j=x_{j+2}$.
We can summarize the main properties of these two sequences in the following
lemma
\begin{lemma}
\label{properties x and b}
Consider the above defined sequences $x_{j}$ and $b_{j}$ then
\begin{itemize}
  \item[a)] $-1 < x_{1} < 0$ and $x_{2} > 1$;
  \item[b)] $x_{1} \leq b_{1}(r)$, for all $r$, with strict inequality when $r \geq 1$.
  \item[c)] If $r=0$ then $b_{1}(r)=x_{1}$ and for $r=1$ we have $b_{1}(r)=x_{3}$;
  \item[d)] The map $r \to b_{1}(r)$ is linear, strictly increasing from $\mathbb{N} \to [x_1, \infty)$
  \item[e)] $b_1 (r+1) - b_1 (r) = 1 - \frac{1}{x_{2}}$ for all $r$.
\end{itemize}
\end{lemma}
\begin{proof}
All five claims are straightforward computations. For example, for (c) we
know that $x_{1}=-1+\frac{2}{n}$, $x_{2}=\frac{2}{n}-\frac{1}{x_{1}}$,
$b_{1}(0)=x_{1} + 0 (1- \frac{1}{x_{2}})=x_{1} $ and
$b_{1}(1)=x_{1} + 1 (1- \frac{1}{x_{2}})=
\frac{2}{n}-\frac{1}{x_{2}}=x_{3}.$
\end{proof}

Our first concern is the dependence of the initial condition $b_1 (r)$ with respect to $r$.
\begin{lemma}
\label{depend   init  cond  to r}
Let $n \geq 8$. If $0 \leq r \leq \lfloor\frac{n}{4}\rfloor$ then
$b_{1}(r)<0$.
\end{lemma}
\begin{proof}
Assume $n \geq 8$.
Recall that $b_{1}(r)=x_{1} + r(1- \frac{1}{x_{2}})$.
On the non-negative reals, define the linear function into $\mathbb{R}$
$$
g(r) = x_{1} + r (1- \frac{1}{x_{2}}).
$$
Then $g(0)=x_{1}<0$.
Since $g'(r) = (1- \frac{1}{x_{2}} )>0$, it is increasing.
By continuity there is a unique point $g(r_0)=0$.
Solving for $r_0$ we have
$r_0=\frac{-x_{1}}{1- \frac{1}{x_{2}}} > 0$.
Since $x_1 = -\frac{n-2}{n}$
and
$x_2 = \frac{n^2 + 2n - 4}{n^2 - 2n}$,
$r_0$ depends rationally
on $n$
$$
     r_0=\,{\frac { \left( n-2 \right)  \left( {n}^{2}+2\,n-4 \right) }{4n
 \left( n-1 \right) }}.
$$
Therefore $b_{1}(r)<0$ if and only if $r\leq \lfloor r_0 \rfloor$.
We claim that
$\frac{n}{4} < r_0$ for
$n \geq 7$.
Indeed, the inequality $\frac{n}{4} < r_0$ can be simplified
to $0 < n^2 - 8n + 8$ whose largest root is $4 + 2\sqrt{2} \approx 6.8$.
\end{proof}

The proof of
Lemma~\ref{depend   init  cond  to r}
shows that the inequality holds when $n \geq 7$, however it will be more convenient
to assume $n \geq 8$.

\begin{corollary}\label{cor:b2}
Let $T$ be a tree with $n \geq 8$ vertices and $u$ a vertex having $0
\leq r \leq \lfloor  \frac{n}{4}\rfloor$ pendant $P_2$'s and a path $P_q$
with $q \geq 2$. Then $b_2(r) >0.$
\end{corollary}
\begin{proof}
We know from
Lemma~\ref{depend   init  cond  to r}
that $b_1(r) <0$ for $0
\leq r \leq \lfloor \frac{n}{4}\rfloor$ and by the recurrence relation
$b_2(r)=\frac{2}{n}-\frac{1}{b_1(r)}$ we conclude that $b_2(r)>0$.
\end{proof}

We observe that these results do not account for the sign associated with the
vertex where the gpp  $P_q*S_r$ is attached.
We now present the first proper transformation that we use.

\begin{proposition}\label{star-up}
(Star-up transform)
Let $u$ be a vertex that is not a leaf of a tree $T$ with $n \geq 8$
vertices. If $u$  has a path  $P_q, ~~ q\geq 2$ connecting $u$ to a
vertex that has exactly $0 \leq r \leq \lfloor  \frac{n}{4}\rfloor -1$
pendant $P_2$, and no other pendant path, then the transformation
$$P_q \ast S_r \to P_{q-2} \ast S_{r+1}$$ is proper.

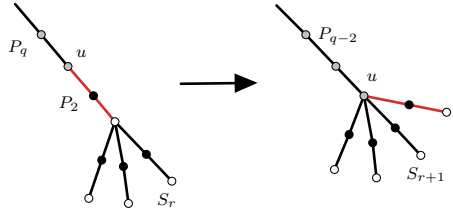
\begin{figure}[!ht]
  \centering
  \definecolor{cqcqcq}{rgb}{0.7529411764705882,0.7529411764705882,0.7529411764705882}
\definecolor{dtsfsf}{rgb}{0.8274509803921568,0.1843137254901961,0.1843137254901961}
\definecolor{ffffff}{rgb}{1,1,1}
\begin{tikzpicture}[line cap=round,line join=round,>=triangle 45,x=1cm,y=1cm, scale=0.6, every node/.style={scale=0.8}]
\clip(-8.3,-1.22) rectangle (9.64,4.2);
\draw [line width=1pt,color=dtsfsf] (-1.636137486647466,2.074186303595756)-- (-1.16,1.5);
\draw [line width=1pt] (-1.16,1.5)-- (-1.452599533144194,0.6482375879164936);
\draw [line width=1pt] (-1.452599533144194,0.6482375879164936)-- (-1.74,-0.2);
\draw [line width=1pt] (-1.16,1.5)-- (-0.9725771932125596,0.507054546760131);
\draw [line width=1pt] (-0.9725771932125596,0.507054546760131)-- (-0.8607602246167193,-0.3130777393171791);
\draw [line width=1pt] (-1.16,1.5)-- (-0.4643182450496526,0.77530232495722);
\draw [line width=1pt] (-0.4643182450496526,0.77530232495722)-- (0.1,0.18);
\draw (-0.42,0.1) node[anchor=north west] {{\tiny $S_r$}};
\draw (-2.22,3.3) node[anchor=north west] {{\tiny $u$}};
\draw [line width=1pt,color=dtsfsf] (-1.636137486647466,2.074186303595756)-- (-2.2008696512729182,2.7236282929150244);
\draw [line width=1pt] (-2.2008696512729182,2.7236282929150244)-- (-2.7938384241296434,3.4295434986968374);
\draw (-2.6,2.25) node[anchor=north west] {{\tiny $P_2$}};
\draw (3.14,3.78) node[anchor=north west] {{\tiny $P_{q-2}$}};
\draw [line width=1pt] (4.3562472196987985,2.0720898207028045)-- (4.027420209340036,1.2150874981592914);
\draw [line width=1pt] (4.027420209340036,1.2150874981592914)-- (3.702699214680401,0.4385807717992969);
\draw [line width=1pt] (4.3562472196987985,2.0720898207028045)-- (4.54979746161858,1.03154954465602);
\draw [line width=1pt] (4.54979746161858,1.03154954465602)-- (4.634507286312397,0.2550428182960255);
\draw [line width=1pt] (4.3562472196987985,2.0720898207028045)-- (5.050008976435578,1.3610707627149674);
\draw [line width=1pt] (5.050008976435578,1.3610707627149674)-- (5.616247219698799,0.7520898207028046);
\draw [line width=1pt,color=dtsfsf] (4.3562472196987985,2.0720898207028045)-- (5.3575056400741365,1.8735652021125633);
\draw [line width=1pt,color=dtsfsf] (5.3575056400741365,1.8735652021125633)-- (6.187520739032391,1.7092281422065607);
\draw (5.1,0.68) node[anchor=north west] {{\tiny $S_{r+1}$}};
\draw (4.22,2.76) node[anchor=north west] {{\tiny $u$}};
\draw [line width=1pt] (4.3562472196987985,2.0720898207028045)-- (3.7346065819817427,2.7277226011085567);
\draw [->,line width=1pt] (0.3,2.34) -- (2,2.36);
%\draw (-0.32,3.24) node[anchor=north west] {{\tiny $Star-up$}};
\draw [line width=1pt] (3.7346065819817427,2.7277226011085567)-- (3.068363810764865,3.4281316682852707);
\draw [line width=1pt] (3.068363810764865,3.4281316682852707)-- (2.4192041875279076,4.077291291522226);
\draw [line width=1pt] (-2.7938384241296434,3.4295434986968374)-- (-3.372688892870732,4.107222096247378);
\draw (-3.72,3.5) node[anchor=north west] {{\tiny $P_q$}};
\begin{scriptsize}
\draw [fill=black] (-1.636137486647466,2.074186303595756) circle (2.5pt);
\draw [fill=ffffff] (-1.16,1.5) circle (2.5pt);
\draw [fill=black] (-1.452599533144194,0.6482375879164936) circle (2.5pt);
\draw [fill=ffffff] (-1.74,-0.2) circle (2.5pt);
\draw [fill=black] (-0.9725771932125596,0.507054546760131) circle (2.5pt);
\draw [fill=ffffff] (-0.8607602246167193,-0.3130777393171791) circle (2.5pt);
\draw [fill=black] (-0.4643182450496526,0.77530232495722) circle (2.5pt);
\draw [fill=ffffff] (0.1,0.18) circle (2.5pt);
\draw [fill=cqcqcq] (-2.2008696512729182,2.7236282929150244) circle (2.5pt);
\draw [fill=cqcqcq] (-2.7938384241296434,3.4295434986968374) circle (2.5pt);
\draw [fill=cqcqcq] (4.3562472196987985,2.0720898207028045) circle (2.5pt);
\draw [fill=black] (4.027420209340036,1.2150874981592914) circle (2.5pt);
\draw [fill=ffffff] (3.702699214680401,0.4385807717992969) circle (2.5pt);
\draw [fill=black] (4.54979746161858,1.03154954465602) circle (2.5pt);
\draw [fill=ffffff] (4.634507286312397,0.2550428182960255) circle (2.5pt);
\draw [fill=black] (5.050008976435578,1.3610707627149674) circle (2.5pt);
\draw [fill=ffffff] (5.616247219698799,0.7520898207028046) circle (2.5pt);
\draw [fill=black] (5.3575056400741365,1.8735652021125633) circle (2.5pt);
\draw [fill=ffffff] (6.187520739032391,1.7092281422065607) circle (2.5pt);
\draw [fill=cqcqcq] (3.7346065819817427,2.7277226011085567) circle (2.5pt);
\draw [fill=cqcqcq] (3.068363810764865,3.4281316682852707) circle (2.5pt);
\end{scriptsize}
\end{tikzpicture}
\caption{Star-up transform}\label{Star-up:fig}
\end{figure}
\end{proposition}
\begin{proof}
We consider the transformation on a tree $T$, as illustrated in
Figure~\ref{Star-up:fig}.
This takes one pendant path $P_{2}$ at the vertex $u$ connected to the sun $S_{r}$ formed by $r\,P_{2}$ paths,
and produces a new tree $T^\prime$ with a sun $S_{r+1}$ attached at $u$. We consider $u$ as
the root of $T$, meaning that it is going to be the last vertex to be
processed.

The signs at the vertices of the branch of $T^\prime$ not containing
$S_{r+1}$ are the same as  the signs in $T$.
By Lemma~\ref{depend   init  cond  to r}
and Corollary~\ref{cor:b2}, we
know that $b_1<0$ and $b_2>0$. Hence after applying the transformation, there
are exactly the same number of negative signs in $P(u)=P_2\ast S_r$ of $T$ as
in $P(u)=P_0\ast S_{r+1}$ of $T^\prime$. Hence, to prove that the
transformation is proper, we need to compare $f_{T}(u) $ and
$f_{T^\prime}(u)$, the values obtained by the application of the algorithm
\texttt{Diagonalize$\left(T,-2+\frac{2}{n}\right)$} and
\texttt{Diagonalize$\left(T^\prime,-2+\frac{2}{n}\right)$}.

Applying the algorithm we obtain
\begin{eqnarray*}
f_{T}(u) & = & {\rm deg}_{T}(u) - d_n - \xi -\frac{1}{b_{2}(r)} \\
f_{T'}(u) & = & {\rm deg}_{T}(u)+r - d_n - \xi  -\frac{r+1}{x_{2}} \\
& = & {\rm deg}_{T}(u)+r (1-\frac{1}{x_{2}} ) - d_n - \xi  -\frac{1}{x_{2}},
\end{eqnarray*}
  where $\xi$ is the processing of the part not affected by the transformation.
  Therefore, using the value of $b_1(r)$ and $b_2(r)$ in (\ref{dependparamcase2})
and $x_2$ in (\ref{dependparamx})
\begin{eqnarray}
f_{T'}(u)- f_{T}(u) & =  & \frac{1}{b_{2}(r)}-\frac{1}{x_{2}} +  r     (1-\frac{1}{x_{2}}    ) \nonumber \\
  & = & \frac{1}{b_{2}(r)}-\frac{1}{x_{2}} -x_{1}+  x_{1}+ r \left(1-\frac{1}{x_{2}} \right) \nonumber \\
  & = & \frac{1}{b_{2}(r)}-\frac{1}{x_{2}} -x_{1}+  b_{1}(r) \nonumber \\
 & = &      \left( b_{1}(r)+\frac{1}{b_{2}(r)}     \right)-      \left( x_{1}+\frac{1}{x_{2}}     \right) \nonumber \\
 &  = &      b_{1}(r)+\frac{1}{\frac{2}{n}-\frac{1} {b_{1}(r)}}     -      \left( x_{1}+\frac{1}{\frac{2}{n}-\frac{1}{x_{1}}}      \right) \label{eq:starup}.
\end{eqnarray}
We need to show that (\ref{eq:starup})
non-negative for $0 \leq r  \leq \lfloor  \frac{n}{4}\rfloor -1$.
Consider the function
$$
g(t)=t+\frac{1}{\frac{2}{n}-\frac{1}{t}}-      \left(
x_{1}+\frac{1}{\frac{2}{n}-\frac{1}{x_{1}}}   \right)
$$
We see that $g(x_1)= 0$.
By Lemma~\ref{properties x and b} and Lemma~\ref{depend   init  cond  to r}
one has $-1<x_{1} \leq b_{1}(r)<0$.
We need to show that
$g(t)$ is positive in the interval $(x_1, 0)$.
Now, $g^\prime(t) = 1- \frac{1}{\left(\frac{2}{n}-\frac{1}{t}\right)^2} \cdot \frac{1}{t^2}$. Because $\frac{2t}{n}-1 < -1$, it implies that $  \frac{1} {\left(\frac{2}{n}- \frac{1}{t} \right)^2} \cdot \frac{1}{t^2} < 1 $, which it turns imply that $g'(t) >0$ and thus $g(t)$ is strictly increasing and, hence $g(t) > 0 $ in the desired interval.
This shows that
for each $r \leq  \lfloor  \frac{n}{4}\rfloor -1$,
one has $g(b_1(r)) \geq 0$,
and therefore $f_{T}(u) \leq f_{T^\prime}(u)$.
Therefore, the number of positive values produced by
\texttt{Diagonalize$\left(T^\prime,-2+\frac{2}{n}\right)$}
either equals or exceeds by one,
the number of positive values in
\texttt{Diagonalize$\left(T,-2+\frac{2}{n}\right)$},
showing the Star-up transformation is proper.
\end{proof}
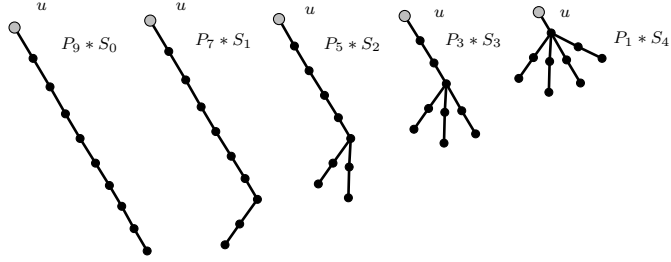
\begin{figure}[t!]
  \centering
   \begin{tikzpicture}[line cap=round,line join=round,>=triangle 45,x=1cm,y=1cm, scale=0.6, every node/.style={scale=0.8}]
\clip(-11.908897407777282,-1.179516504425329) rectangle (7.932543127725042,4.839113942622997);
\draw (-9.048105964315377,3.986228229976376) node[anchor=north west] {{\tiny $P_9\ast S_0$}};
\draw [line width=1pt] (-9.88518021533145,4.020394525936216)-- (-9.475184663813371,3.33706860673942);
\draw [line width=1pt] (-9.475184663813371,3.33706860673942)-- (-9.099355408255134,2.704992131482385);
\draw [line width=1pt] (-9.099355408255134,2.704992131482385)-- (-8.757692448656735,2.1070819521851885);
\draw [line width=1pt] (-8.757692448656735,2.1070819521851885)-- (-8.433112637038255,1.5604212168277525);
\draw [line width=1pt] (-8.433112637038255,1.5604212168277525)-- (-8.091449677439856,1.0137604814703163);
\draw [line width=1pt] (-8.091449677439856,1.0137604814703163)-- (-7.783953013801298,0.5183491900526396);
\draw [line width=1pt] (-7.783953013801298,0.5183491900526396)-- (-7.51062264612258,0.05710419459480273);
\draw [line width=1pt] (-7.51062264612258,0.05710419459480273)-- (-7.23729227844386,-0.43830709682287394);
\draw [line width=1pt] (-7.23729227844386,-0.43830709682287394)-- (-6.946878762785221,-0.9337183882405506);
\draw (-6.075638215809309,4.071643969875975) node[anchor=north west] {{\tiny $P_7\ast S_1$}};
\draw [line width=1pt] (-6.87854617086554,4.174142857755493)-- (-6.468550619347463,3.4908169385586993);
\draw [line width=1pt] (-6.468550619347463,3.4908169385586993)-- (-6.092721363789224,2.858740463301664);
\draw [line width=1pt] (-6.092721363789224,2.858740463301664)-- (-5.751058404190825,2.260830284004468);
\draw [line width=1pt] (-5.751058404190825,2.260830284004468)-- (-5.426478592572346,1.7141695486470314);
\draw [line width=1pt] (-5.426478592572346,1.7141695486470314)-- (-5.084815632973948,1.1675088132895954);
\draw [line width=1pt] (-5.084815632973948,1.1675088132895954)-- (-4.777318969335388,0.6720975218719185);
\draw [line width=1pt] (-4.777318969335388,0.6720975218719185)-- (-4.503988601656669,0.2108525264140817);
\draw [line width=1pt] (-4.503988601656669,0.2108525264140817)-- (-4.896901005194829,-0.3358082089433548);
\draw [line width=1pt] (-4.896901005194829,-0.3358082089433548)-- (-5.221480816813307,-0.7970532044011917);
\draw (-3.222752503162679,4.003311377956296) node[anchor=north west] {{\tiny $P_5\ast S_2$}};
\draw [line width=1pt] (-4.025660458218912,4.208309153715334)-- (-3.6839974986205144,3.610398974418139);
\draw [line width=1pt] (-3.6839974986205144,3.610398974418139)-- (-3.3594176870020354,3.0637382390607026);
\draw [line width=1pt] (-3.3594176870020354,3.0637382390607026)-- (-3.017754727403637,2.517077503703266);
\draw [line width=1pt] (-3.017754727403637,2.517077503703266)-- (-2.7102580637650773,2.0216662122855893);
\draw [line width=1pt] (-2.7102580637650773,2.0216662122855893)-- (-2.436927696086358,1.5604212168277531);
\draw [line width=1pt] (-2.436927696086358,1.5604212168277531)-- (-2.8298400996245183,1.0137604814703163);
\draw [line width=1pt] (-2.8298400996245183,1.0137604814703163)-- (-3.1544199112429965,0.5525154860124795);
\draw [line width=1pt] (-2.436927696086358,1.5604212168277531)-- (-2.471093992046198,0.9283447415707166);
\draw [line width=1pt] (-2.471093992046198,0.9283447415707166)-- (-2.488177140026118,0.2450188223739212);
\draw (-0.540698270315249,4.088727117855895) node[anchor=north west] {{\tiny $P_3\ast S_3$}};
\draw [line width=1pt] (-1.2411073374919628,4.276641745635015)-- (-0.8994443778935653,3.729981010277579);
\draw [line width=1pt] (-0.8994443778935653,3.729981010277579)-- (-0.5919477142550056,3.234569718859902);
\draw [line width=1pt] (-0.5919477142550056,3.234569718859902)-- (-0.3186173465762866,2.7733247234020655);
\draw [line width=1pt] (-0.3186173465762866,2.7733247234020655)-- (-0.7115297501144466,2.226663988044629);
\draw [line width=1pt] (-0.7115297501144466,2.226663988044629)-- (-1.0361095617329248,1.765418992586792);
\draw [line width=1pt] (-0.3186173465762866,2.7733247234020655)-- (-0.35278364253612626,2.1412482481450286);
\draw [line width=1pt] (-0.35278364253612626,2.1412482481450286)-- (-0.36986679051604654,1.4579223289482333);
\draw [line width=1pt] (-0.3186173465762866,2.7733247234020655)-- (0.023045613022112804,2.1754145441048682);
\draw [line width=1pt] (0.023045613022112804,2.1754145441048682)-- (0.33054227666067165,1.6629201047072717);
\draw (3.192605406471941,4.120803593112931) node[anchor=north west] {{\tiny $P_1\ast S_4$}};
\draw [line width=1pt] (1.7313604110141059,4.362057485534614)-- (2.004690778692825,3.900812490076778);
\draw [line width=1pt] (2.004690778692825,3.900812490076778)-- (1.6117783751546648,3.35415175471934);
\draw [line width=1pt] (1.6117783751546648,3.35415175471934)-- (1.2871985635361867,2.8929067592615034);
\draw [line width=1pt] (2.004690778692825,3.900812490076778)-- (1.9705244827329853,3.2687360148197397);
\draw [line width=1pt] (1.9705244827329853,3.2687360148197397)-- (1.953441334753065,2.5854100956229447);
\draw [line width=1pt] (2.004690778692825,3.900812490076778)-- (2.3463537382912243,3.3029023107795794);
\draw [line width=1pt] (2.3463537382912243,3.3029023107795794)-- (2.653850401929783,2.7904078713819827);
\draw (-9.594766699672816,4.77205303705269) node[anchor=north west] {{\tiny $u$}};
\draw (-6.502716915307307,4.754969889072771) node[anchor=north west] {{\tiny $u$}};
\draw (-3.6498312026606774,4.669554149173171) node[anchor=north west] {{\tiny $u$}};
\draw (-0.7798623420341281,4.618304705233411) node[anchor=north west] {{\tiny $u$}};
\draw (2.0,4.60871822089205) node[anchor=north west] {{\tiny $u$}};
\draw [line width=1pt] (2.004690778692825,3.900812490076778)-- (2.602600957990023,3.593315826438219);
\draw [line width=1pt] (2.602600957990023,3.593315826438219)-- (3.1321785453675415,3.3370686067394204);
\begin{scriptsize}
\draw [fill=lightgray] (-9.88518021533145,4.020394525936216) circle (3.5pt);
\draw [fill=black] (-9.475184663813371,3.33706860673942) circle (2.5pt);
\draw [fill=black] (-9.099355408255134,2.704992131482385) circle (2.5pt);
\draw [fill=black] (-8.757692448656735,2.1070819521851885) circle (2.5pt);
\draw [fill=black] (-8.433112637038255,1.5604212168277525) circle (2.5pt);
\draw [fill=black] (-8.091449677439856,1.0137604814703163) circle (2.5pt);
\draw [fill=black] (-7.783953013801298,0.5183491900526396) circle (2.5pt);
\draw [fill=black] (-7.51062264612258,0.05710419459480273) circle (2.5pt);
\draw [fill=black] (-7.23729227844386,-0.43830709682287394) circle (2.5pt);
\draw [fill=black] (-6.946878762785221,-0.9337183882405506) circle (2.5pt);
\draw [fill=lightgray] (-6.87854617086554,4.174142857755493) circle (3.5pt);
\draw [fill=black] (-6.468550619347463,3.4908169385586993) circle (2.5pt);
\draw [fill=black] (-6.092721363789224,2.858740463301664) circle (2.5pt);
\draw [fill=black] (-5.751058404190825,2.260830284004468) circle (2.5pt);
\draw [fill=black] (-5.426478592572346,1.7141695486470314) circle (2.5pt);
\draw [fill=black] (-5.084815632973948,1.1675088132895954) circle (2.5pt);
\draw [fill=black] (-4.777318969335388,0.6720975218719185) circle (2.5pt);
\draw [fill=black] (-4.503988601656669,0.2108525264140817) circle (2.5pt);
\draw [fill=black] (-4.896901005194829,-0.3358082089433548) circle (2.5pt);
\draw [fill=black] (-5.221480816813307,-0.7970532044011917) circle (2.5pt);
\draw [fill=lightgray] (-4.025660458218912,4.208309153715334) circle (3.5pt);
\draw [fill=black] (-3.6839974986205144,3.610398974418139) circle (2.5pt);
\draw [fill=black] (-3.3594176870020354,3.0637382390607026) circle (2.5pt);
\draw [fill=black] (-3.017754727403637,2.517077503703266) circle (2.5pt);
\draw [fill=black] (-2.7102580637650773,2.0216662122855893) circle (2.5pt);
\draw [fill=black] (-2.436927696086358,1.5604212168277531) circle (2.5pt);
\draw [fill=black] (-2.8298400996245183,1.0137604814703163) circle (2.5pt);
\draw [fill=black] (-3.1544199112429965,0.5525154860124795) circle (2.5pt);
\draw [fill=black] (-2.471093992046198,0.9283447415707166) circle (2.5pt);
\draw [fill=black] (-2.488177140026118,0.2450188223739212) circle (2.5pt);
\draw [fill=lightgray] (-1.2411073374919628,4.276641745635015) circle (3.5pt);
\draw [fill=black] (-0.8994443778935653,3.729981010277579) circle (2.5pt);
\draw [fill=black] (-0.5919477142550056,3.234569718859902) circle (2.5pt);
\draw [fill=black] (-0.3186173465762866,2.7733247234020655) circle (2.5pt);
\draw [fill=black] (-0.7115297501144466,2.226663988044629) circle (2.5pt);
\draw [fill=black] (-1.0361095617329248,1.765418992586792) circle (2.5pt);
\draw [fill=black] (-0.35278364253612626,2.1412482481450286) circle (2.5pt);
\draw [fill=black] (-0.36986679051604654,1.4579223289482333) circle (2.5pt);
\draw [fill=black] (0.023045613022112804,2.1754145441048682) circle (2.5pt);
\draw [fill=black] (0.33054227666067165,1.6629201047072717) circle (2.5pt);
\draw [fill=lightgray] (1.7313604110141059,4.362057485534614) circle (3.5pt);
\draw [fill=black] (2.004690778692825,3.900812490076778) circle (2.5pt);
\draw [fill=black] (1.6117783751546648,3.35415175471934) circle (2.5pt);
\draw [fill=black] (1.2871985635361867,2.8929067592615034) circle (2.5pt);
\draw [fill=black] (1.9705244827329853,3.2687360148197397) circle (2.5pt);
\draw [fill=black] (1.953441334753065,2.5854100956229447) circle (2.5pt);
\draw [fill=black] (2.3463537382912243,3.3029023107795794) circle (2.5pt);
\draw [fill=black] (2.653850401929783,2.7904078713819827) circle (2.5pt);
\draw [fill=black] (2.602600957990023,3.593315826438219) circle (2.5pt);
\draw [fill=black] (3.1321785453675415,3.3370686067394204) circle (2.5pt);
\end{scriptsize}
\end{tikzpicture}
 \caption{Repeated applications of Star-up}\label{fig:ex4}
\end{figure}

\begin{example}
\label{ex4}
We consider the proper transformation of a tree $T$ with a path $P_9$ into a
new tree $T^\prime$ with the replacement of $P_9$ by $P_1\ast S_4$, by
successive applications of the Star-up transform. See Figure \ref{fig:ex4}
for an illustration.
\end{example}

We now present the second proper transformation that we will use.

\begin{proposition}(Star-down transform) \label{star-down transformation}
Consider a transformation on a tree $T$ that takes one pendant path $P_{1}$
on a vertex $u$ connected to a sun $S_{r}$ and another pendant path $P_{2}$
on the vertex $u$ and produces a new tree $T'$ with a sun $S_{r+1}$ attached
in $P_{1}$. If $0 \leq r\leq \left\lfloor \frac{n}{4} \right\rfloor -1$  then
$$ P_{1}*S_{r} \oplus  P_{2} \to P_{1} * S_{r+1}$$
is a proper transformation.
\end{proposition}
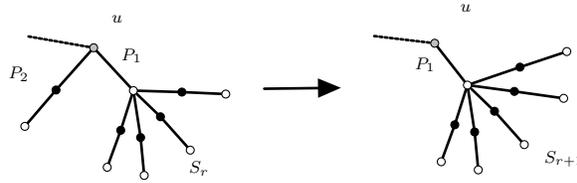
\begin{figure}[!h]
  \centering
  \definecolor{cqcqcq}{rgb}{0.7529411764705882,0.7529411764705882,0.7529411764705882}
\definecolor{ffffff}{rgb}{1,1,1}
\begin{tikzpicture}[line cap=round,line join=round,>=triangle 45,x=1cm,y=1cm,scale=0.6, every node/.style={scale=0.8}]
\clip(0.376571428571434,-0.0) rectangle (28.10152380952386,4.912190476190467);
\draw [line width=1pt] (19.35466666666666,2.199238095238094)-- (19.07466666666666,1.3192380952380938);
\draw [line width=1pt] (19.07466666666666,1.3192380952380938)-- (18.77466666666666,0.4992380952380935);
\draw [line width=1pt] (19.35466666666666,2.199238095238094)-- (19.514666666666656,1.1592380952380936);
\draw [line width=1pt] (19.514666666666656,1.1592380952380936)-- (19.59466666666666,0.3192380952380933);
\draw [line width=1pt] (19.35466666666666,2.199238095238094)-- (19.95466666666666,1.6192380952380936);
\draw [line width=1pt] (19.95466666666666,1.6192380952380936)-- (20.614666666666658,0.8792380952380934);
\draw [line width=1pt] (19.35466666666666,2.199238095238094)-- (20.39466666666666,2.0592380952380944);
\draw [line width=1pt] (20.39466666666666,2.0592380952380944)-- (21.486857142857165,1.905714285714279);
\draw (19.00764,4.17562) node[anchor=north west] {{\tiny $u$}};
\draw [line width=1pt,dash pattern=on 1pt off 1pt] (18.628952380952448,3.1272380952380883)-- (17.269142857142917,3.2885714285714216);
\draw (18,3.0) node[anchor=north west] {{\tiny $P_{1}$}};
\draw (11.5,3.2) node[anchor=north west] {{\tiny $P_{1}$}};
\draw [line width=1pt] (11.953523809523798,2.0766666666666653)-- (11.6735238095238,1.1966666666666657);
\draw [line width=1pt] (11.6735238095238,1.1966666666666657)-- (11.3735238095238,0.3766666666666654);
\draw [line width=1pt] (11.953523809523798,2.0766666666666653)-- (12.1135238095238,1.0366666666666655);
\draw [line width=1pt] (12.1135238095238,1.0366666666666655)-- (12.193523809523798,0.19666666666666566);
\draw [line width=1pt] (11.953523809523798,2.0766666666666653)-- (12.553523809523798,1.4966666666666655);
\draw [line width=1pt] (12.553523809523798,1.4966666666666655)-- (13.2135238095238,0.7566666666666653);
\draw [line width=1pt] (11.953523809523798,2.0766666666666653)-- (13.028380952380964,2.044);
\draw [line width=1pt] (13.028380952380964,2.044)-- (14,2);
\draw (13,0.69488) node[anchor=north west] {{\tiny $S_{r}$}};
\draw (11.28784,3.93604) node[anchor=north west] {{\tiny $u$}};
\draw [line width=1pt] (11.953523809523798,2.0766666666666653)-- (11.073523809523799,3.0266666666666664);
\draw [line width=1pt,dash pattern=on 1pt off 1pt] (11.073523809523799,3.0266666666666664)-- (9.581333333333323,3.2857142857142847);
\draw (8.99852,2.79138) node[anchor=north west] {{\tiny $P_{2}$}};
\draw [line width=1pt] (10.239619047619062,2.090095238095238)-- (11.073523809523799,3.0266666666666664);
\draw [->,line width=1pt] (14.852571428571416,2.118380952380951) -- (16.55257142857142,2.1383809523809516);
\draw (20.87174,0.92798) node[anchor=north west] {{\tiny $S_{r+1}$}};
\draw [line width=1pt] (18.628952380952448,3.1272380952380883)-- (19.35466666666666,2.199238095238094);
\draw [line width=1pt] (10.239619047619062,2.090095238095238)-- (9.548190476190475,1.283428571428565);
\draw [line width=1pt] (19.35466666666666,2.199238095238094)-- (20.518857142857165,2.5971428571428503);
\draw [line width=1pt] (20.518857142857165,2.5971428571428503)-- (21.556,2.942857142857136);
\begin{scriptsize}
\draw [fill=ffffff] (19.35466666666666,2.199238095238094) circle (2.5pt);
\draw [fill=black] (19.07466666666666,1.3192380952380938) circle (2.5pt);
\draw [fill=ffffff] (18.77466666666666,0.4992380952380935) circle (2.5pt);
\draw [fill=black] (19.514666666666656,1.1592380952380936) circle (2.5pt);
\draw [fill=ffffff] (19.59466666666666,0.3192380952380933) circle (2.5pt);
\draw [fill=black] (19.95466666666666,1.6192380952380936) circle (2.5pt);
\draw [fill=ffffff] (20.614666666666658,0.8792380952380934) circle (2.5pt);
\draw [fill=black] (20.39466666666666,2.0592380952380944) circle (2.5pt);
\draw [fill=ffffff] (21.486857142857165,1.905714285714279) circle (2.5pt);
\draw [fill=cqcqcq] (18.628952380952448,3.1272380952380883) circle (2.5pt);
\draw [fill=ffffff] (11.953523809523798,2.0766666666666653) circle (2.5pt);
\draw [fill=black] (11.6735238095238,1.1966666666666657) circle (2.5pt);
\draw [fill=ffffff] (11.3735238095238,0.3766666666666654) circle (2.5pt);
\draw [fill=black] (12.1135238095238,1.0366666666666655) circle (2.5pt);
\draw [fill=ffffff] (12.193523809523798,0.19666666666666566) circle (2.5pt);
\draw [fill=black] (12.553523809523798,1.4966666666666655) circle (2.5pt);
\draw [fill=ffffff] (13.2135238095238,0.7566666666666653) circle (2.5pt);
\draw [fill=black] (13.028380952380964,2.044) circle (2.5pt);
\draw [fill=ffffff] (14,2) circle (2.5pt);
\draw [fill=cqcqcq] (11.073523809523799,3.0266666666666664) circle (2.5pt);
\draw [fill=black] (10.239619047619062,2.090095238095238) circle (2.5pt);
\draw [fill=ffffff] (9.548190476190475,1.283428571428565) circle (2.5pt);
\draw [fill=black] (20.518857142857165,2.5971428571428503) circle (2.5pt);
\draw [fill=ffffff] (21.556,2.942857142857136) circle (2.5pt);
\draw[color=ffffff] (21.77612,3.52343) node {{\tiny $N$}};
\end{scriptsize}
\end{tikzpicture}
  \caption{Star-down.}\label{stardown}
\end{figure}
\begin{proof}
We root both trees at $u$.
We apply the diagonalization algorithm on $T$ and $T^\prime$ and notice that
the number of positive vertices remains unchanged except at the vertex $u$
(as $r \leq \left\lfloor \frac{n}{4} \right\rfloor -1$ we know that $b_1(r)<0$, by Lemma \ref{depend   init  cond  to r}).
Therefore we must compare $f_{T}(u) $ and $f_{T'}(u)$.
We will have $\sigma(T) \leq \sigma(T^\prime)$
if $f_{T}(u) \leq f_{T'}(u)$.
Applying the algorithm we obtain
\begin{eqnarray*}
f_{T}(u) & = & {\rm deg}_{T}(u) - d_n - \xi -\frac{1}{b_{1}(r)}-\frac{1}{x_{2}} \\
f_{T'}(u) & = & {\rm deg}_{T}(u)-1 - d_n - \xi  -\frac{1}{b_{1}(r+1)}
\end{eqnarray*}
 where $\xi$ is the processing of the part not affected by the transform.
Using Lemma~\ref{properties x and b}, part e
\begin{eqnarray}
     f_{T'}(u)- f_{T}(u)  &  =  & -\left(1-\frac{1}{x_{2}}\right)+\left(\frac{1}{b_{1}(r)}-\frac{1}{b_{1}(r+1)}\right) \nonumber \\
&  = &\left(1-\frac{1}{x_{2}}\right) \left(-1+ \frac{1}{b_{1}(r)b_{1}(r+1)}\right) \label{eqstardown}
\end{eqnarray}
Using the properties in Lemma~\ref{properties x and b} and Lemma~\ref{depend   init  cond  to r},
$-1< b_{1}(r)< b_{1}(r+1)<0$ and $x_2 > 1$,
and so (\ref{eqstardown}) is positive.
This completes the proof.
\end{proof}
  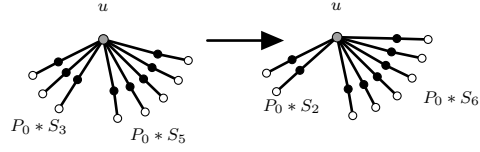
\begin{figure}[!h]
    \centering
    \definecolor{ffffff}{rgb}{1,1,1}
\definecolor{aqaqaq}{rgb}{0.6274509803921569,0.6274509803921569,0.6274509803921569}
\begin{tikzpicture}[line cap=round,line join=round,>=triangle 45,x=1cm,y=1cm, scale=0.6, every node/.style={scale=0.8}]
\clip(-13.6012295968823,0.0) rectangle (6.463158959939575,4.805676620378542);
\draw (-6.906316985123277,1.3991466915441925) node[anchor=north west] {{\tiny $P_0\ast S_5$}};
\draw [line width=1pt] (-6.712667580186978,2.1382756118033517)-- (-6.377235745288217,1.5787974130889981);
\draw [line width=1pt] (-6.413750662374381,2.3113327747474863)-- (-5.973241520334764,1.8865561020664288);
\draw [line width=1pt] (-6.224961030071689,2.5630522844844092)-- (-5.642859663805052,2.279867836030371);
\draw (-7.618541731055087,4.16315577689871) node[anchor=north west] {{\tiny $u$}};
\draw [line width=1pt] (-6.0991012752032265,2.8619692022970056)-- (-5.422605092785244,2.7203769780699862);
\draw [line width=1pt] (-8.30164698540131,2.688912039352871)-- (-8.868015882309388,2.389995121540275);
\draw [line width=1pt] (-8.128589822457174,2.4529249989745057)-- (-8.64776131128958,1.9809509182177751);
\draw [line width=1pt] (-7.876870312720248,2.2956003053889287)-- (-8.285914516042753,1.6505690616880637);
\draw [line width=1pt] (-7.069111212764833,2.0468308175584737)-- (-6.987714098944055,1.436352463902636);
\draw (-9.551723184298574,1.6) node[anchor=north west] {{\tiny $P_0\ast S_3$}};
\draw (-0.4148971579161992,2.2317207159311977) node[anchor=north west] {{\tiny $P_0\ast S_6$}};
\draw (-3.9556716091200594,2.0503236021104196) node[anchor=north west] {{\tiny $P_0\ast S_2$}};
%\draw [->,line width=1pt] (-5.0,0.67) -- (-3.52,0.67);
%\draw (-5.176628316431735,1.432089578058847) node[anchor=north west] {{\tiny Star-star}};
\draw [line width=1pt] (-8.30164698540131,2.688912039352871)-- (-7.303111235491821,3.18809139660251);
\draw [line width=1pt] (-8.128589822457174,2.4529249989745057)-- (-7.303111235491821,3.18809139660251);
\draw [line width=1pt] (-7.876870312720248,2.2956003053889287)-- (-7.303111235491821,3.18809139660251);
\draw [line width=1pt] (-7.069111212764833,2.0468308175584737)-- (-7.303111235491821,3.18809139660251);
\draw [line width=1pt] (-6.712667580186978,2.1382756118033517)-- (-7.303111235491821,3.18809139660251);
\draw [line width=1pt] (-6.413750662374381,2.3113327747474863)-- (-7.303111235491821,3.18809139660251);
\draw [line width=1pt] (-6.224961030071689,2.5630522844844092)-- (-7.303111235491821,3.18809139660251);
\draw [line width=1pt] (-6.0991012752032265,2.8619692022970056)-- (-7.303111235491821,3.18809139660251);
\draw [line width=1pt] (-1.533792892847703,2.1874478223882)-- (-1.2695668530343724,1.680543805364971);
\draw [line width=1pt] (-1.234875975035106,2.360504985332335)-- (-0.7943668329954888,1.935728312651277);
\draw [line width=1pt] (-1.046086342732413,2.6122244950692575)-- (-0.46398497646577663,2.3290400466152197);
\draw (-2.449825003435659,4.2038543338091) node[anchor=north west] {{\tiny $u$}};
\draw [line width=1pt] (-0.920226587863951,2.911141412881854)-- (-0.2437304054459689,2.7695491886548345);
\draw [line width=1pt] (-3.12277229806203,2.7380842499377196)-- (-3.6891411949701074,2.4391673321251237);
\draw [line width=1pt] (-2.949715135117897,2.5020972095593543)-- (-3.4688866239502993,2.0301231288026234);
\draw [line width=1pt] (-0.8015334485648966,3.2067396895045666)-- (-0.10965798108828018,3.186390411049372);
\draw [line width=1pt] (-1.8800452066902102,2.0875293744688626)-- (-1.7782988144142373,1.4974002992682196);
\draw [line width=1pt] (-3.12277229806203,2.7380842499377196)-- (-2.124236548152547,3.237263607187358);
\draw [line width=1pt] (-2.949715135117897,2.5020972095593543)-- (-2.124236548152547,3.237263607187358);
\draw [line width=1pt] (-0.8015334485648966,3.2067396895045666)-- (-2.124236548152547,3.237263607187358);
\draw [line width=1pt] (-1.8800452066902102,2.0875293744688626)-- (-2.124236548152547,3.237263607187358);
\draw [line width=1pt] (-1.533792892847703,2.1874478223882)-- (-2.124236548152547,3.237263607187358);
\draw [line width=1pt] (-1.234875975035106,2.360504985332335)-- (-2.124236548152547,3.237263607187358);
\draw [line width=1pt] (-1.046086342732413,2.6122244950692575)-- (-2.124236548152547,3.237263607187358);
\draw [line width=1pt] (-0.920226587863951,2.911141412881854)-- (-2.124236548152547,3.237263607187358);
\begin{scriptsize}
\draw [fill=aqaqaq] (-7.303111235491821,3.18809139660251) circle (3pt);
\draw [fill=black] (-6.712667580186978,2.1382756118033517) circle (2.5pt);
\draw [fill=ffffff] (-6.377235745288217,1.5787974130889981) circle (2.5pt);
\draw [fill=black] (-6.413750662374381,2.3113327747474863) circle (2.5pt);
\draw [fill=ffffff] (-5.973241520334764,1.8865561020664288) circle (2.5pt);
\draw [fill=black] (-6.224961030071689,2.5630522844844092) circle (2.5pt);
\draw [fill=ffffff] (-5.642859663805052,2.279867836030371) circle (2.5pt);
\draw [fill=black] (-6.0991012752032265,2.8619692022970056) circle (2.5pt);
\draw [fill=ffffff] (-5.422605092785244,2.7203769780699862) circle (2.5pt);
\draw [fill=black] (-8.30164698540131,2.688912039352871) circle (2.5pt);
\draw [fill=ffffff] (-8.868015882309388,2.389995121540275) circle (2.5pt);
\draw [fill=black] (-8.128589822457174,2.4529249989745057) circle (2.5pt);
\draw [fill=ffffff] (-8.64776131128958,1.9809509182177751) circle (2.5pt);
\draw [fill=black] (-7.876870312720248,2.2956003053889287) circle (2.5pt);
\draw [fill=ffffff] (-8.285914516042753,1.6505690616880637) circle (2.5pt);
\draw [fill=black] (-7.069111212764833,2.0468308175584737) circle (2.5pt);
\draw [fill=ffffff] (-6.987714098944055,1.436352463902636) circle (2.5pt);
\draw [fill=aqaqaq] (-2.124236548152547,3.237263607187358) circle (3pt);
\draw [fill=black] (-1.533792892847703,2.1874478223882) circle (2.5pt);
\draw [fill=ffffff] (-1.2695668530343724,1.680543805364971) circle (2.5pt);
\draw [fill=black] (-1.234875975035106,2.360504985332335) circle (2.5pt);
\draw [fill=ffffff] (-0.7943668329954888,1.935728312651277) circle (2.5pt);
\draw [fill=black] (-1.046086342732413,2.6122244950692575) circle (2.5pt);
\draw [fill=ffffff] (-0.46398497646577663,2.3290400466152197) circle (2.5pt);
\draw [fill=black] (-0.920226587863951,2.911141412881854) circle (2.5pt);
\draw [fill=ffffff] (-0.2437304054459689,2.7695491886548345) circle (2.5pt);
\draw [fill=black] (-3.12277229806203,2.7380842499377196) circle (2.5pt);
\draw [fill=ffffff] (-3.6891411949701074,2.4391673321251237) circle (2.5pt);
\draw [fill=black] (-2.949715135117897,2.5020972095593543) circle (2.5pt);
\draw [fill=ffffff] (-3.4688866239502993,2.0301231288026234) circle (2.5pt);
\draw [fill=black] (-0.8015334485648966,3.2067396895045666) circle (2.5pt);
\draw [fill=ffffff] (-0.10965798108828018,3.186390411049372) circle (2.5pt);
\draw [fill=black] (-1.8800452066902102,2.0875293744688626) circle (2.5pt);
\draw [fill=ffffff] (-1.7782988144142373,1.4974002992682196) circle (2.5pt);
\draw [->,line width=1pt] (-5,3.16315577689871) -- (-3.3,3.16315577689871);
\end{scriptsize}
\end{tikzpicture}
    \caption{An example of a trivial Star-star transform}\label{star-star:fig1}
  \end{figure}
The third transformation we use is the following.
\begin{samepage}
\begin{proposition}\label{star-star}(Star-star transform)
Suppose $u$ has at least two generalized pendant paths
$P_{q_1}\ast S_{r_1}$ and $P_{q_2}\ast S_{r_2}$, where $q_1, q_2 \in \{ 0,1 \}$.
The following transformations are proper:
\begin{itemize}
  \item[a)] If $q_{1}=q_{2}=0$,
for any $r'_{1},r'_{2}$ such that $r'_{1}+r'_{2}= r_{1}+r_{2}$;
$$
(P_{0} \ast S_{r_1}) \oplus (P_{0} \ast S_{r_2}) \to (P_{0} \ast S_{r'_1}) \oplus (P_{0} \ast S_{r'_2}) ,
$$

  \item[b)] If $q_{1}=q_{2}=1$ and $0 \leq r_1,r_2 \leq \lfloor \frac{n}{4} \rfloor$;
$$(P_{1} \ast S_{r_1}) \oplus (P_1 \ast S_{r_2})
      \to \left\{\begin{array}{lrr}
P_2 \ast S_{r_1+r_2} & \mbox{~if~} & r_1+r_2 \leq \lfloor \frac{n}{4}\rfloor \\
P_0 \ast S_{r_1+r_2-\lfloor \frac{n}{4}\rfloor } \oplus (P_2 \ast
S_{\lfloor \frac{n}{4}\rfloor})
& \mbox{~if~} & r_1+r_2 > \lfloor \frac{n}{4}\rfloor \\
\end{array}\right.
$$
  \item[c)] If $q_{1}=1$, $q_{2}=0$ and $0 \leq r_1,r_2 \leq \lfloor \frac{n}{4} \rfloor$;
$$(P_1 \ast S_{r_1}) \oplus (P_{0} \ast
      S_{r_2}) \to \left\{\begin{array}{lrr}
P_1 \ast S_{r_1+r_2} & \mbox{~if~} & r_1+r_2 \leq \lfloor \frac{n}{4}\rfloor \\
P_0 \ast S_{r_1+r_2-\lfloor \frac{n}{4}\rfloor } \oplus (P_1 \ast
S_{\lfloor \frac{n}{4}\rfloor})
& \mbox{~if~} & r_1+r_2 > \lfloor \frac{n}{4}\rfloor \\
\end{array}\right.$$
\end{itemize}
\end{proposition}
\end{samepage}

\begin{figure}[!ht]
    \centering
    \input{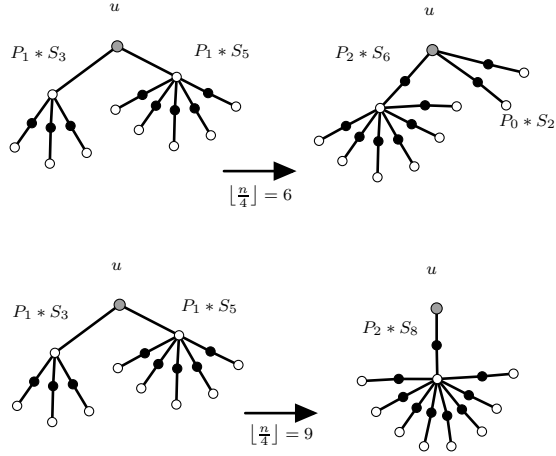}
    \caption{An example of a 1-1 Star-star transform. On top $\left\lfloor
    \frac{n}{4}\right\rfloor = 6$. On bottom $\left\lfloor
    \frac{n}{4}\right\rfloor = 9$.}\label{star-star:fig2}
\end{figure}
\begin{proof}
The case $q_{1}=0$ and $q_{2}=0$ (see Figure~\ref{star-star:fig1}) is
actually a formal rearrangement. We do not change the graph, only perform a
different partition of the $r_1+r_2$ original $P_2's$ attached to $u$. In
this case it is not necessary to consider the sign of any vertices. It is a
{\em trivial} Star-star transform,  shown in Figure~\ref{star-star:fig1}.

We notice that by Corollary \ref{cor:b2}, the case $q_{1}=1$ and $q_{2}=1$
(see Figure~\ref{star-star:fig2}) is clearly proper because the number of
positive signs (black vertices) increases by one. Therefore, it does not
matter what happens in $u$, the transformation is proper.

The case $q_{1}=1$ and $q_{2}=0$ (see Figure~\ref{star-star:fig3})  is
actually a particular application of several Star-down transformations
(Proposition~\ref{star-down transformation}).
In each step, a $P_2$ from the $P_0\ast
S_{r_2}$ is brought down to the gpp $P_1\ast S_{r_1}$. Thus the Star-star
transformation is proper.
\begin{figure}[!ht]
    \centering
    \input{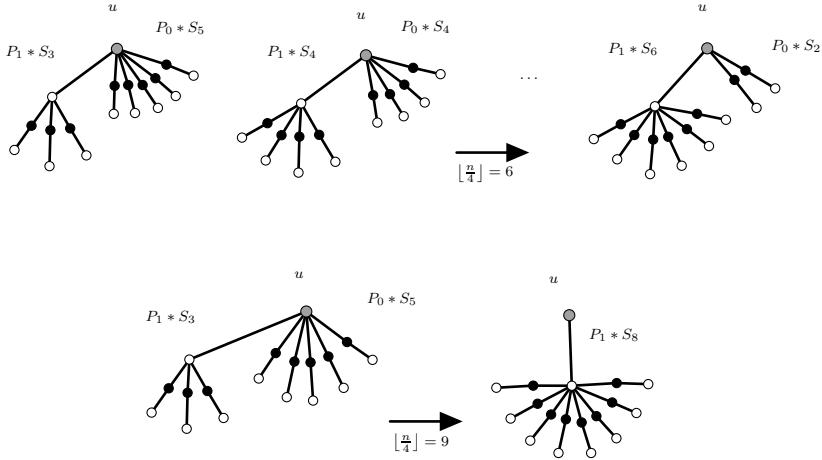}
\caption{An example of a 1-0 Star-star transform.
On top $\left\lfloor \frac{n}{4}\right\rfloor = 6$.
On bottom $\left\lfloor \frac{n}{4}\right\rfloor = 9$.}
\label{star-star:fig3}
  \end{figure}
\end{proof}

\begin{figure}[!h]
\input{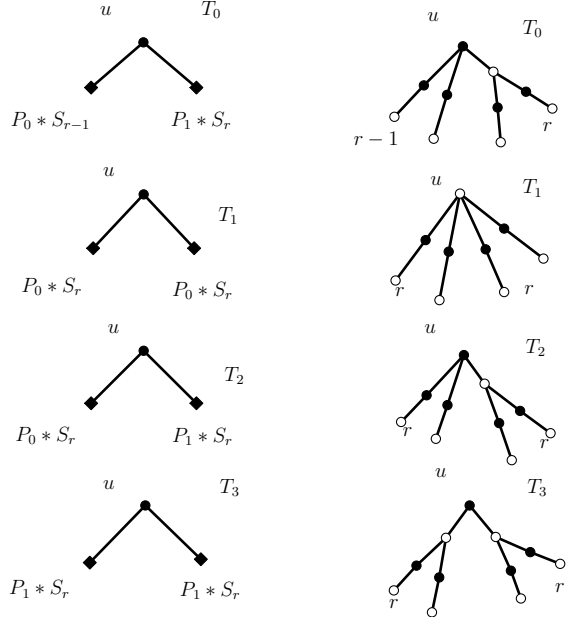}
\caption{Left: trees $T_\alpha, ~\alpha=\{0,1,2,3\}$ in $(P_q,S_r)$
notation. Right: the actual tree $T_\alpha$. The minus signs (white vertices)
and plus signs(black vertices) are defined by Lemma~\ref{properties x and b},
Lemma~\ref{depend   init  cond  to r} and Theorem~\ref{thr:Talpha}.} \label{TreeAlpha}
\end{figure}

\section{Equality for prototype trees}\label{sec:prot}
In Section \ref{sec:red} we present a reduction procedure that properly
transforms several  gpp's attached to a starlike vertex to a single gpp,
leading to an eventual
decrease on the number of starlike vertices of a tree.
Then, in Section \ref{sec:small}, we show that a tree $T$ of order $n$ with $k=0$ or $1$
starlike vertices can be properly transformed into a prototype tree $T_\alpha$,
where $n \equiv \alpha \pmod 4$. These prototype trees $T_\alpha$ are
defined as follows.

\begin{definition}\label{prot}
Let $r \geq 2$ and $\alpha \in \{0,1,2,3\}$. For $n = 4r +\alpha \geq 8$,
define the tree $T_\alpha$ of order $n$ given in Figure \ref{TreeAlpha}.
\end{definition}

If $T$ is properly transformed to $T_\alpha$, that is $\sigma(T) \leq
\sigma(T_\alpha)$, in order to prove the conjecture is true, it will
remain to prove that  $\sigma(T_\alpha) \leq \lfloor\frac{n}{2}\rfloor$.
As a warm up application of the results of the previous section, we
prove here  this result.
In fact, we prove that $T_\alpha$ satisfy the
equality in Theorem \ref{thr:main}, that is $\sigma(T_\alpha) =
\lfloor\frac{n}{2}\rfloor$.

\begin{samepage}
\begin{theorem}\label{thr:Talpha}
For $r \geq 2$, $\alpha \in \{0,1,2,3\}$, let
 $n = 4r +\alpha$, and $d_n = 2- \frac{2}{n}$.  Then
\begin{enumerate}
  \item[(a)] ${\displaystyle m_{T_0}([0, d_n)) = 2r}$;
  \item[(b)] ${\displaystyle m_{T_1}([0, d_n)) = 2r+1}$;
  \item[(c)] ${\displaystyle m_{T_2}([0, d_n)) = 2r+1}$;
  \item[(d)] ${\displaystyle m_{T_3}([0, d_n)) = 2r+2}$.
\end{enumerate}
In particular ${\displaystyle m_{T_\alpha}([0, d_n)) = \left\lceil
\frac{n}{2} \right\rceil}$, or equivalently, ${\displaystyle\sigma(T_\alpha)
= \left \lfloor\frac{n}{2} \right\rfloor}$.
\end{theorem}
\end{samepage}

\begin{proof}
We first observe that in all cases
$\frac{n}{4}= r + \frac{\alpha}{4},$ and
as $0\leq \alpha/4 < 1$,
we have
$r = \left\lfloor \frac{n}{4}\right\rfloor$,
so therefore
when applying
\texttt{Diagonalize($T_\alpha,-2+\frac{2}{n}$)}, the value $b_1 <0$, by Lemma \ref{depend   init  cond  to r}.
We used Maple to compute $f(u)$ for each $\alpha$, as shown below.
\begin{itemize}
  \item [(a)]
      We observe that
$T_0=P(u)=  P_0\ast S_{r-1}  \oplus P_1\ast S_{r}$,
and $n = 4r$.
In $P_0\ast S_{r-1}$ there are $r-1$ negative signs and
      $r-1$ positive signs. Due to the observation above, in $P_1 \ast S_r$
      there are $r+1$ negative signs and $r$ positive.
      To show equality we need to show that
      the value at the vertex $u$ is positive.
Then
\begin{eqnarray}
f(u) & = & \deg u - d_n - (r-1)\frac{1}{x_2} - \frac{1}{b_1} \nonumber \\
& = & r -(2-\frac{2}{4r}) - \frac{(r-1)}{x_2} - \frac{1}{b_1} \nonumber \\
& = & r-2 +\frac{1}{2r} - \frac{r-1}{x_2} - \frac{1}{b_1}.  \label{eq:tzero}
\end{eqnarray}
One can write
$x_2 = \frac{4r^2 + 2r -1 }{(2r)(2r-1)}$
and $b_1 = \frac{-2r^2 + 4r - 1}{2r(4r^2 + 2r -1)}$.
Substituting into  (\ref{eq:tzero}) we obtain
$$
 f(u)= ~
\frac{1}{2}~ {\frac {64\,{r}^{6}+64\,{r}^{5}-36\,{r}^{4}+36\,{r}^{3}
-32\,{r}^{
2}+10\,r-1}{r \left( 4\,{r}^{2}+2\,r-1 \right)  \left( 2\,{r}^
{2}-4\,r
+1 \right) }}
$$
A plot of $f(u)$
shows it is positive for all $r\geq 2$. This proves that $\sigma(T_0) = \lfloor
      \frac{n}{2}\rfloor$.

\vspace{.2in}

\item[(b)] The same reasoning can be done by applying our method, but we may also recall the result \cite[Proposition 3.1]{AKBARI2020}, where it is shown that the spectrum of $T_1$ is $$ \left[0^{[1]},\theta^{[2r-1]},\lambda_1^{[1]},\overline{\theta}^{[2r-1]}, \lambda_2^{[1]}\right],$$
    where $(\theta, \overline{\theta}) = (\frac{3-\sqrt{5}}{2}, \frac{3+ \sqrt{5}}{2}), ~~ \lambda_i,~ i=1,2$ are the roots of $p(x)= x^2 - (2r+3)x  +(4r+1)$. Now it is easy to see that $0,\theta, \lambda_1 < 2-2/n$, while $ \overline{\theta}, \lambda_2 > 2-2/n$. Hence  exactly $2r+1$ Laplacian eigenvalues are smaller than $2-2/n$.

\vspace{.2in}

\item[(c)] $T_2=P(u) = P_0\ast S_r \oplus P_1\ast S_{r}$.
There are $2r+1$ negative signs and $r$ positive
      signs except at the vertex $u$.
Here $n = 4r + 2$.
$$
f(u) = r+1 - d_n - \frac{r}{x_2} - \frac{1}{b_1}
$$
Expressing $n$, $x_2$ and $b_1$ in terms of $r$, and substituting in $f(u)$,
      leads to
$$
f(u)
= \frac {64\,{r}^{6}+256\,{r}^{5}+348\,{r}^{4}+260\,{r}^{3}+95\,
{r}^{2}
+16\,r+1}{ \left( 2\,r+1 \right)  \left( 4\,{r}^{2}+6\,r+1
\right) r
 \left( 6\,r+1 \right) }
$$
which is clearly positive
       proving that $\sigma(T_2) = \lfloor \frac{n}{2}\rfloor$.

\vspace{.2in}

  \item[(d)] Proceeding similarly as in the previous case,
$T_3=P(u)= P_1\ast S_r  \oplus P_1\ast S_{r}$, we obtain $2r+2$ negative signs and
      $2r$ positive signs. To show equality we need to show that the value at $u$ is
      positive.
Here $n = 4r + 3$, and
$$
f(u) = 2 - d_n - \frac{2}{b_1}.
$$
Writing $n$ and $b_1$ in terms of $r$, and substituting into $f(u)$,
yields
$$f(u)=
 \,{\frac { 4 ( 128\,{r}^{4}+448\,{r}^{3}+576\,{r}^{2}+302\,r+55 )}
 { \left( 4
\,r+3 \right)  \left( 64\,{r}^{2}+52\,r+11 \right) }}
$$
which is clearly positive showing that $\sigma(T_3)
      =2r+1= \lfloor \frac{n}{2}\rfloor$.

\end{itemize}
\end{proof}

\section{Reduction: Starlike vertices}
\label{sec:red}

In some cases a starlike vertex $u$ can be properly
transformed in such a way that its generalized pendant paths are reduced to a
{\em single} generalized pendant path. This is possible when the weight of $u$ is
small, more precisely when $w(u) \leq 2\lfloor n/4 \rfloor$.

\begin{theorem} \label{annihilation}
Consider a tree $T$ in $(P_q, S_r)$ representation and $u$ a starlike vertex
with $\ell \geq 2$ generalized pendant paths, or $P(u)= P_{q_1}\ast S_{r_1}
\oplus \cdots \oplus P_{q_\ell }\ast S_{r_\ell}$. If
$w(u) \leq 2\lfloor \frac{n}{4}\rfloor$
then we can properly transform $T$ to $T^\prime$ obtaining
  $$P_{q_1}\ast S_{r_1} \oplus \cdots \oplus P_{q_t}\ast S_{r_t} \Rightarrow  P_{q' }\ast S_{r'}$$
  where $2\leq t \leq \ell$ and
\[  \left\{
      \begin{array}{l}
        q'\equiv \sum q_i \pmod 2 \in \{0,1\} \\
        r' =\frac{w(u) -q'}{2} \leq \lfloor \frac{n}{4}\rfloor
      \end{array}
    \right.
\]
\end{theorem}
\begin{proof}
The proof is by induction on $t$.
For $t =2$ assume
$$
P(u)= P_{q_1}\ast S_{r_1} \oplus  P_{q_2}\ast S_{r_2},
$$
where $w(u) \leq 2\lfloor \frac{n}{4}\rfloor$.
Then we may write
$$
w(u)= q_1 + 2r_1 + q_2 + 2r_2= \alpha_1+\alpha_2 + 2(k_1+k_2 + r_1+r_2)
$$
where $\alpha_i \equiv q_i \pmod 2$
and $k_i = \lfloor \frac{q_i}{2} \rfloor$
for $i=1,2$.
Since
$r_1+k_1 < \frac{w(u)}{2} \leq \lfloor \frac{n}{4}\rfloor$,
by Proposition~\ref{star-up},
the Star-up transform
$$P_{q_1}\ast S_{r_1} \Rightarrow P_{q_1-2}\ast S_{r_1+1}$$
can be performed $k_1$ times.
Hence, the transformation
$$P_{q_1}\ast S_{r_1} \Rightarrow P_{\alpha_1}\ast S_{r_1+k_1}$$
is proper, and similarly,
$$P_{q_2}\ast S_{r_2} \Rightarrow P_{\alpha_2}\ast S_{r_2+k_2}$$
is proper.
As $\alpha_1, \alpha_2 \in \{0,1\}$ by Proposition~\ref{star-star},
a Star-star operation can be
performed since  $r_1+k_1 + r_2+k_2 \leq \lfloor \frac{n}{4}\rfloor$
producing
$$
P_{\alpha_1}\ast S_{r_1+k_1} \oplus P_{\alpha_2}\ast S_{r_2+k_2} \Rightarrow P_{\alpha_1+\alpha_2}\ast S_{r_1+k_1 + r_2+k_2}.
$$

If $q'=\alpha_1+\alpha_2 \in \{0,1\}$ we are done, otherwise, if
$\alpha_1+\alpha_2=2$ we can perform an additional Star-up transformation
$$P_{\alpha_1+\alpha_2}\ast S_{r_1+k_1 + r_2+k_2} \Rightarrow P_{0} \ast S_{r_1+k_1 + r_2+k_2 +1}$$
because $2(r_1+k_1 + r_2+k_2 +1)=w(u) \leq 2\lfloor \frac{n}{4}\rfloor$,
and so $r_1+k_1 + r_2+k_2 \leq \lfloor \frac{n}{4}\rfloor -1$.

To complete the induction, assume the theorem is true for some $t\geq 2$. To obtain the correctness for $t+1$, first  transform
$$
P_{q_1}\ast S_{r_1} \oplus \ldots \oplus P_{q_t}\ast S_{r_t} \Rightarrow P_{q'}\ast S_{r'},
$$
and then repeat the proof for $t = 2$
on $P_{q'}\ast S_{r'} \oplus P_{q_{t+1}}\ast S_{r_{t+1}}$.
\end{proof}

\begin{center}
\begin{figure}[!ht]
{\small {\tt
\begin{tabbing}
aaa\=aaa\=aaa\=aaa\=aaa\=aaa\=aaa\=aaa\= \kill
     \> \texttt{ReduceStarVertex}($T,u$)\\
     \> {\bf input}: a tree $T$ with $n$ vertices \\
     \> \> \> ~ a starlike vertex $u$ with $P(u)= P_{q_1}\ast S_{r_1} \oplus \cdots \oplus P_{q_\ell}\ast S_{r_\ell}$. \\
     \> \> \> ~ precondition $w(u) \leq 2 \lfloor \frac{n}{4} \rfloor$   \\
     \> {\bf output}: a tree $T^\prime$ where the {\gpps} at $u$ are replaced by a single \gpp. \\
     \> \> Let $w = w(u)$ \\
     \> \> Compute  $q'=\sum q_i \mod 2 $. \\
     \> \> Compute $r' =\frac{w -q'}{2}$.      \\
     \> \> Replace in $T$ all {\gpps} at $u$ with $P(u) = P_{q'}\ast S_{r'}$, forming $T'$.\\
     \> \> {\bf if}   $\deg_{T'\setminus P(u)} (u) = 1$  {\bf then} \\
     \> \> \>  find $v$, the nearest vertex to $u$ having $\deg_{T'}(v) > 2$ \\
     \> \> \>  remove $u$ and path to $v$\\
     \> \> \>  create {\gpp} $P(v) = P_{q''} * (P_{q'} * S_{r'}) = P_{q''+q'} * S_{r'}$, where $d(u,v) = q''$ \\
     \>{\bf return} $T'$.
\end{tabbing}}
\caption{Procedure \texttt{ReduceStarVertex}.}\label{red-star}}
\end{figure}
\end{center}

Theorem \ref{annihilation} above justifies the introduction of our Procedure
\texttt{ReduceStarVertex}
shown in Figure \ref{red-star}
which replaces {\em all} the generalized pendant paths at
a starlike vertex by a {\em single} \gpp.
It also labels the new generalized pendant path with
$P_{q'} * S_{r'}$
calculated in the procedure.
If $u$ is a leaf in the tree $T'\setminus P(u)$,
we remove it from starlike vertices list and collapse the path
to the nearest vertex $v$ whose degree is greater than two,
creating the gpp $P_{q'+q''} * S_{r'}$ at $v$ where $d(u,v) = q''$.
If the procedure were applied to $u_1$ in the
tree of Figure \ref{fig:an_example:0001}(right), for example,
we would first form $P_1* S_1$
and then the path of length one
would be collapsed
creating the gpp $P_2* S_1$, as Figure \ref{fig:an_example:0203}(left) illustrates.

As this example illustrates, the vertex $u$, after applying \text{ReduceStarLike}, will no longer be a starlike vertex or it will collapse and removed from the representation.
Hence the number of starlike vertices may be reduced by one.
However, a new starlike vertex may be created.
As an example, consider applying
\texttt{ReduceStarVertex} at vertex $u_1$ of Figure
\ref{fig:an_example:0001}(left).
Initially $P(u_1) = P_1 \ast S_0 \oplus P_1 \ast S_0$, having weight 2.
After applying the procedure we obtain
$P(u_1)= P_1 \ast S_1$.
Now $u_1$ is no
longer a starlike vertex, but its neighbor $v$ now has
$P(v) = P_1 \ast S_1 \oplus P_1 \ast S_0$ with weight 4.
In general, this happens when
the body of the if-statement in \texttt{ReduceStarVertex}
is executed, and the vertex $v$ {\em already} had a single \gpp.
The if-statement is {\em not} always executed as
Figures~\ref{fig:an_example:0405}~and~\ref{fig:an_example:0607} show.
A full example of the reduction procedure is executed at the end of
this paper in Section~\ref{sec:exa}.

We claim that \texttt{ReduceStarVertex} will always be called with a vertex
$u$ satisfying the precondition $w(u) \leq 2 \lfloor \frac{n}{4}\rfloor$.
Note it is only called when there are at least two tarlike vertices,
and is passed the one with smallest weight.

\begin{lemma}\label{lem:nu-star-vertices} Let $T$ be a tree with $n\geq 8$ vertices and $k\geq 2$ starlike vertices. Then at least one of the starlike vertices has weight $w \leq 2 \lfloor \frac{n}{4}\rfloor$.

\end{lemma}
\begin{proof} Let $u_1, u_2, \ldots, u_k$ be the $k\geq 2$ starlike vertices ordered by weight so that $w(u_1)\geq w(u_2)\geq \ldots \geq w(u_k)$. Suppose, by contradiction, that $w(u_k)\geq 2\lfloor \frac{n}{4}\rfloor + 1$.
As $n= 4 \lfloor \frac{n}{4}\rfloor + \alpha$, with $\alpha \in \{0,1,2,3\}$, we see that $n \geq w(u_1)+w(u_2) +2  \geq 2 \, w(u_k) + 2 \geq 2(2\lfloor \frac{n}{4}\rfloor+1)+2 = 4 \lfloor \frac{n}{4}\rfloor+4$, an absurd. Hence, we conclude that $w(u_k) \leq 2 \lfloor \frac{n}{4}\rfloor$ and the result follows.
\end{proof}

The correctness of algorithm \texttt{Transform} depends on
not only showing each procedure is correct,
but also showing that it {\em halts}.
Clearly it halts if and only if its {\bf while} loop halts.
Being a local transformation,
\texttt{ReduceStarVertex}
only operates on a starlike vertex $u$, and can not create
more than one new starlike vertex
since $u$ is adjacent
(or has a path)
to only one other vertex.
Hence the number of starlike vertices
{\em does not} increase by the application of \texttt{ReduceStarVertex}.
Moreover, once a new starlike vertex is created, its weight includes the weight of $u$ and,
hence, the total weight of the starlike vertices increase.
Therefore \texttt{Transform} must stop, as the total weight is bounded by $n$.

\section{Reduction: Small number of starlike vertices}
\label{sec:small}
In this section we show that a tree with fewer than
two starlike vertices can be properly transformed into a prototype tree.
We introduce a new notation for the entire tree.
Here
$$
T = v + X \oplus Y
$$
denotes a tree in which $v$ is the root, $X$ and $Y$
are gpps attached to $v$.
Using this notation,
$T=v+ P_{q'} \ast  S_{r'} \oplus  P_{q''} \ast  S_{r''}$
or
$T=u+ P_a \ast S_{r'} \oplus  P_b \ast S_{r''}$
are different representations of the same tree, provided $q'+q'' = a+ b$ (see Figure~\ref{fig:diff_rep})

\begin{figure}[!ht]
  \centering
\definecolor{ffffff}{rgb}{1,1,1}
\begin{tikzpicture}[line cap=round,line join=round,>=triangle 45,x=1cm,y=1cm,scale=0.7, every node/.style={scale=0.65}]
\clip(-10.981674418604651,-1.031069767441856) rectangle (9.278790697674415,4.4);
\draw [line width=1pt] (3.9948837209302326,2.935116279069768)-- (2.954883720930233,0.9451162790697673);
\draw [line width=1pt] (3.9948837209302326,2.935116279069768)-- (4.934883720930233,0.9451162790697673);
\draw [line width=1pt] (-2.5,0.5)-- (-2.88,-0.14);
\draw [line width=1pt] (-2.04,0.38)-- (-2.02,-0.34);
\draw [line width=1pt] (0.1,0.84)-- (0.13901111225519902,-0.08257277074754431);
\draw [line width=1pt] (0.13901111225519902,-0.08257277074754431)-- (0.17377014694634793,-0.7429944298793733);
\draw [line width=1pt] (0.1,0.84)-- (0.54,0.12);
\draw [line width=1pt] (0.54,0.12)-- (0.92,-0.57);
\draw (-6.681953488372095,3.483534883720928) node[anchor=north west] {$v$};
\draw (-0.5,2.9) node[anchor=north west] {$u$};
\draw (-5.691441860465118,3.5908837209302304) node[anchor=north west] {$T$};
\draw (-2.3372093023255847,3.5008372093023237) node[anchor=north west] {$T$};
\draw (-2.7199069767441895,-0.234046511627908435) node[anchor=north west] {$r'$};
\draw (0.29665116279069337,-0.66120930232557944) node[anchor=north west] {$r''$};
\draw [line width=1pt] (-2.5,0.5)-- (-2.08,1.18);
\draw [line width=1pt] (-2.08,1.18)-- (-2.04,0.38);
\draw (-8.662976744186048,0.7245581395348842) node[anchor=north west] {$X=P_{q'} \ast S_{r'}$};
\draw (-5.894046511627909,0.7245581395348842) node[anchor=north west] {$Y=P_{q''} \ast S_{r''}$};
\draw (-0.964,4.0) node[anchor=north west] {$v$};
\draw [line width=1pt] (-1.02,3.52)-- (-1.3269625666099947,2.8771470259876106);
\draw [line width=1pt] (-1.3269625666099947,2.8771470259876106)-- (-1.6,2.26);
\draw [line width=1pt] (-1.6,2.26)-- (-1.8411290764915043,1.7323989473834942);
\draw [line width=1pt] (-1.8411290764915043,1.7323989473834942)-- (-2.08,1.18);
\draw [line width=1pt] (-1.02,3.52)-- (-0.8127960567284852,2.9935620848287074);
\draw [line width=1pt] (-0.8127960567284852,2.9935620848287074)-- (-0.5799659390462922,2.4696943200437724);
\draw [line width=1pt] (-0.5799659390462922,2.4696943200437724)-- (-0.3568370762675239,1.9458265552588379);
\draw [line width=1pt] (-0.3568370762675239,1.9458265552588379)-- (-0.12400695858533095,1.3831537708602044);
\draw [line width=1pt] (-0.12400695858533095,1.3831537708602044)-- (0.1,0.84);
\draw [line width=1pt] (-6.224837209302326,2.912883720930232)-- (-7.264837209302326,0.9228837209302319);
\draw [line width=1pt] (-6.224837209302326,2.912883720930232)-- (-5.304837209302326,0.9428837209302319);
\draw (4.506325581395343,3.5908837209302304) node[anchor=north west] {$T$};
\draw (1.9174883720930183,0.747069767441861) node[anchor=north west] {$X=P_{a} \ast S_{r'}$};
\draw (4.461302325581389,0.7245581395348842) node[anchor=north west] {$Y=P_{b} \ast S_{r''}$};
\draw [line width=1pt] (-0.3128563387297367,0.07384288536262569)-- (-0.6952057203323747,-0.6039582911147776);
\draw [line width=1pt] (-0.3128563387297367,0.07384288536262569)-- (0.1,0.84);
\draw (3.47,3.5) node[anchor=north west] {$u$};
\begin{scriptsize}
\draw [fill=black] (3.9948837209302326,2.935116279069768) circle (2.5pt);
\draw [fill=black] (2.954883720930233,0.9451162790697673) ++(-3.5pt,0 pt) -- ++(3.5pt,3.5pt)--++(3.5pt,-3.5pt)--++(-3.5pt,-3.5pt)--++(-3.5pt,3.5pt);
\draw [fill=black] (4.934883720930233,0.9451162790697673) ++(-3.5pt,0 pt) -- ++(3.5pt,3.5pt)--++(3.5pt,-3.5pt)--++(-3.5pt,-3.5pt)--++(-3.5pt,3.5pt);
\draw [fill=black] (-1.02,3.52) circle (2.5pt);
\draw [fill=black] (-2.5,0.5) circle (2.5pt);
\draw [fill=black] (-2.88,-0.14) circle (2.5pt);
\draw [fill=black] (-2.04,0.38) circle (2.5pt);
\draw [fill=black] (-2.02,-0.34) circle (2.5pt);
\draw [fill=black] (0.1,0.84) circle (2.5pt);
\draw [fill=black] (0.13901111225519902,-0.08257277074754431) circle (2.5pt);
\draw [fill=black] (0.17377014694634793,-0.7429944298793733) circle (2.5pt);
\draw [fill=black] (0.54,0.12) circle (2.5pt);
\draw [fill=black] (0.92,-0.57) circle (2.5pt);
\draw [fill=black] (-2.08,1.18) circle (2.5pt);
\draw [fill=black] (-2.5,0.5) circle (2.5pt);
\draw [fill=black] (-2.04,0.38) circle (2.5pt);
\draw [fill=black] (-1.3269625666099947,2.8771470259876106) circle (2.5pt);
\draw [fill=black] (-1.6,2.26) circle (2.5pt);
\draw [fill=black] (-1.8411290764915043,1.7323989473834942) circle (2.5pt);
\draw [fill=black] (-0.8127960567284852,2.9935620848287074) circle (2.5pt);
\draw [fill=black] (-0.5799659390462922,2.4696943200437724) circle (2.5pt);
\draw [fill=black] (-0.3568370762675239,1.9458265552588379) circle (2.5pt);
\draw [fill=black] (-0.12400695858533095,1.3831537708602044) circle (2.5pt);
\draw [fill=black] (-6.224837209302326,2.912883720930232) circle (2.5pt);
\draw [fill=black] (-7.264837209302326,0.9228837209302319) ++(-3.5pt,0 pt) -- ++(3.5pt,3.5pt)--++(3.5pt,-3.5pt)--++(-3.5pt,-3.5pt)--++(-3.5pt,3.5pt);
\draw [fill=black] (-5.304837209302326,0.9428837209302319) ++(-3.5pt,0 pt) -- ++(3.5pt,3.5pt)--++(3.5pt,-3.5pt)--++(-3.5pt,-3.5pt)--++(-3.5pt,3.5pt);
\draw [fill=black] (-0.3128563387297367,0.07384288536262569) circle (2.5pt);
\draw [fill=black] (-0.6952057203323747,-0.6039582911147776) circle (2.5pt);
\end{scriptsize}
\end{tikzpicture} 
\caption{Different representations for $T$ as $T=v+ P_{4} \ast  S_{2} \oplus  P_{5} \ast  S_{3}$ or $T=u+ P_{6} \ast  S_{2} \oplus  P_{3} \ast  S_{3}$ and $4+5=6+3$.}\label{fig:diff_rep}
\end{figure}
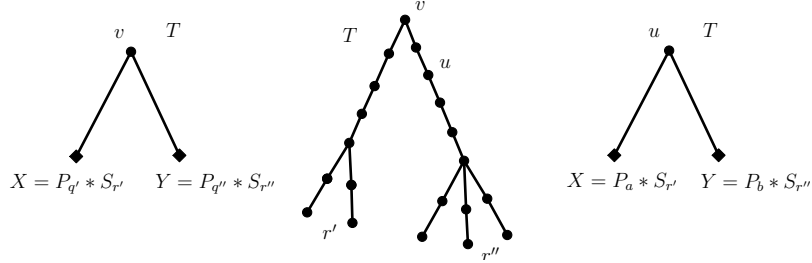

Given a tree in $(P_q, S_r)$ representation, recall that a starlike vertex is
a vertex $u$ of degree $\geq 3$ having at least two generalized pendant paths
$P_q\ast S_r$. In this section we deal with trees that have
fewer than 2 starlike
vertices. We always assume that $n= 4 \lfloor \frac{n}{4}\rfloor + \alpha$
and $\alpha=0,1,2,3$. Let $k \geq 0$ be the number of starlike vertices in
this representation. We first handle the case $k = 0$.

\begin{lemma}
\label{lem:no-star}
$T$ has a starlike vertex iff it has a vertex $v$,
$\deg(v) \geq 3$, that is not a part of any {\gpp}.
\end{lemma}
\begin{proof}

The `only if' is trivial. Conversely, assume $T$ has a vertex $v$, $\deg(v)
\geq 3$, not part of a {\gpp}. Note that trees in $(P_q, S_r)$ representation
have squares for leaves, and pendant paths are represented by each {\gpp},
and appear to have length one, although these paths represent a path $P_q$.
Indeed this property exists at initialization and is maintained by
\texttt{ReduceStarVertex}. Consider the tree $\tilde{T}$ of $T$ obtained by
converting the square leaves to normal vertices and removing all labels.
Since $q + r \geq 1$, $\tilde{T}$ is a subtree of $T$, and its pendant paths
of length one correspond to {\gpps} in $T$. This tree also has a vertex of
degree greater than $2$, so it is not a path. By Lemma~\ref{lemma-starlike}, it has
a vertex $u$ having degree at least $3$ with at least two pendant paths. This
must be a starlike vertex in $T$.
\end{proof}

From Lemma~\ref{lem:no-star} if $T$ does {\em not} have a starlike vertex,
then it is a path with a {\gpp} on each end, and we may write
$$
T= u+ P(u)= u + P_{q_1} \ast S_{r_1}\oplus P_{q_2} \ast S_{r_2} .
$$
An illustration of a general form of this tree is given in Figure \ref{fig:k0}.
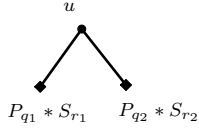
\begin{figure}[!ht]
  \centering
\begin{tikzpicture}[line cap=round,line join=round,>=triangle 45,x=1cm,y=1cm,scale=0.6, every node/.style={scale=0.9}]
\clip(-18.7,-0.2) rectangle (7.7,3);
\draw [line width=1pt] (-4.96,1.99)-- (-5.88,0.71);
\draw [line width=1pt] (-4.96,1.99)-- (-3.98,0.75);
\draw (-5.58,2.8) node[anchor=north west] {{\tiny $u$}};
\draw (-6.82,0.54) node[anchor=north west] {{\tiny $P_{q_1} \ast S_{r_1}$}};
\draw (-4.36,0.56) node[anchor=north west] {{\tiny $P_{q_2} \ast S_{r_2}$}};
\begin{scriptsize}
\draw [fill=black] (-4.96,1.99) circle (2.5pt);
\draw [fill=black] (-5.88,0.71) ++(-3.5pt,0 pt) -- ++(3.5pt,3.5pt)--++(3.5pt,-3.5pt)--++(-3.5pt,-3.5pt)--++(-3.5pt,3.5pt);
\draw [fill=black] (-3.98,0.75) ++(-3.5pt,0 pt) -- ++(3.5pt,3.5pt)--++(3.5pt,-3.5pt)--++(-3.5pt,-3.5pt)--++(-3.5pt,3.5pt);
\end{scriptsize}
\end{tikzpicture} 
\caption{The case $k=0$ starlike vertices.}\label{fig:k0}
\end{figure}
\begin{theorem}
\label{th:path} Let $T$ be a tree of order $n$  having no starlike vertices.
Then $T$ can be
properly transformed into $T_\alpha, ~\alpha \in \{0,1,2,3\}$
according to $n \equiv \alpha \pmod 4$.
\end{theorem}

\begin{proof} By Lemma~\ref{lem:no-star}, $T$ may written as
$T= u+ P(u)= u + P_{q_1} \ast S_{r_1}\oplus P_{q_2} \ast S_{r_2}$.
Assuming that $n  = 4r + \alpha$ and $r= \lfloor \frac{n}{4}
\rfloor$, we divide our proof in two cases.

\noindent{\bf Case 1: $r_1, r_2 \geq r$.}
It follows that $w(P_{q_1} \ast S_{r_1}\oplus P_{q_2} \ast S_{r_2}) + 1 = n = 4r + \alpha$ or
$$
q_1+2r_1+q_2+2r_2+1 = 4r + \alpha .
$$
Using the assumption of Case~1,
we conclude that $q_1+q_2 \leq \alpha -1$.
Let us analyze  the possibilities  of $\alpha$.

Note that $\alpha = 0$ is impossible because $q_1+q_2 \geq 0$.
When $\alpha =1$, we must have $q_1=q_2=0$, and then $T$ is already $T_1$.

Consider next $\alpha =2$. This implies that $q_1+q_2 \leq 1$.
A possible solution is $q_1=q_2=0$.
This is not feasible, as otherwise
$1 + 2r_1+2r_2 = n \equiv 1 \pmod 2$
but $n$ is even when $\alpha = 2$.
The other possible
solution is $q_1=0$ and $q_2=1$, leading to
$  1 + 1 + 2r_1 + 2r_2 = n = 4r +\alpha=4r+2$,
or $2r = r_1 + r_2$.
As $r_1, r_2 \geq r$, we have
$r_1=r_2=r$. We conclude that $T=u+ P_{0} \ast S_{r}\oplus P_{1} \ast S_{r} =
T_2$, as we claimed. We notice that $q_1 =1$ and $q_2 =0$ also leads to $T_2$
by symmetry.

Finally consider
\textbf{$\alpha =3$}.
This implies that $q_1+q_2\leq 2$ and the
possibilities are
\begin{itemize}
\item[(i)] $q_1=q_2 =0$,
\item [(ii)] $q_1=q_2=1$,
\item [(iii)] $q_1=0$ and $q_2=1$,
\item [(iv)] $q_1 = 0$ and $q_2 =2$.
\end{itemize}
Because
$1 + 1 + 2r_1  + 2r_2 = n$
is even
we see that (iii) is not feasible.
If  (ii) happens, we are faced to $1 +
q_1 + 2r_1 + q_2 + 2r_2 = 3 + 2r_1 + 2r_2 = n = 4r+ 3$,
or $r_1+r_2 = 4r$.
And as before, $r_1=r_2=2r$ or
$T= u + P_{1} \ast S_{r}\oplus P_{1} \ast S_{r} = T_3$,
as claimed.
If (iv) happens, a counting argument similar
to (ii) shows that
$T= u + P_{0} \ast S_{r}\oplus P_{2} \ast S_{r}$.
If we root $T$ at the middle vertex of $P_2$ rather than $u$
it has the form $P_3$.

We consider now the case (i)
$q_1=q_2=0$. As $n=4r + 3 = 1 + 2r_1 + 2r_2$, it
implies $r_1 + r_2 = 2r+1$, $r_1=r$ and  $r_2=r+1$ because $r_1, r_2 \geq r$.
We consider the transformation described in Figure~\ref{T_star_T3} which
transforms a tree
$T_{*}=v + P_{0} \ast S_{r} \oplus P_{0} \ast S_{r+1}$ with
$n=4r+3$ vertices into the tree
$T_{3}=u + P_{1} \ast S_{r} \oplus P_{1} \ast S_{r}$.
Using the Diagonalize algorithm in both trees we have the knowledge
of all signs except that of $v$ in $T_{*}$
(we know the signs in $T_{3}$ from Theorem~\ref{thr:Talpha}).\\
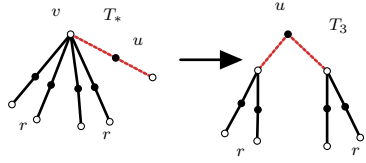
\begin{figure}[!ht]
  \centering
  \definecolor{dtsfsf}{rgb}{0.8274509803921568,0.1843137254901961,0.1843137254901961}
\definecolor{ffffff}{rgb}{1,1,1}
\begin{tikzpicture}[line cap=round,line join=round,>=triangle 45,x=1cm,y=1cm, scale=0.5, every node/.style={scale=0.8}]
\clip(-9.7,-1.5) rectangle (9.7,3);
\draw [line width=1pt] (-0.28,0.14)-- (-0.72,-0.66);
\draw [line width=1pt] (0.16,-0.03)-- (0.14,-0.86);
\draw [line width=1pt,dash pattern=on 1pt off 1pt,color=dtsfsf] (0.96,1.98)-- (2,1);
\draw [line width=1pt] (2,1)-- (2,0);
\draw [line width=1pt] (2,0)-- (2,-1);
\draw [line width=1pt] (2,1)-- (2.48,0.05);
\draw [line width=1pt] (2.48,0.05)-- (2.88,-0.77);
\draw (-3.4,2.12) node[anchor=north west] {{\tiny $u$}};
\draw (0.38,3.04) node[anchor=north west] {{\tiny $u$}};
\draw (1.82,2.62) node[anchor=north west] {{\tiny $T_3$}};
\draw (-0.64,-0.84) node[anchor=north west] {{\tiny $r$}};
\draw (2.42,-0.82) node[anchor=north west] {{\tiny $r$}};
\draw [line width=1pt,dash pattern=on 1pt off 1pt,color=dtsfsf] (0.96,1.98)-- (0.16,1.02);
\draw [line width=1pt] (-0.28,0.14)-- (0.16,1.02);
\draw [line width=1pt] (0.16,1.02)-- (0.16,-0.03);
\draw [line width=1pt] (-5.76,0.85)-- (-6.38,0.11);
\draw [line width=1pt] (-5.34,0.63)-- (-5.72,-0.29);
\draw [line width=1pt] (-4.18,0.57)-- (-3.78,-0.35);
\draw [line width=1pt,dash pattern=on 1pt off 1pt,color=dtsfsf] (-3.62,1.35)-- (-2.64,0.83);
\draw (-4.18,2.9) node[anchor=north west] {{\tiny $T_*$}};
\draw (-6.42,-0.14) node[anchor=north west] {{\tiny $r$}};
\draw (-4.2,-0.38) node[anchor=north west] {{\tiny $r$}};
\draw [line width=1pt] (-4.82,1.98)-- (-5.76,0.85);
\draw [line width=1pt] (-4.82,1.98)-- (-5.34,0.63);
\draw [line width=1pt] (-4.82,1.98)-- (-4.18,0.57);
\draw [line width=1pt,dash pattern=on 1pt off 1pt,color=dtsfsf] (-4.82,1.98)-- (-3.62,1.35);
\draw [line width=1pt] (-4.82,1.98)-- (-4.62,0.53);
\draw [line width=1pt] (-4.62,0.53)-- (-4.54,-0.47);
\draw [->,line width=1pt] (-1.9,1.3) -- (-0.24,1.3);
\draw (-5.56,2.9) node[anchor=north west] {{\tiny $v$}};
\begin{scriptsize}
\draw [fill=black] (0.96,1.98) circle (2.5pt);
\draw [fill=black] (-0.28,0.14) circle (2.5pt);
\draw [fill=ffffff] (-0.72,-0.66) circle (2.5pt);
\draw [fill=black] (0.16,-0.03) circle (2.5pt);
\draw [fill=ffffff] (0.14,-0.86) circle (2.5pt);
\draw [fill=ffffff] (2,1) circle (2.5pt);
\draw [fill=black] (2,0) circle (2.5pt);
\draw [fill=ffffff] (2,-1) circle (2.5pt);
\draw [fill=black] (2.48,0.05) circle (2.5pt);
\draw [fill=ffffff] (2.88,-0.77) circle (2.5pt);
\draw [fill=ffffff] (0.16,1.02) circle (2.5pt);
\draw [fill=ffffff] (-4.82,1.98) circle (2.5pt);
\draw [fill=black] (-5.76,0.85) circle (2.5pt);
\draw [fill=ffffff] (-6.38,0.11) circle (2.5pt);
\draw [fill=black] (-5.34,0.63) circle (2.5pt);
\draw [fill=ffffff] (-5.72,-0.29) circle (2.5pt);
\draw [fill=black] (-4.18,0.57) circle (2.5pt);
\draw [fill=ffffff] (-3.78,-0.35) circle (2.5pt);
\draw [fill=black] (-3.62,1.35) circle (2.5pt);
\draw [fill=ffffff] (-2.64,0.83) circle (2.5pt);
\draw [fill=black] (-4.62,0.53) circle (2.5pt);
\draw [fill=ffffff] (-4.54,-0.47) circle (2.5pt);
\end{scriptsize}
\end{tikzpicture} 
  \caption{On the left the tree $T_{*}$ on the right side  the tree $T_{3}$.}
  \label{T_star_T3}
\end{figure}
Let us call $f_{T_{*}}(v)$ the value obtained by the Diagonalize algorithm
$$f_{T_{*}}(v)=(2r+1) -(2-\frac{2}{n}) -(2r+1) \frac{1}{x_{2}}=-2+\frac{2}{n}+ (2r+1)\left( 1 - \frac{1}{x_{2}}\right)$$
where ${\rm deg}_{T_{*}}(v)=2r+1$.
Analyzing the correspondence $r \to f_{T_{*}}(v)$,
using Maple, we
expressed $x_2$ in terms of $r$,
obtaining rational function
$$
-4\,{\frac {24\,{r}^{2}+22\,r+5}{ \left( 4\,r+3 \right)  \left( 16\,{r
}^{2}+32\,r+11 \right) }}
$$
which is always negative for $r \geq 2$, thus $f_{T_{*}}(v)<0$ and we paint this
vertex as white.\\
Counting the black vertices in $T_{*}$ we see that $\sigma(T_{*})= 2r+1$ and we
already know that $\sigma(T_{3})= 2r+1$. Therefore this transformation is proper.\\

\noindent{\bf Case 2: $r_1 < r$ or $r_2 < r$.}  Let us assume, say, that $r_1
< r$. If $q_1+q_2 \geq 2$ we can rewrite $T$ as
$$
v + P_{q_1+q_2} \ast S_{r_1}\oplus P_{0} \ast S_{r_2} .
$$
Then , using a Star-up Proposition \ref{star-up}, we may transform $T
\to v + P_{q_1+q_2 - 2} \ast S_{r_1+1}\oplus P_{0} \ast S_{r_2}$. More
generally,  performing several Star-up transformations we may assume that $T$
is transformed into
$$
T^\prime = v+ P_{q^\prime_1} \ast S_{r^\prime_1}\oplus P_{q^\prime_2} \ast S_{r^\prime_2},
$$
with either
\begin{itemize}
\item [(a)] $r^\prime_1, r^\prime_2 \geq r$ or
\item [(b)] ~$r^\prime_1 < r$ or $r^\prime_2<r$ and $q^\prime_1+q^\prime_2 \leq 1$.
\end{itemize}
If (a) happens the result follows by the previous Case 1.
We may suppose then $r^\prime_1 < r$ and $q^\prime_1+q^\prime_2 \leq 1$. The possibilities
are the following.

\noindent $ \bullet~ q^\prime_1=q^\prime_2 =0$, which corresponds to
$T^\prime = v+ P_{0} \ast S_{r^\prime_1}\oplus P_{0} \ast S_{r^\prime_2}$ and
$n = 1 + 2r^\prime_1 + 2 r^\prime_2$ is odd.
Thus $n=4r+1$ or
$n=4r+3$, corresponding to $\alpha = 1$ or $\alpha = 3$, respectively.
If $\alpha=1$, then
$4r+1 = 1 + 2r^\prime_1 + 2 r^\prime_2$ which is equivalent to $r^\prime_1 +
r^\prime_2=2r$.
A Star-star (Proposition \ref{star-star})transformation produces the tree $v + P_{0} \ast
S_{r}\oplus P_{0} \ast S_{r} = T_1$.
If $\alpha=3$, then  $4r+3 = 1 +
2r^\prime_1 + 2 r^\prime_2$ which is equivalent to $r^\prime_1 +
r^\prime_2=2r+1$. A Star-star transformation produces a tree $v + P_{0} \ast
S_{r}\oplus P_{0} \ast S_{r+1} = T_*$. As we saw
in Case 1 above,
this tree can be
properly transformed into $T_3$.

\noindent $\bullet~ q^\prime_1=0$ and $q^\prime_2 =1$, which corresponds to
$T^\prime = v+ P_{0} \ast S_{r^\prime_1}\oplus P_{1} \ast S_{r^\prime_2}$ and
so $n = 1 + 2r^\prime_1 + 1+ 2 r^\prime_2= 2+ 2r^\prime_1 + 2 r^\prime_2$
is even,
implying  $n=4r$ or $n=4r+2$, which corresponds to $\alpha =0$ or
$\alpha=2$, respectively.
If $\alpha=0$, then  $4r = 2 + 2r^\prime_1 + 2 r^\prime_2$ which
is equivalent to $r^\prime_1 + r^\prime_2=2r-1$.
Using the $0-1$ Star-star transformation (Proposition \ref{star-star}) we can obtain
$T_0 = v + P_{0} \ast S_{r-1}\oplus P_{1} \ast S_{r}$.
If $\alpha=2$, then  $4r+2 = 1 + 2r^\prime_1 +1+ 2 r^\prime_2$ which is
equivalent to $r^\prime_1 + r^\prime_2=2r$.
Using the $0-1$ Star-star transformation we can obtain
$T_2 = v + P_{0} \ast S_{r}\oplus P_{1} \ast S_{r}$.
\end{proof}

We now turn our attention to a tree $T$ with a single starlike vertex.
By Lemma~\ref{lem:no-star} there is exactly
one vertex of degree $\geq 3$ not in a \gpp.
Let us nominate $u$ to be this starlike vertex, then $T$ has the form
\begin{equation}
\label{eq:twostars}
T=u+P(u)=u+P_{q_1}\ast S_{r_1} \oplus \cdots \oplus P_{q_L}\ast S_{r_L},
\end{equation}
for $L \geq 3$.
An example
is given in Figure
\ref{fig:k1} (left).
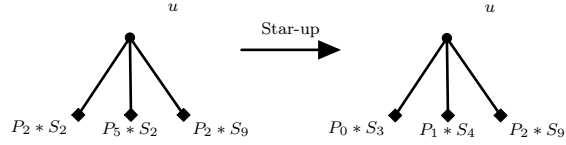
\begin{figure}[!ht]
\centering
\definecolor{aqaqaq}{rgb}{0.6274509803921569,0.6274509803921569,0.6274509803921569}
\begin{tikzpicture}[line cap=round,line join=round,>=triangle 45,x=1cm,y=1cm,scale=0.7, every node/.style={scale=0.8}]
\clip(-14.286619453801121,-0.8) rectangle (10.83924646512042,2.172994011716072);
\draw [line width=1pt] (-6.02,1.47)-- (-7,0);
\draw [line width=1pt] (-6.02,1.47)-- (-5,0);
\draw (-5.468091775382867,2.280649574588768) node[anchor=north west] {{\tiny $u$}};
\draw (-8.444835961429378,0.06624231423709696) node[anchor=north west] {{\tiny $P_{2} \ast S_{2}$}};
\draw (-4.996168916619395,0.029940555870676133) node[anchor=north west] {{\tiny $P_{2} \ast S_{9}$}};
\draw (-6.720502439024387,0.029940555870676133) node[anchor=north west] {{\tiny $P_{5} \ast S_{2}$}};
\draw [line width=1pt] (-6.02,1.47)-- (-6,0);
\draw [line width=1pt] (0,1.46)-- (-0.98,-0.01);
\draw [line width=1pt] (0,1.46)-- (1.02,-0.01);
\draw (0.5580001134429987,2.2624986954055575) node[anchor=north west] {{\tiny $u$}};
\draw (-2.418744072603513,0.04809143505388655) node[anchor=north west] {{\tiny $P_{0} \ast S_{3}$}};
\draw (1.0117720930232597,0.011789676687465715) node[anchor=north west] {{\tiny $P_{2} \ast S_{9}$}};
\draw (-0.6944105501985215,0.011789676687465715) node[anchor=north west] {{\tiny $P_{1} \ast S_{4}$}};
\draw [line width=1pt] (0,1.46)-- (0.02,-0.01);
\draw [->,line width=1pt] (-3.9,1.236) -- (-2.01,1.236);
\draw (-3.689305615428244,1.93) node[anchor=north west] {{\tiny Star-up}};
\begin{scriptsize}
\draw [fill=black] (-6.02,1.47) circle (2.5pt);
\draw [fill=black] (-7,0) ++(-3.5pt,0 pt) -- ++(3.5pt,3.5pt)--++(3.5pt,-3.5pt)--++(-3.5pt,-3.5pt)--++(-3.5pt,3.5pt);
\draw [fill=black] (-5,0) ++(-3.5pt,0 pt) -- ++(3.5pt,3.5pt)--++(3.5pt,-3.5pt)--++(-3.5pt,-3.5pt)--++(-3.5pt,3.5pt);
\draw [fill=black] (-6,0) ++(-3.5pt,0 pt) -- ++(3.5pt,3.5pt)--++(3.5pt,-3.5pt)--++(-3.5pt,-3.5pt)--++(-3.5pt,3.5pt);
\draw [fill=black] (0,1.46) circle (2.5pt);
\draw [fill=black] (-0.98,-0.01) ++(-3.5pt,0 pt) -- ++(3.5pt,3.5pt)--++(3.5pt,-3.5pt)--++(-3.5pt,-3.5pt)--++(-3.5pt,3.5pt);
\draw [fill=black] (1.02,-0.01) ++(-3.5pt,0 pt) -- ++(3.5pt,3.5pt)--++(3.5pt,-3.5pt)--++(-3.5pt,-3.5pt)--++(-3.5pt,3.5pt);
\draw [fill=black] (0.02,-0.01) ++(-3.5pt,0 pt) -- ++(3.5pt,3.5pt)--++(3.5pt,-3.5pt)--++(-3.5pt,-3.5pt)--++(-3.5pt,3.5pt);
\end{scriptsize}
\end{tikzpicture}
\caption{The case $k=1$ starlike vertex. In this example $n=36
\equiv 0 \pmod 4$ and $r=\lfloor\frac{n}{4}\rfloor=8$
therefore we can not perform  Star-up (Proposition \ref{star-up}) on the gpp $P_2\ast S_9$.}
\label{fig:k1}
\end{figure}

\begin{theorem}
\label{th:starlike} Let $T$ be a tree of order $n$  having a single starlike
vertex. Then $T$ can be properly transformed into $T_\alpha, ~\alpha \in
\{0,1,2,3\}$ according to $n \equiv \alpha \pmod 4$.
\end{theorem}

\begin{proof}
We write $n=4r+\alpha$. For each gpp $P_{q_i} \ast S_{r_i}$
in (\ref{eq:twostars})
we consider three
possibilities. Either (i) $q_i \geq 2$ and $r_i \leq r-1$ or (ii) $q_i \in
\{0, 1\}$ or (iii) $r_i \geq r$.
If (i) happens, then we can execute some
Star-up (Proposition \ref{star-up}) transformations to shorten the path $P_{q_i}$ as much as possible.

In this proof we will use
\texttt{ReduceStarVertex} to simply combine several gpps into a single gpp,
but we will {\em not} execute the if-statement inside it.

Without loss of generality assume the \gpps\ are ordered in non-decreasing $r_i$.
Since $L \geq 3$ there are three positive indices in $\{L-2, L-1, L\}$.
We claim
that there exists an index
$l_0 \in \{L-2, L-1, L\}$
such that
$ r_i \leq r-1$ if and only if $i \leq l_0$.
Moreover
$q_i \in \{ 0, 1\}$ when $i \leq l_0$.
To see this, we observe that if $r_i \geq r$ for $L-2, L-1$ and
$L$, we would have $n\geq 6r$, a contradiction.
And if for some $i \leq l_0$
$q_i \geq 2$ we could use Star-up transform contradicting the fact that we
have made these transformations as much as possible. Let us analyze  then the
possible values of $l_0$.

\noindent{\bf Case 1: $l_0 = L-2$.} Here $r_{L-1}, r_L \geq r$, meaning that
these two gpps have weight at least $4r$. In fact, we can see that if $r_{L-1},r_{L} \not = r$, then the degree of $v$ could be less than 3, contradicting the fact that it is a starlike vertex.
Taking into account the root $u$,
there are at least $4r+1$ vertices.
This means that $\alpha =0$ is impossible. If $\alpha =1$, we would have
$n=4r + 1$, leaving nothing for the first $L-2$ gpps, an impossibility.

If $\alpha = 2$, we have $n=4r+2$ and there is a single vertex to be distributed to the other gpps.
This implies that
$T=u + P_1 \ast S_0 \oplus P_0 \ast S_r \oplus P_0 \ast S_r$.
Let $v$ be the vertex in $P_1$ having degree one.
By using Star-down (Proposition \ref{star-down transformation})
$r$ times, we can move $r$ of the $P_2$'s to $v$,
and the tree is properly
transformed into $u + P_1\ast S_r \oplus P_0 \ast S_r = T_2$, as we want.

Finally, if $\alpha = 3$, we have $n= 4r + 3$, which gives more freedom for
the choice of the particular tree.
We have
\begin{eqnarray*}
n= 4r + 3 & = & 1 + w(P_{q_1}\ast S_{r_1} \oplus \ldots \oplus
P_{q_{l_0}}\ast S_{r_{l_0}}) + q_{L-1} + 2r_{L-1} + q_{L} + 2r_{L} \\
& \geq & 1 + 1 + q_{L-1} + 2r_{L-1} + q_{L} + 2r_{L} \\
& = & 2 + q_{L-1} + q_{L} + 4r,
\end{eqnarray*}
implying $ 1 \geq q_{L-1}  + q_{L}$.
The possible cases choices for $q_{L-1}$
and $q_L$ are $q_{L-1}=q_{L}=0$, implying
\begin{eqnarray*}
n= 4r + 3 & = & 1 + w(P_{q_1}\ast S_{r_1} \oplus \ldots \oplus P_{q_{l_0}}\ast S_{r_{l_0}}) + 2(r_{L-1} + r_L) \\
& = & 1 + w(P_{q_1}\ast S_{r_1} \oplus \ldots \oplus P_{q_{l_0}}\ast S_{r_{l_0}}) + 4r,
\end{eqnarray*}
meaning that $w(P_{q_1}\ast S_{r_1} \oplus \ldots \oplus P_{q_{l_0}}\ast
S_{r_{l_0}}) = 2$, leading to the following possibilities:
\begin{itemize}
\item[(i)] $P_{q_1}\ast S_{r_1} \oplus \ldots \oplus P_{q_{l_0}}\ast S_{r_{l_0}} = $
$P_0 \ast S_{1}$
\item[(ii)] $P_{q_1}\ast S_{r_1}
\oplus \ldots \oplus P_{q_{l_0}}\ast S_{r_{l_0}} = P_1 \ast S_0 \oplus P_1 \ast S_0$.
\end{itemize}

If Subcase (i)  happens then the result follows from the analysis of the previous case,
namely using $r$ Star-down transformations we can achieve $T_3$.

If case (ii) happens, then $T = u + P_1 \ast S_0
\oplus P_1 \ast S_0 \oplus P_0 \ast S_{r_{L-1}} \oplus P_0 \ast S_{r_L}$ and
by counting the number of vertices
since
$r_{L-1}, r_L \geq r$,
we must have
$r_{L-1} = r_L = r$.
Now, for each pair  $P_1 \ast
S_{0}\oplus P_0 \ast S_{r}$ we apply Star-down transforming into $P_1 \ast
S_{r}$, transforming $T$ into $u + P_1 \ast S_{r} \oplus P_1 \ast S_{r}=T_3$,
as we wish.

If $q_{L-1} =1$ and $q_L=0$,  proceeding as before, we obtain that the only
possibility for  $P_{q_1}\ast S_{r_1} \oplus \ldots \oplus P_{q_{l_0}}\ast
S_{r_{l_0}}$ is  $P_1\ast S_0$.
This means that
$T= u + P_1\ast S_0 \oplus P_0 \ast S_{r_{L-1}} \oplus P_1\ast S_{r_L}$.
As $r_{L-1} + r_{L}= 2r$ and $r_{L-1}, r_L \geq r$ and
we must have $r_{L-1}=r_L =r$.
Apply Star-star (Proposition \ref{star-star}) to the first two summands getting $P_1 * S_r$. Now $T = T_3$.

\noindent{\bf Case 2: $l_0 = L-1$.} Here only $ r_L \geq r$, meaning that the
gpp $P_{q_L} \ast S_{r_L}$ has weight at least $2r$. Taking into account the
root $u$, there are at least $2r+1$ vertices. The whole tree is $T= u +
P_{q_1}\ast S_{r_1} \oplus P_{q_2} \ast S_{r_2} \oplus \cdots \oplus P_{q_{L-1}} \ast S_{r_{L-1}} \oplus P_{q_L}\ast S_{r_L}$ whose weight is $n$. We analyze the possible values of $\alpha$.

If $\alpha =0$, then
\begin{eqnarray*}
n=4r & = & 1 + w(P_1\ast S_{r_1} \oplus P_{q_2} \ast
S_{r_2} \oplus \cdots \oplus P_{q_{L-1}} \ast S_{r_{L-1}}) + w(P_{q_L}\ast S_{r_L}) \\
& \geq & 1+w(P_1\ast S_{r_1} \oplus P_{q_2} \ast S_{r_2} \oplus \cdots \oplus
P_{q_{l_0}} \ast S_{r_{l_0}})+2r,
\end{eqnarray*}
meaning that
$$
w(P_1\ast S_{r_1} \oplus P_{q_2}
\ast S_{r_2} \oplus \cdots \oplus P_{q_{l_0}}\ast S_{r_{l_0}})\leq 2r -1 .
$$
We then may apply \texttt{ReduceStarVertex} transforming it into a single gpp
$P_{q^\prime}\ast S_{r^\prime}$ and
$$
T=u+P_{q^\prime}\ast S_{r^\prime} \oplus
 P_{q_L}\ast S_{r_L}.
$$
Now $T$ has no starlike vertices and then by Theorem
 \ref{th:path} we can transform $T$ into $T_0$.

If $\alpha= 1$, proceeding as in the previous paragraph, we see that
$$
w(P_1\ast S_{r_1} \oplus P_{q_2} \ast S_{r_2} \oplus \cdots \oplus
P_{q_{l_0}}\ast S_{r_{l_0}})\leq 2r.
$$
Hence we can still apply \texttt{ReduceStarVertex},
so that the tree has a single gpp.
Now the tree has no starlike vertices and
by Theorem \ref{th:path}, $T$ is properly transformed into $T_1$.

For $\alpha = 2$, the similar analysis allows one to conclude that
\begin{equation}
\label{Casetwo:inequal}
w(P_{q_1}\ast S_{r_1} \oplus P_{q_2} \ast S_{r_2} \oplus \cdots \oplus
P_{q_{l_0}} \ast S_{r_{l_0}})\leq 2r+1.
\end{equation}
If the left of (\ref{Casetwo:inequal}) is $ \leq 2r$,
\texttt{ReduceStarVertex} can be
applied and, as before, Theorem~\ref{th:path} shows that $T$ is properly transformed into $T_2$.
Therefore we may assume
equality occurs in (\ref{Casetwo:inequal}).
This means that $q_L=0$ and $r_L= r$ and so
$$
T = u + (P_{q_1}\ast S_{r_1} \oplus P_{q_2} \ast S_{r_2} \oplus \cdots
\oplus P_{q_{l_0}}\ast S_{r_{l_0}}) \oplus P_0 \ast S_r .
$$
As $L \geq 3$, we see
that $l_0 \geq 2$. Moreover, there exists at least an index $i \in \{1,
\ldots, l_0\}$ such that $q_i$ is 1 for, otherwise $n = 4r+2$ would be odd.
Consider the \gpps
$$
(P_{q_1}\ast S_{r_1}
\oplus P_{q_2} \ast S_{r_2} \oplus \cdots \oplus P_{q_{l_0}}\ast S_{r_{l_0}})
- P_{q_i}\ast S_{r_i}.
$$
Since their weight $ \leq 2r$,
we can apply \texttt{ReduceStarVertex}
transforming $T$ into
$$
u + P_{q_i} \ast S_{r_i} \oplus P_{q^\prime} \ast S_{r^\prime} \oplus P_0 \ast S_r.
$$

As $w(P_{q_i} \ast S_{r_i} \oplus P_{q^\prime} \ast S_{r^\prime}) = 2r +1$,
it follows that
$r_i +  r^\prime = r$.
Since $q_i = 1$, by parity we must have $q'=0$. Therefore we
apply Star-star transformations (Proposition \ref{star-star}) to convert $P_1 \ast S_{r_i} \oplus P_{0} \ast S_{r^\prime}$ into
$P_1 \ast S_r$ reducing $T$ to $ u + P_1 \ast S_r \oplus P_0 \ast S_r = T_2$.

For $\alpha =3$, we have
$n= 4r+3   $
\begin{eqnarray*}
& = & 1+
w(P_{q_1}\ast S_{r_1} \oplus P_{q_2} \ast S_{r_2} \oplus \cdots \oplus P_{q_{l_0}} \ast S_{r_{l_0}})   +q_L + 2r_L \\
& \geq &
w(P_{q_1}\ast S_{r_1} \oplus P_{q_2} \ast S_{r_2} \oplus \cdots \oplus P_{q_{l_0}} \ast S_{r_{l_0}}) +  2r +1,
\end{eqnarray*}
and, consequently,
$$
w(P_{q_1}\ast S_{r_1} \oplus P_{q_2} \ast S_{r_2}
\oplus \cdots \oplus P_{q_{l_0}} \ast S_{r_{l_0}})\leq 2r +2.
$$
Now, we look at
$
P_{q_1}\ast S_{r_1} \oplus P_{q_2} \ast S_{r_2} \oplus \cdots \oplus P_{q_{l_0}} \ast S_{r_{l_0}}$
and claim that there exists at least one gpp
$ P_{q_j} \ast S_{r_j}$ with $r_j \geq 1$ for, otherwise all the gpps
have the form $P_1 \ast S_0$ and then we can take two of them and transform into $P_0
\ast S_1$, by Star-star transformation.
Consider the tree
 $$
  u +
P_{q_1}\ast
 S_{r_1} \oplus P_{q_2} \ast S_{r_2} \oplus \cdots \oplus P_{q_{j-1}} \ast
 S_{r_{j-1}} \oplus P_{q_{j+1}} \ast S_{r_{j+1}}\oplus \cdots \oplus
 P_{q_{l_0}} \ast S_{r_{l_0}}.
 $$
Its weight of the last tree is $\leq 2r$. Hence we
can apply \texttt{ReduceStarVertex} transforming it into a single gpp
$P_{q^{\prime}} \ast S_{r^{\prime}}$ and
$T = u + P_{q_{j}} \ast S_{r_{j}}\oplus P_{q^{\prime}} \ast S_{r^{\prime}} \oplus P_{q_L} \ast S_{r_L}$.

If $q_L \geq 2$, it is easy to show that $w(P_{q_{j}} \ast S_{r_{j}}\oplus
P_{q^{\prime}} \ast S_{r^{\prime}}) \leq 2r$ and then, by applying
\texttt{ReduceStarVertex}, we transform it into a gpp and
$T= u + P_{q^{\prime\prime}} \ast S_{r^{\prime\prime}} \oplus P_{q_L} \ast S_{r_L}$,
which by Theorem \ref{th:path} may be transformed into $T_3$.

If $q_L=1$ then
$w(P_{q_{j}} \ast S_{r_{j}}\oplus P_{q^{\prime}} \ast S_{r^{\prime}}) \leq
2r+1$ and, as before, we only need to consider
when
$w(P_{q_{j}} \ast S_{r_{j}}\oplus P_{q^{\prime}} \ast S_{r^{\prime}}) = 2r+1$.
From this we
conclude that $r_L = r$ and
$T = u + P_{q_{j}} \ast S_{r_{j}}\oplus P_{q^{\prime}} \ast S_{r^{\prime}} \oplus P_{1} \ast S_{r}$.
Since
$q_j + 2r_j+ q^\prime + 2r^\prime = 2r +1$,
and
$q_j, q^\prime \in \{0,1\}$,
by parity
we conclude that $q_j + q^\prime =1$
and so $r_j + r^\prime =r$.
Thus we can apply Star-star (Proposition \ref{star-star})
transforming $ P_{q_{j}} \ast S_{r_{j}}\oplus P_{q^{\prime}} \ast S_{r^{\prime}}$
into $P_{1} \ast S_{r}$ and $T = u + P_{1} \ast S_{r}\oplus
P_{1} \ast S_{r} = T_3$.

If $q_L = 0$, then
$T = u + P_{q_{j}} \ast S_{r_{j}}\oplus P_{q^{\prime}} \ast S_{r^{\prime}} \oplus P_{0} \ast S_{r_L}$,
where $1 \leq r_j \leq r-1$, $r_L \geq r$, and
$q_j, q^\prime \in \{0,1\}$.
We transform this into
$u + P_{q_{j}} \ast S_{r }\oplus P_{q^{\prime}} \ast S_{r^{\prime}} \oplus P_{0} \ast S_{ r_L - (r-r_j) }$,
using Star-star.
Since $n$ is odd, either $q_j = q' = 0$  or
$q_j = q' = 1$.
In the first case $T = u + P_{0} \ast S_{r + r' + r_L - (r - r_j) } =  u + P_0 \ast  S_{r_j + r' + r_L  }$.
From $n = 4r + 3 = 1 + 2(r_j + r' + r_L)$, we deduce that  $r_j + r' + r_L = 2r + 1$,
so this is the tree $T_{*}$
in Figure~\ref{T_star_T3} that was transformed into $T_{3}$.
If $q_j = q' = 1$, the summands are
$P_{1} \ast S_{r }\oplus P_{1} \ast S_{r^{\prime}} \oplus P_{0} \ast S_{ r_L - (r-r_j) }$.
It is easy to see that
$r = r^{\prime} + r_L - (r-r_j)$, so
$P_{1} \ast S_{r^{\prime}} \oplus P_{0} \ast S_{ r_L - (r-r_j) } \Rightarrow P_1 \ast S_r$, obtaining $T_3$.

\noindent{\bf Case 3: $l_0 = L$.} In this case we have
$$
 T=u+ P(u)=u +P_{q_1}
\ast S_{r_1} \oplus \cdots \oplus P_{q_L}\ast S_{r_L},
$$
with
\begin{equation}
\label{condtionsCasethree}
q_i \in \{0, 1\} \mbox{ and } r_i \leq r-1, \mbox{ for all  }  i =1,\ldots, L .
\end{equation}
We claim that any tree having $L \geq 3$ gpps
and satisfying the conditions in
(\ref{condtionsCasethree})
can be reduced to a
prototype tree.

We may assume that $r_1
\leq r_2 \leq \cdots \leq r_L$. Consider the weight $w$ of $P_{q_1}\ast
S_{r_1} \oplus P_{q_2}\ast S_{r_2}$. If $w \leq 2r$, we can use
\texttt{ReduceStarVertex} and the 2 gpps into a single one and transforming
the original tree into $ T=u+P(u)=u+P_{q_1}\ast S_{r_1} \oplus \cdots \oplus
P_{q_{L-1}}\ast S_{r_{L-1}}$ having $L-1$ gpps. If $L-1 = 2$, then the tree
has no starlike vertices and the result follows from Theorem \ref{th:path}.
If there remains any pair of gpps whose weight is bounded by $2r$, this
argument can be repeated. Hence we can assume that the tree is of form
$T=u+P(u)=u+P_{q_1} \ast S_{r_1} \oplus \cdots \oplus P_{q_L}\ast S_{r_L}$,
with $L \geq 3, q_i \in \{0, 1\}$ and $w(P_{q_1}\ast S_{r_1} \oplus
P_{q_2}\ast S_{r_2}) \geq 2r +1$.
Since $r_1 + r_2$ is integer, it follows that $r_1 + r_2 \geq r$.
We also exploit the fact that $r_i \leq r - 1$.
We consider the possible values for
$\alpha$.\\

$\bullet~\alpha=0$. As $n = 4r = 1 +w(P_{q_1}\ast S_{r_1} \oplus P_{q_2}\ast
S_{r_2} ) + w(P_{q_3} \ast S_{r_3} \oplus \cdots \oplus P_{q_L}\ast S_{r_L})
\geq 1 + 2r+1 + w(P_{q_3} \ast S_{r_3} \oplus \cdots \oplus P_{q_L}\ast
S_{r_L})$, implying that $w(P_{q_3} \ast S_{r_3} \oplus \cdots \oplus
P_{q_L}\ast S_{r_L}) \leq 2r-2$, which means we can use the
\texttt{ReduceStarVertex} procedure to reduce it to a single gpp $
P_{q^\prime} \ast S_{r^{\prime}}$ and $T$ becomes $T=u+P_{q_1} \ast S_{r_1}
\oplus P_{q_2}\ast S_{r_2} \oplus  P_{q^\prime}\ast S_{r^\prime}$, with
$r^\prime \leq r-1, q^\prime + 2r^\prime \leq 2r -2$, and $q^\prime \in \{0,
1\}$. As $n=4r=1+q_1+q_2+q^\prime + 2r_1 + 2r_2 + 2r^\prime$, we see that
either (i) only $q_1=1$ or only $q_2=1$ or only $q^\prime =1$ or (ii)
$q_1=q_2=q^\prime = 1$. If (say) $q_2=1$ and $q_1=q^\prime =0$, we see that
$T$ has the form $T=u+P(u)=u+P_{0} \ast S_{r_1} \oplus P_{1}\ast S_{r_2}
\oplus P_{0}\ast S_{r^\prime}$. By star-down transform (Proposition \ref{star-down transformation}) we reduce it to
$u+P_{0} \ast S_{r_1-t} \oplus P_{1}\ast S_{r_2+t} \oplus  P_{0}\ast
S_{r^\prime}$, with $r_2+t=r$. If $r_1-t \geq 1$, we can use star-down
transform to pass the remaining $r_1 -t$ paths $P_2$ to $P_0 \ast
S_{r^\prime}$ to produce $u+P(u)=u+P_{1} \ast S_{r} \oplus  P_{0}\ast
S_{r^\prime+r_1-t}$, which has no starlike vertices and the result follows
from Theorem~\ref{th:path}. The case $q_1=1$ with $q_2=q^\prime=0$ and the
case $q^\prime =1$ with $q_1=0$ and $q_2=0$ are similar.

It remains to
analyse when $q_1=q_2=q^\prime=1$. This means that $T=u+P(u)=u+P_{1} \ast
S_{r_1} \oplus P_{1}\ast S_{r_2} \oplus P_{1}\ast S_{r^\prime}$. Since $r_1,
r_2 \leq r-1$, we use star-star transform to reduce  $T$ to $u+P(u)=u+P_{0}
\ast S_{r_1+r_2-r} \oplus P_{2}\ast S_{r} \oplus  P_{1}\ast S_{r^\prime}$.
Since $2(r_1 + r_2 - r + r^\prime) + 2r + 1 + 2 + 1 = n = 4r$,
we have $r_1+r_2-r + r^\prime < r$.

Therefore
we can use star-star to transform
$T$ into $u+P(u)=u+P_{1} \ast S_{r_1+r_2-r+r^\prime} \oplus P_{2}\ast S_{r}$,
which has no starlike vertices and then we can invoke Theorem \ref{th:path}
to claim the result.\\

$\bullet~\alpha=1$. As $n = 4r +1= 1 + w(P_{q_1}\ast S_{r_1} \oplus
P_{q_2}\ast S_{r_2}) + w(P_{q_3} \ast S_{r_3} \oplus \cdots \oplus
P_{q_L}\ast S_{r_L}) \geq 1 + 2r+1 + w(P_{q_3} \ast S_{r_3} \oplus \cdots
\oplus P_{q_L}\ast S_{r_L})$, implying that $w(P_{q_3} \ast S_{r_3} \oplus
\cdots \oplus P_{q_L}\ast S_{r_L}) \leq 2r-1$, which means we can use the
\texttt{ReduceStarVertex} procedure to reduce it to a single gpp $
P_{q^\prime} \ast S_{r^\prime}$ and since its weight is $\leq 2r$, we know
that $q^\prime \in \{0, 1\}$, $r^\prime \leq r-1$ and $T$ is transformed into
$T=P(u)=u+P_{q_1} \ast S_{r_1} \oplus P_{q_2}\ast S_{r_2} \oplus P_{q^\prime}
\ast S_{r^\prime}$. As $n=4r +1=1+q_1+q_2+q^\prime + 2r_1 + 2r_2 +
2r^\prime$ is odd, we see that either (i) $q_1=q_2=q^\prime = 0$ or (ii) two
$q$'s equal 1 and  the other 0. If $q_1=q_2=q^\prime = 0$, we see that the
tree has the form $T=P(u)=u+P_{0} \ast S_{r_1} \oplus P_{0}\ast S_{r_2}
\oplus P_{0}\ast S_{r^\prime}$. Since $n=4r+1= 2(r_1 + r_2 + r^\prime) +1$,
we use star-star (Proposition \ref{star-star}) to transform the tree into  $T=P(u)=u+P_{0} \ast S_{r}
\oplus P_{0}\ast S_{r}=T_1$, as claimed.

It remains to analyze when two $q$'s
are 1 and the other is zero, say $q_1=0$ and $q_2=q^\prime=1$.
This means that $T=P(u)=u+P_{0} \ast S_{r_1} \oplus P_{1}\ast S_{r_2} \oplus P_{1}\ast
S_{r^\prime}$.
Since $r_1 + r_2 \geq r$
and $r_1, r_2 \leq r-1$ we can use star-star transform
reducing it to $T=P(u)=u+P_{0} \ast S_{r_1+r_2-r} \oplus P_{1}\ast S_{r}
\oplus  P_{1}\ast S_{r^\prime}$. Now, since $n=4r+1= 1+ 2(r_1+r_2-r)+1 + 2r+
1+2r^\prime$, we conclude that $r-1 = (r_1+r_2-r)+r^\prime$, and we can use
star-star to transform $T$ into $T=P(u)=u+P_{1} \ast S_{r-1} \oplus P_{1}\ast
S_{r}$, which has no starlike vertices and then we can invoke Theorem
\ref{th:path} to claim the result. The other possible values of $q$ are
similar, concluding the case $\alpha=1$.\\

$\bullet~ \alpha=2$. As $n = 4r+2 = 1 + w(P_{q_1}\ast S_{r_1} \oplus
P_{q_2}\ast S_{r_2}) + w(P_{q_3} \ast S_{r_3} \oplus \cdots \oplus
P_{q_L}\ast S_{r_L}) \geq 1 + 2r+1 + w(P_{q_3} \ast S_{r_3} \oplus \cdots
\oplus P_{q_L}\ast S_{r_L})$, implying that $w(P_{q_3} \ast S_{r_3} \oplus
\cdots \oplus P_{q_L}\ast S_{r_L})\leq 2r$, which means we can use the
\texttt{ReduceStarVertex} procedure to reduce it to a single gpp $
P_{q^\prime} \ast S_{r^\prime}$ and since its weight is $\leq 2r$, we may
assume that $q^\prime \in \{0, 1\}$ and $T$ is transformed into
$u+P(u)=u+P_{q_1} \ast S_{r_1} \oplus P_{q_2}\ast S_{r_2} \oplus P_{q^\prime}
\ast S_{r^\prime}$. Because $n=4r +2=1+q_1+q_2+q^\prime + 2r_1 + 2r_2 +
2r^\prime$ is even, we see that either (i) One of the $q$'s is 1 the other
two are 0 or (ii) $q_1=q_2=q^\prime = 1$. Consider when (i) occurs, say
$q_1=1$ and $q_2=q^\prime=0$ (the other possibilities are similar). We see
that the tree has the form $T=u+P(u)=u+P_{1} \ast S_{r_1} \oplus P_{0}\ast
S_{r_2} \oplus P_{0}\ast S_{r^\prime}$. Since $1+ 2(r_1 + r_2)\geq 2r+1$,
we conclude that $r_1+r_2 \geq r$ but $r_1,r_2 \leq r$, then by star-star
transform we reduce $T$ to $u+P(u)=u+P_{0} \ast S_{r_1+r_2-r} \oplus
P_{1}\ast S_{r} \oplus  P_{0}\ast S_{r^\prime}$. We can use star-star to
transform $T$ into $u+P(u)=u+P_{0} \ast S_{r_1+r_2-r+r^\prime} \oplus
P_{1}\ast S_{r}$. Now, by counting, $2r = 2 (r_1+r_2-r + r^\prime)$, we see
that $r_1+r_2-r + r^\prime=r$ and  $T=T_2$.

If (ii) occurs, then
$T=u+P(u)=u+P_{1} \ast S_{r_1} \oplus P_{1} \ast S_{r_2} \oplus P_{1}\ast S_{r^\prime}$. If $r_1 + r_2 = r$ then using star-star
we obtain
$T= u + P_{2} \ast S_{r} \oplus P_{1}\ast S_{r^\prime}$,
which has no starlike vertices and we are through.
Otherwise $r_1+r_2 \geq r+1$,
and we first use
star-star to transform $T$ into $u+P(u)=u+P_{2} \ast S_{r} \oplus P_{0}\ast
S_{r_1+r_2-r} \oplus P_{1}\ast S_{r^\prime}$. By counting, we see that $r-1 =
r_1+r_2-r+r^\prime$, and we use again a star-star to transform $T$ into
$u+P(u)=u+P_{2} \ast S_{r} \oplus P_{1}\ast S_{r-1}$. We then change the root
and see $T$ as $T=v+P(v)=v +P_{0} \ast S_{r} \oplus P_{3}\ast S_{r-1}$, and
by executing a star-up (Proposition \ref{star-up}), $T$ is transformed into $v+P(v)=v +P_{0} \ast S_{r}
\oplus P_{1}\ast S_{r}=T_2$, as we claimed.\\

$\bullet~\alpha=3$. As $n = 4r+3 = 1 + w(P_{q_1}\ast S_{r_1} \oplus
P_{q_2}\ast S_{r_2}) + w(P_{q_3} \ast S_{r_3} \oplus \cdots \oplus
P_{q_L}\ast S_{r_L}) \geq 1 + 2r+1 + w(P_{q_3} \ast S_{r_3} \oplus \cdots
\oplus P_{q_L}\ast S_{r_L})$, we conclude that $w(P_{q_3} \ast S_{r_3} \oplus
\cdots \oplus P_{q_L}\ast S_{r_L})\leq 2r+1$.

We consider two cases.
If $L=3$ (only 3 gpp's), by assumption $r_3 \leq r-1$ and $q_3 \in \{0, 1\}$.
We now consider the two possibilities for $q_3$.
In case $q_3=0$ we have
$T=u + P_{q_1} \ast S_{r_1} \oplus P_{q_2} \ast S_{r_2} \oplus P_{0}\ast S_{r_3}$.
Counting the number of vertices, we have $4r+3 =1+ q_1+ q_2 + 2r_1+2r_2+ 2r_3$ or
$4r+2 = q_1+ q_2 + 2r_1+2r_2+
2r_3$, meaning that $q_1+ q_2$ is even.
If
$q_1= q_2=0$ then
$T=u + P_{0} \ast S_{r_1} \oplus P_{0} \ast S_{r_2} \oplus P_{0}\ast S_{r_3}$,
using star-star $T$ can be written as
$T=u + P_{0} \ast S_{r} \oplus P_{0} \ast S_{r+1}$.
This has no starlike vertices and we are done.
On the other hand, assume $q_1= q_2=1$,
and so $T=u + P_{1} \ast S_{r_1} \oplus P_{1} \ast S_{r_2} \oplus P_{0}\ast S_{r_3}$.
We distribute the $r_3$ $P_2$ paths as follows. Let $t=r-r_1$ and $s=r-r_2$. Then, using star-star, write this as $T=u + P_{1} \ast S_{r_1 + t} \oplus P_{1} \ast S_{r_2} \oplus P_{0}\ast S_{r_3 -t}$,
and then
$T=u + P_{1} \ast S_{r_1 + t} \oplus P_{1} \ast S_{r_2 + s} \oplus P_{0}\ast S_{r_3 -t -s}$.
Since $r_1 + r_2 + r_3 = 2r$, this last tree must be
$u + P_{1} \ast S_{r} \oplus P_{1} \ast S_{r} \oplus P_{0}\ast S_{0}$, or $T_3 = u + P_{1} \ast S_{r} \oplus P_{1} \ast S_{r}$.
We can handle the case $q_3=1$ in the same way
because, in this case, $4r+3 =1+ q_1+ q_2 + 1 + 2r_1+2r_2+ 2r_3$ or $4r+1 =
q_1+ q_2  + 2r_1+2r_2+ 2r_3$ then one should consider $q_1=0, q_2=1$, which
is a symmetric possibility.

If $L \geq 4$ then consider the ggp $P_{r_3} \ast S_{r_3}$. Since its weight is at least 1 (corresponding to $P_1 \ast S_0$)
we have $w(P_{q_4} \ast S_{r_4} \oplus \cdots \oplus P_{q_L}\ast S_{r_L})\leq
2r$.
We apply \texttt{ReduceStarVertex} procedure to reduce it to a single
gpp $ P_{q^\prime} \ast S_{r^\prime}$ with $q^\prime \in \{0, 1\}$ and
$r^\prime \leq r-1$.
Thus $T$ is transformed into $u+ P(u)=u+P_{q_1} \ast
S_{r_1} \oplus P_{q_2}\ast S_{r_2} \oplus  P_{q_3}\ast S_{r_3} \oplus
P_{q^\prime} \ast S_{r^\prime}$.

There are four gpps and we know
$w(P_{q_1} \ast S_{r_1} \oplus P_{q_2}\ast S_{r_2}) \geq 2r + 1$.
If $w(P_{q_3}\ast S_{r_3} \oplus P_{q^\prime} \ast S_{r^\prime}) \leq 2r$
we could reduce this pair and we would be back to $L = 3$,
so assume $w(P_{q_3}\ast S_{r_3} \oplus P_{q^\prime} \ast S_{r^\prime}) \geq 2r+1$.
Therefore
$n=4r +3=1+ w(P_{q_1} \ast S_{r_1} \oplus P_{q_2}\ast S_{r_2}) +
 w(P_{q_3}\ast S_{r_3} \oplus P_{q^\prime} \ast S_{r^\prime})
 \geq 1 +2r+1+2r+1 = n$.  We obtain
 $w(P_{q_1} \ast S_{r_1} \oplus P_{q_2}\ast S_{r_2}) + w(P_{q_3}\ast S_{r_3} \oplus
 P_{q^\prime} \ast S_{r^\prime})=  4r+2$.
 As $w(P_{q_1} \ast S_{r_1} \oplus P_{q_2}\ast S_{r_2})\geq 2r +1$ and
 $w(P_{q_3}\ast S_{r_3} \oplus P_{q^\prime} \ast S_{r^\prime}) \geq 2r +1$, we must have
 equality, that is,
 $w(P_{q_1} \ast S_{r_1} \oplus P_{q_2} \ast S_{r_2})= 2r +1 $
 and $w(P_{q_3} \ast S_{r_3} \oplus P_{q^\prime} \ast S_{r^\prime}) = 2r +1$.

The possible values for the $q$'s are the following. Either (i) $q_1=1$, $q_2 =0$ combined with $q_3=1$, $q^\prime =0$ or (ii)  $q_1=0$, $q_2 =1$ combined with $q_3=1$, $q^\prime =0$
or (iii)    $q_1=1$, $q_2 =0$ combined with $q_3=0$, $q^\prime =1$
or (iv) $q_1=0$, $q_2 =1$ combined with $q_3=0$, $q^\prime =1$.
We consider the case (ii) $q_1=q^\prime=0$ with $q_2=q_3= 1$ (the other possibilities are similar). We see that the tree has
the form $T=u+P(u)=u+(P_{0} \ast S_{r_1} \oplus P_{1}\ast S_{r_2}) \oplus
(P_{1}\ast S_{r_3} \oplus P_{0}\ast S_{r^\prime})$, with $r_1+r_2= r$,
$r_3+r^\prime = r$ but $r_1,r_2, r_3, r^\prime \leq r-1$.

We apply star-star
within each pair
reducing $T$ to $u+P_{0} \ast S_{r_1+r_2-r} \oplus P_{1}\ast S_{r}
\oplus P_{0} \ast S_{r_3+r^\prime-r} \oplus P_{1}\ast S_{r}$, but since
$r_1 + r_2 - r = 0$ and $r_3 + r^\prime - r =0$, the proof is complete.

\end{proof}

\section{A complete example}
\label{sec:exa}

Consider the tree $T$ given by Figure~\ref{fig:ex3:1}. This tree has $n=53$
vertices. Since $53 \equiv 1 \pmod 4$ we are going to properly transform it into
the prototype tree $T_1$. We notice that, in this case, $r=
\lfloor\frac{n}{4}\rfloor=13$. Thus we can always apply
\texttt{ReduceStarVertex}$(T, u)$ provided that $w(u) \leq 26$.
\begin{figure}[!ht]
  \centering
\begin{tikzpicture}[line cap=round,line join=round,>=triangle 45,x=1cm,y=1cm,scale=0.6, every node/.style={scale=0.8}]
\clip(-4.9,-3.8) rectangle (6.5,3.44);
\draw [line width=1pt] (-4,0)-- (-4.44,-1.01);
\draw [line width=1pt] (-4,0)-- (-3.46,-1.03);
\draw (-4.26,0.8) node[anchor=north west] {{\tiny $u_{4}$}};
\draw (-4.96,-1.12) node[anchor=north west] {{\tiny $3 \times P_0 \ast S_{1}$}};
\draw (-3.3,1.96) node[anchor=north west] {{\tiny $P_7 \ast S_{1}$}};
\draw [line width=1pt] (-4,0)-- (-3,0);
\draw [line width=1pt] (2,0)-- (3,0);
\draw [line width=1pt] (-1,0)-- (-1,-1);
\draw [line width=1pt] (-1,-1)-- (-1.7085590969545494,-1.995084663720346);
\draw [line width=1pt] (-1,-2.02)-- (-1.44,-3.03);
\draw [line width=1pt] (-1,-2.02)-- (-0.46,-3.05);
\draw [line width=1pt] (-1,1)-- (-1.42,1.99);
\draw [line width=1pt] (-1,1)-- (-0.46,2.01);
\draw [line width=1pt] (4,-1)-- (3.5301160671617895,-1.9877781701461203);
\draw [line width=1pt] (4,-1)-- (4.539533675750619,-1.9877781701461203);
\draw [line width=1pt] (1,0)-- (1,1);
\draw [line width=1pt] (-0.02,-0.01)-- (-0.02,-1.01);
\draw [line width=1pt] (4,0)-- (5,0);
\draw [line width=1pt] (-1.98,0)-- (-1.98,1);
\draw [line width=1pt] (-3,0)-- (-1.98,0);
\draw [line width=1pt] (-1.98,0)-- (-1,0);
\draw [line width=1pt] (-1,0)-- (-0.02,-0.01);
\draw [line width=1pt] (-1,-1)-- (-1,-2.02);
\draw [line width=1pt] (-1,0)-- (-1,1);
\draw [line width=1pt] (-0.02,-0.01)-- (1,0);
\draw [line width=1pt] (1,0)-- (2,0);
\draw [line width=1pt] (4.014083413745475,0.9782502824881779)-- (3.779013559690542,2.008409348787737);
\draw [line width=1pt] (4.014083413745475,0.9782502824881779)-- (4.249153267800409,2.008409348787737);
\draw [line width=1pt] (4,0)-- (4.014083413745475,0.9782502824881779);
\draw [line width=1pt] (4,0)-- (3,0);
\draw [line width=1pt] (4,-1)-- (4,0);
\draw [line width=1pt] (3.3503567670021344,1.9945817103139172)-- (4.014083413745475,0.9782502824881779);
\draw [line width=1pt] (4.014083413745475,0.9782502824881779)-- (4.719292975910273,1.9945817103139172);
\draw [line width=1pt] (-4,0)-- (-4,-1);
\draw (3.22,3.14) node[anchor=north west] {{\tiny $4 \times P_0 \ast S_{1}$}};
\draw (-3.04,-1.94) node[anchor=north west] {{\tiny $P_1 \ast S_{0}$}};
\draw (-0.46,3.06) node[anchor=north west] {{\tiny $P_1 \ast S_{0}$}};
\draw (-2.44,3.04) node[anchor=north west] {{\tiny $P_0 \ast S_{1}$}};
\draw (3.3,-2.02) node[anchor=north west] {{\tiny $2 \times P_0 \ast S_{1}$}};
\draw (4.64,0.02) node[anchor=north west] {{\tiny $P_0 \ast S_{1}$}};
\draw (1.26,1.8) node[anchor=north west] {{\tiny $P_1 \ast S_{0}$}};
\draw (0.02,-0.88) node[anchor=north west] {{\tiny $P_1 \ast S_{1}$}};
\draw (-1.76,-3.08) node[anchor=north west] {{\tiny $2 \times P_1 \ast S_{0}$}};
\draw (-0.84,1.48) node[anchor=north west] {{\tiny $u_{2}$}};
\draw (-0.8,-1.48) node[anchor=north west] {{\tiny $u_{1}$}};
\draw (4.28,1.3) node[anchor=north west] {{\tiny $u_{5}$}};
\draw (3.38,-0.46) node[anchor=north west] {{\tiny $u_{3}$}};
\begin{scriptsize}
\draw [fill=black] (-4,0) circle (2.5pt);
\draw [fill=black] (-4.44,-1.01) ++(-3.5pt,0 pt) -- ++(3.5pt,3.5pt)--++(3.5pt,-3.5pt)--++(-3.5pt,-3.5pt)--++(-3.5pt,3.5pt);
\draw [fill=black] (-3.46,-1.03) ++(-3.5pt,0 pt) -- ++(3.5pt,3.5pt)--++(3.5pt,-3.5pt)--++(-3.5pt,-3.5pt)--++(-3.5pt,3.5pt);
\draw [fill=black] (-3,0) circle (2.5pt);
\draw [fill=black] (-1,0) circle (2.5pt);
\draw [fill=black] (2,0) circle (2.5pt);
\draw [fill=black] (3,0) circle (2.5pt);
\draw [fill=black] (-1,-1) circle (2.5pt);
\draw [fill=black] (-1.7085590969545494,-1.995084663720346) ++(-3.5pt,0 pt) -- ++(3.5pt,3.5pt)--++(3.5pt,-3.5pt)--++(-3.5pt,-3.5pt)--++(-3.5pt,3.5pt);
\draw [fill=black] (-1,-2.02) circle (2.5pt);
\draw [fill=black] (-1.44,-3.03) ++(-3.5pt,0 pt) -- ++(3.5pt,3.5pt)--++(3.5pt,-3.5pt)--++(-3.5pt,-3.5pt)--++(-3.5pt,3.5pt);
\draw [fill=black] (-0.46,-3.05) ++(-3.5pt,0 pt) -- ++(3.5pt,3.5pt)--++(3.5pt,-3.5pt)--++(-3.5pt,-3.5pt)--++(-3.5pt,3.5pt);
\draw [fill=black] (-1,1) circle (2.5pt);
\draw [fill=black] (-1.42,1.99) ++(-3.5pt,0 pt) -- ++(3.5pt,3.5pt)--++(3.5pt,-3.5pt)--++(-3.5pt,-3.5pt)--++(-3.5pt,3.5pt);
\draw [fill=black] (-0.46,2.01) ++(-3.5pt,0 pt) -- ++(3.5pt,3.5pt)--++(3.5pt,-3.5pt)--++(-3.5pt,-3.5pt)--++(-3.5pt,3.5pt);
\draw [fill=black] (4,-1) circle (2.5pt);
\draw [fill=black] (3.5301160671617895,-1.9877781701461203) ++(-3.5pt,0 pt) -- ++(3.5pt,3.5pt)--++(3.5pt,-3.5pt)--++(-3.5pt,-3.5pt)--++(-3.5pt,3.5pt);
\draw [fill=black] (4.539533675750619,-1.9877781701461203) ++(-3.5pt,0 pt) -- ++(3.5pt,3.5pt)--++(3.5pt,-3.5pt)--++(-3.5pt,-3.5pt)--++(-3.5pt,3.5pt);
\draw [fill=black] (1,0) circle (2.5pt);
\draw [fill=black] (1,1) ++(-3.5pt,0 pt) -- ++(3.5pt,3.5pt)--++(3.5pt,-3.5pt)--++(-3.5pt,-3.5pt)--++(-3.5pt,3.5pt);
\draw [fill=black] (-0.02,-0.01) circle (2.5pt);
\draw [fill=black] (-0.02,-1.01) ++(-3.5pt,0 pt) -- ++(3.5pt,3.5pt)--++(3.5pt,-3.5pt)--++(-3.5pt,-3.5pt)--++(-3.5pt,3.5pt);
\draw [fill=black] (4,0) circle (2.5pt);
\draw [fill=black] (5,0) ++(-3.5pt,0 pt) -- ++(3.5pt,3.5pt)--++(3.5pt,-3.5pt)--++(-3.5pt,-3.5pt)--++(-3.5pt,3.5pt);
\draw [fill=black] (-1.98,0) circle (2.5pt);
\draw [fill=black] (-1.98,1) ++(-3.5pt,0 pt) -- ++(3.5pt,3.5pt)--++(3.5pt,-3.5pt)--++(-3.5pt,-3.5pt)--++(-3.5pt,3.5pt);
\draw [fill=black] (4.014083413745475,0.9782502824881779) circle (2.5pt);
\draw [fill=black] (3.779013559690542,2.008409348787737) ++(-3.5pt,0 pt) -- ++(3.5pt,3.5pt)--++(3.5pt,-3.5pt)--++(-3.5pt,-3.5pt)--++(-3.5pt,3.5pt);
\draw [fill=black] (4.249153267800409,2.008409348787737) ++(-3.5pt,0 pt) -- ++(3.5pt,3.5pt)--++(3.5pt,-3.5pt)--++(-3.5pt,-3.5pt)--++(-3.5pt,3.5pt);
\draw [fill=black] (3.3503567670021344,1.9945817103139172) ++(-3.5pt,0 pt) -- ++(3.5pt,3.5pt)--++(3.5pt,-3.5pt)--++(-3.5pt,-3.5pt)--++(-3.5pt,3.5pt);
\draw [fill=black] (4.719292975910273,1.9945817103139172) ++(-3.5pt,0 pt) -- ++(3.5pt,3.5pt)--++(3.5pt,-3.5pt)--++(-3.5pt,-3.5pt)--++(-3.5pt,3.5pt);
\draw [fill=black] (-4,-1) ++(-3.5pt,0 pt) -- ++(3.5pt,3.5pt)--++(3.5pt,-3.5pt)--++(-3.5pt,-3.5pt)--++(-3.5pt,3.5pt);
\end{scriptsize}
\end{tikzpicture} \;\begin{tikzpicture}[line cap=round,line join=round,>=triangle 45,x=1cm,y=1cm,scale=0.6, , every node/.style={scale=0.8}]
\clip(-4.9,-3.0) rectangle (6.5,3.44);
\draw [line width=1pt] (-4,0)-- (-4.44,-1.01);
\draw [line width=1pt] (-4,0)-- (-3.46,-1.03);
\draw (-4.26,0.8) node[anchor=north west] {{\tiny $u_{4}$}};
\draw (-4.96,-1.12) node[anchor=north west] {{\tiny $3 \times P_0 \ast S_{1}$}};
\draw (-3.3,1.96) node[anchor=north west] {{\tiny $P_7 \ast S_{1}$}};
\draw [line width=1pt] (-4,0)-- (-3,0);
\draw [line width=1pt] (2,0)-- (3,0);
\draw [line width=1pt] (-1,0)-- (-1,-1);
\draw [line width=1pt] (-1,-1)-- (-1.7085590969545494,-1.995084663720346);
\draw [line width=1pt] (-1,1)-- (-1.42,1.99);
\draw [line width=1pt] (-1,1)-- (-0.46,2.01);
\draw [line width=1pt] (4,-1)-- (3.5301160671617895,-1.9877781701461203);
\draw [line width=1pt] (4,-1)-- (4.539533675750619,-1.9877781701461203);
\draw [line width=1pt] (1,0)-- (1,1);
\draw [line width=1pt] (-0.02,-0.01)-- (-0.02,-1.01);
\draw [line width=1pt] (4,0)-- (5,0);
\draw [line width=1pt] (-1.98,0)-- (-1.98,1);
\draw [line width=1pt] (-3,0)-- (-1.98,0);
\draw [line width=1pt] (-1.98,0)-- (-1,0);
\draw [line width=1pt] (-1,0)-- (-0.02,-0.01);
\draw [line width=1pt] (-1,0)-- (-1,1);
\draw [line width=1pt] (-0.02,-0.01)-- (1,0);
\draw [line width=1pt] (1,0)-- (2,0);
\draw [line width=1pt] (4.014083413745475,0.9782502824881779)-- (3.779013559690542,2.008409348787737);
\draw [line width=1pt] (4.014083413745475,0.9782502824881779)-- (4.249153267800409,2.008409348787737);
\draw [line width=1pt] (4,0)-- (4.014083413745475,0.9782502824881779);
\draw [line width=1pt] (4,0)-- (3,0);
\draw [line width=1pt] (4,-1)-- (4,0);
\draw [line width=1pt] (3.3503567670021344,1.9945817103139172)-- (4.014083413745475,0.9782502824881779);
\draw [line width=1pt] (4.014083413745475,0.9782502824881779)-- (4.719292975910273,1.9945817103139172);
\draw [line width=1pt] (-4,0)-- (-4,-1);
\draw (3.22,3.14) node[anchor=north west] {{\tiny $4 \times P_0 \ast S_{1}$}};
\draw (-3.08,-1.98) node[anchor=north west] {{\tiny $P_1 \ast S_{0}$}};
\draw (-0.46,3.06) node[anchor=north west] {{\tiny $P_1 \ast S_{0}$}};
\draw (-2.44,3.04) node[anchor=north west] {{\tiny $P_0 \ast S_{1}$}};
\draw (3.3,-2.02) node[anchor=north west] {{\tiny $2 \times P_0 \ast S_{1}$}};
\draw (4.64,0.02) node[anchor=north west] {{\tiny $P_0 \ast S_{1}$}};
\draw (1.26,1.8) node[anchor=north west] {{\tiny $P_1 \ast S_{0}$}};
\draw (0.02,-0.88) node[anchor=north west] {{\tiny $P_1 \ast S_{1}$}};
\draw (-0.84,1.48) node[anchor=north west] {{\tiny $u_{1}$}};
\draw (-1.72,-0.36) node[anchor=north west] {{\tiny $u_{2}$}};
\draw (4.28,1.3) node[anchor=north west] {{\tiny $u_{5}$}};
\draw (3.38,-0.46) node[anchor=north west] {{\tiny $u_{3}$}};
\draw (-0.96,-2.08) node[anchor=north west] {{\tiny $P_1 \ast S_{1}$}};
\draw [line width=1pt] (-1,-1)-- (-0.48,-2.02);
\begin{scriptsize}
\draw [fill=black] (-4,0) circle (2.5pt);
\draw [fill=black] (-4.44,-1.01) ++(-3.5pt,0 pt) -- ++(3.5pt,3.5pt)--++(3.5pt,-3.5pt)--++(-3.5pt,-3.5pt)--++(-3.5pt,3.5pt);
\draw [fill=black] (-3.46,-1.03) ++(-3.5pt,0 pt) -- ++(3.5pt,3.5pt)--++(3.5pt,-3.5pt)--++(-3.5pt,-3.5pt)--++(-3.5pt,3.5pt);
\draw [fill=black] (-3,0) circle (2.5pt);
\draw [fill=black] (-1,0) circle (2.5pt);
\draw [fill=black] (2,0) circle (2.5pt);
\draw [fill=black] (3,0) circle (2.5pt);
\draw [fill=black] (-1,-1) circle (2.5pt);
\draw [fill=black] (-1.7085590969545494,-1.995084663720346) ++(-3.5pt,0 pt) -- ++(3.5pt,3.5pt)--++(3.5pt,-3.5pt)--++(-3.5pt,-3.5pt)--++(-3.5pt,3.5pt);
\draw [fill=black] (-0.48,-2.02) ++(-3.5pt,0 pt) -- ++(3.5pt,3.5pt)--++(3.5pt,-3.5pt)--++(-3.5pt,-3.5pt)--++(-3.5pt,3.5pt);
\draw [fill=black] (-1,1) circle (2.5pt);
\draw [fill=black] (-1.42,1.99) ++(-3.5pt,0 pt) -- ++(3.5pt,3.5pt)--++(3.5pt,-3.5pt)--++(-3.5pt,-3.5pt)--++(-3.5pt,3.5pt);
\draw [fill=black] (-0.46,2.01) ++(-3.5pt,0 pt) -- ++(3.5pt,3.5pt)--++(3.5pt,-3.5pt)--++(-3.5pt,-3.5pt)--++(-3.5pt,3.5pt);
\draw [fill=black] (4,-1) circle (2.5pt);
\draw [fill=black] (3.5301160671617895,-1.9877781701461203) ++(-3.5pt,0 pt) -- ++(3.5pt,3.5pt)--++(3.5pt,-3.5pt)--++(-3.5pt,-3.5pt)--++(-3.5pt,3.5pt);
\draw [fill=black] (4.539533675750619,-1.9877781701461203) ++(-3.5pt,0 pt) -- ++(3.5pt,3.5pt)--++(3.5pt,-3.5pt)--++(-3.5pt,-3.5pt)--++(-3.5pt,3.5pt);
\draw [fill=black] (1,0) circle (2.5pt);
\draw [fill=black] (1,1) ++(-3.5pt,0 pt) -- ++(3.5pt,3.5pt)--++(3.5pt,-3.5pt)--++(-3.5pt,-3.5pt)--++(-3.5pt,3.5pt);
\draw [fill=black] (-0.02,-0.01) circle (2.5pt);
\draw [fill=black] (-0.02,-1.01) ++(-3.5pt,0 pt) -- ++(3.5pt,3.5pt)--++(3.5pt,-3.5pt)--++(-3.5pt,-3.5pt)--++(-3.5pt,3.5pt);
\draw [fill=black] (4,0) circle (2.5pt);
\draw [fill=black] (5,0) ++(-3.5pt,0 pt) -- ++(3.5pt,3.5pt)--++(3.5pt,-3.5pt)--++(-3.5pt,-3.5pt)--++(-3.5pt,3.5pt);
\draw [fill=black] (-1.98,0) circle (2.5pt);
\draw [fill=black] (-1.98,1) ++(-3.5pt,0 pt) -- ++(3.5pt,3.5pt)--++(3.5pt,-3.5pt)--++(-3.5pt,-3.5pt)--++(-3.5pt,3.5pt);
\draw [fill=black] (4.014083413745475,0.9782502824881779) circle (2.5pt);
\draw [fill=black] (3.779013559690542,2.008409348787737) ++(-3.5pt,0 pt) -- ++(3.5pt,3.5pt)--++(3.5pt,-3.5pt)--++(-3.5pt,-3.5pt)--++(-3.5pt,3.5pt);
\draw [fill=black] (4.249153267800409,2.008409348787737) ++(-3.5pt,0 pt) -- ++(3.5pt,3.5pt)--++(3.5pt,-3.5pt)--++(-3.5pt,-3.5pt)--++(-3.5pt,3.5pt);
\draw [fill=black] (3.3503567670021344,1.9945817103139172) ++(-3.5pt,0 pt) -- ++(3.5pt,3.5pt)--++(3.5pt,-3.5pt)--++(-3.5pt,-3.5pt)--++(-3.5pt,3.5pt);
\draw [fill=black] (4.719292975910273,1.9945817103139172) ++(-3.5pt,0 pt) -- ++(3.5pt,3.5pt)--++(3.5pt,-3.5pt)--++(-3.5pt,-3.5pt)--++(-3.5pt,3.5pt);
\draw [fill=black] (-4,-1) ++(-3.5pt,0 pt) -- ++(3.5pt,3.5pt)--++(3.5pt,-3.5pt)--++(-3.5pt,-3.5pt)--++(-3.5pt,3.5pt);
\end{scriptsize}
\end{tikzpicture} 
\caption{Application of procedures \texttt{InitiateRepresentation} and identification of 5 starlike vertices.}
\label{fig:an_example:0001}
\end{figure}
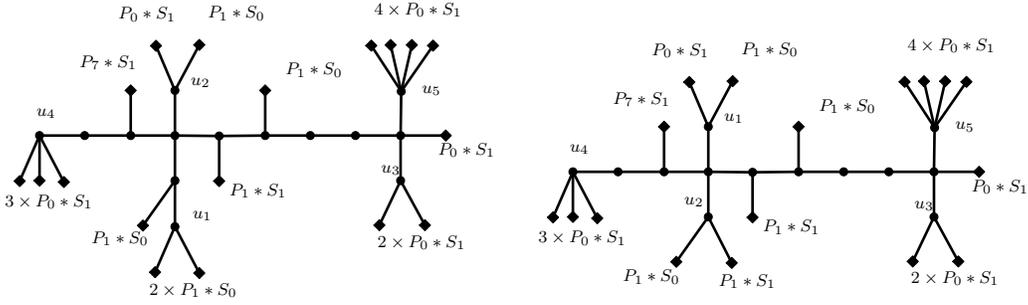

The initialization step of algorithm \texttt{Transform} is to apply the
procedure \texttt{I\-ni\-ti\-a\-te\-Re\-pre\-sen\-ta\-tion} to the tree $T$ of
Figure~\ref{fig:ex3:1}, obtaining the $(P_q, S_r)$ representation illustrated
in Figure~\ref{fig:an_example:0001} (left).
The second step is to identify
the starlike vertices.
There are five, $u_1, \, u_2, \, u_3, \, u_4, \, u_5$
in increasing order of weights $2, \, 3, \, 4, \, 6, \, 8$.
As prescribed by the Algorithm \texttt{Transform}$(T)$, for $k=5 \geq 3$
we should apply \texttt{ReduceStarVertex}$(T, u_1)$, which is
$P(u_1) = P_1 \ast S_0 \oplus P_1 \ast S_0$.
After \texttt{ReduceStarVertex}$(T, u_1)$, the gpps at $u_1$
are reduced to
$P(u_1) = P_0\ast S_1$.
However $q$ gets incremented by one in the if-statement of
\texttt{ReduceStarVertex}
so the new gpp is $P_1\ast S_1$.
The starlike vertex $u_1$ is
eliminated producing a new starlike vertex with weight 4 which is labeled
as $u_2$, where
$P(u_2) = P_1\ast S_0 \oplus P_1 \ast S_1$.
The previous $u_2$ is labeled as the new $u_1$. This is illustrated in
Figure~\ref{fig:an_example:0001} (right).

In Figure~\ref{fig:an_example:0001} (right) we still have $k=5 \geq 3$ so
we apply \texttt{ReduceStarVertex}$(T, u_1)$.
After
\texttt{ReduceStarVertex}$(T, u_1)$, the starlike vertex $u_1$ is transformed
to $P(u_1)= P_1 \ast S_1$.
However, as $u_1$ is attached to its neighbor by
the path $P_1$, its neighbor will have the single ggp $P_2 \ast S_1$.
As no new starlike vertex is created, the number of starlike vertices is reduced by
one, therefore the previous $u_{i}$ is labeled as the new $u_{i-1}$ for
$i=2,3,4,5$, as in Figure~\ref{fig:an_example:0203} (left).
Now we have $k=4
\geq 3$ and after applying \texttt{ReduceStarVertex}$(T, u_1)$, we obtain a
new starlike vertex
$ P_1 \ast S_2 \oplus P_2 \ast S_1$, whose weight is 9.
Hence, the number of starlike vertices remains the same and the remaining
starlike vertices are relabeled according to its weights, as in
Figure~\ref{fig:an_example:0203} (right).
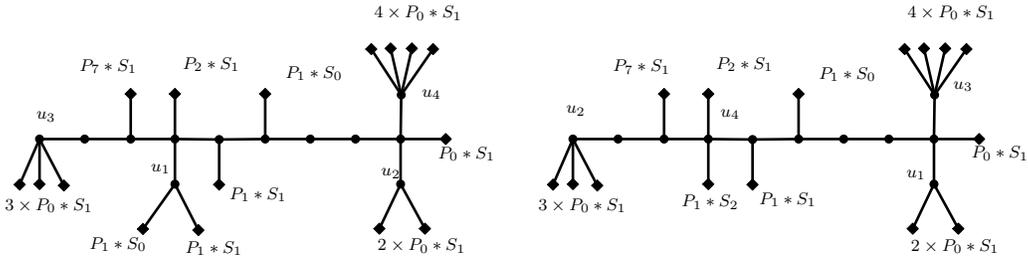
\begin{figure}[!ht]
  \centering
\begin{tikzpicture}[line cap=round,line join=round,>=triangle 45,x=1cm,y=1cm,scale=0.6, every node/.style={scale=0.8}]
\clip(-4.9,-3.0) rectangle (6.5,3.44);
\draw [line width=1pt] (-4,0)-- (-4.44,-1.01);
\draw [line width=1pt] (-4,0)-- (-3.46,-1.03);
\draw (-4.26,0.8) node[anchor=north west] {{\tiny $u_{3}$}};
\draw (-4.96,-1.12) node[anchor=north west] {{\tiny $3 \times P_0 \ast S_{1}$}};
\draw (-3.3,1.96) node[anchor=north west] {{\tiny $P_7 \ast S_{1}$}};
\draw [line width=1pt] (-4,0)-- (-3,0);
\draw [line width=1pt] (2,0)-- (3,0);
\draw [line width=1pt] (-1,0)-- (-1,-1);
\draw [line width=1pt] (-1,-1)-- (-1.7085590969545494,-1.995084663720346);
\draw [line width=1pt] (4,-1)-- (3.5301160671617895,-1.9877781701461203);
\draw [line width=1pt] (4,-1)-- (4.539533675750619,-1.9877781701461203);
\draw [line width=1pt] (1,0)-- (1,1);
\draw [line width=1pt] (-0.02,-0.01)-- (-0.02,-1.01);
\draw [line width=1pt] (4,0)-- (5,0);
\draw [line width=1pt] (-1.98,0)-- (-1.98,1);
\draw [line width=1pt] (-3,0)-- (-1.98,0);
\draw [line width=1pt] (-1.98,0)-- (-1,0);
\draw [line width=1pt] (-1,0)-- (-0.02,-0.01);
\draw [line width=1pt] (-0.02,-0.01)-- (1,0);
\draw [line width=1pt] (1,0)-- (2,0);
\draw [line width=1pt] (4.014083413745475,0.9782502824881779)-- (3.779013559690542,2.008409348787737);
\draw [line width=1pt] (4.014083413745475,0.9782502824881779)-- (4.249153267800409,2.008409348787737);
\draw [line width=1pt] (4,0)-- (4.014083413745475,0.9782502824881779);
\draw [line width=1pt] (4,0)-- (3,0);
\draw [line width=1pt] (4,-1)-- (4,0);
\draw [line width=1pt] (3.3503567670021344,1.9945817103139172)-- (4.014083413745475,0.9782502824881779);
\draw [line width=1pt] (4.014083413745475,0.9782502824881779)-- (4.719292975910273,1.9945817103139172);
\draw [line width=1pt] (-4,0)-- (-4,-1);
\draw (3.22,3.14) node[anchor=north west] {{\tiny $4 \times P_0 \ast S_{1}$}};
\draw (-3.08,-1.98) node[anchor=north west] {{\tiny $P_1 \ast S_{0}$}};
\draw (3.3,-2.02) node[anchor=north west] {{\tiny $2 \times P_0 \ast S_{1}$}};
\draw (4.64,0.02) node[anchor=north west] {{\tiny $P_0 \ast S_{1}$}};
\draw (1.26,1.8) node[anchor=north west] {{\tiny $P_1 \ast S_{0}$}};
\draw (0.02,-0.88) node[anchor=north west] {{\tiny $P_1 \ast S_{1}$}};
\draw (-1.72,-0.36) node[anchor=north west] {{\tiny $u_{1}$}};
\draw (4.28,1.3) node[anchor=north west] {{\tiny $u_{4}$}};
\draw (3.38,-0.46) node[anchor=north west] {{\tiny $u_{2}$}};
\draw (-0.96,-2.08) node[anchor=north west] {{\tiny $P_1 \ast S_{1}$}};
\draw [line width=1pt] (-1,-1)-- (-0.48,-2.02);
\draw [line width=1pt] (-1,1)-- (-1,0);
\draw (-1.02,2) node[anchor=north west] {{\tiny $P_2 \ast S_{1}$}};
\begin{scriptsize}
\draw [fill=black] (-4,0) circle (2.5pt);
\draw [fill=black] (-4.44,-1.01) ++(-3.5pt,0 pt) -- ++(3.5pt,3.5pt)--++(3.5pt,-3.5pt)--++(-3.5pt,-3.5pt)--++(-3.5pt,3.5pt);
\draw [fill=black] (-3.46,-1.03) ++(-3.5pt,0 pt) -- ++(3.5pt,3.5pt)--++(3.5pt,-3.5pt)--++(-3.5pt,-3.5pt)--++(-3.5pt,3.5pt);
\draw [fill=black] (-3,0) circle (2.5pt);
\draw [fill=black] (-1,0) circle (2.5pt);
\draw [fill=black] (2,0) circle (2.5pt);
\draw [fill=black] (3,0) circle (2.5pt);
\draw [fill=black] (-1,-1) circle (2.5pt);
\draw [fill=black] (-1.7085590969545494,-1.995084663720346) ++(-3.5pt,0 pt) -- ++(3.5pt,3.5pt)--++(3.5pt,-3.5pt)--++(-3.5pt,-3.5pt)--++(-3.5pt,3.5pt);
\draw [fill=black] (-0.48,-2.02) ++(-3.5pt,0 pt) -- ++(3.5pt,3.5pt)--++(3.5pt,-3.5pt)--++(-3.5pt,-3.5pt)--++(-3.5pt,3.5pt);
\draw [fill=black] (-1,1) ++(-3.5pt,0 pt) -- ++(3.5pt,3.5pt)--++(3.5pt,-3.5pt)--++(-3.5pt,-3.5pt)--++(-3.5pt,3.5pt);
\draw [fill=black] (4,-1) circle (2.5pt);
\draw [fill=black] (3.5301160671617895,-1.9877781701461203) ++(-3.5pt,0 pt) -- ++(3.5pt,3.5pt)--++(3.5pt,-3.5pt)--++(-3.5pt,-3.5pt)--++(-3.5pt,3.5pt);
\draw [fill=black] (4.539533675750619,-1.9877781701461203) ++(-3.5pt,0 pt) -- ++(3.5pt,3.5pt)--++(3.5pt,-3.5pt)--++(-3.5pt,-3.5pt)--++(-3.5pt,3.5pt);
\draw [fill=black] (1,0) circle (2.5pt);
\draw [fill=black] (1,1) ++(-3.5pt,0 pt) -- ++(3.5pt,3.5pt)--++(3.5pt,-3.5pt)--++(-3.5pt,-3.5pt)--++(-3.5pt,3.5pt);
\draw [fill=black] (-0.02,-0.01) circle (2.5pt);
\draw [fill=black] (-0.02,-1.01) ++(-3.5pt,0 pt) -- ++(3.5pt,3.5pt)--++(3.5pt,-3.5pt)--++(-3.5pt,-3.5pt)--++(-3.5pt,3.5pt);
\draw [fill=black] (4,0) circle (2.5pt);
\draw [fill=black] (5,0) ++(-3.5pt,0 pt) -- ++(3.5pt,3.5pt)--++(3.5pt,-3.5pt)--++(-3.5pt,-3.5pt)--++(-3.5pt,3.5pt);
\draw [fill=black] (-1.98,0) circle (2.5pt);
\draw [fill=black] (-1.98,1) ++(-3.5pt,0 pt) -- ++(3.5pt,3.5pt)--++(3.5pt,-3.5pt)--++(-3.5pt,-3.5pt)--++(-3.5pt,3.5pt);
\draw [fill=black] (4.014083413745475,0.9782502824881779) circle (2.5pt);
\draw [fill=black] (3.779013559690542,2.008409348787737) ++(-3.5pt,0 pt) -- ++(3.5pt,3.5pt)--++(3.5pt,-3.5pt)--++(-3.5pt,-3.5pt)--++(-3.5pt,3.5pt);
\draw [fill=black] (4.249153267800409,2.008409348787737) ++(-3.5pt,0 pt) -- ++(3.5pt,3.5pt)--++(3.5pt,-3.5pt)--++(-3.5pt,-3.5pt)--++(-3.5pt,3.5pt);
\draw [fill=black] (3.3503567670021344,1.9945817103139172) ++(-3.5pt,0 pt) -- ++(3.5pt,3.5pt)--++(3.5pt,-3.5pt)--++(-3.5pt,-3.5pt)--++(-3.5pt,3.5pt);
\draw [fill=black] (4.719292975910273,1.9945817103139172) ++(-3.5pt,0 pt) -- ++(3.5pt,3.5pt)--++(3.5pt,-3.5pt)--++(-3.5pt,-3.5pt)--++(-3.5pt,3.5pt);
\draw [fill=black] (-4,-1) ++(-3.5pt,0 pt) -- ++(3.5pt,3.5pt)--++(3.5pt,-3.5pt)--++(-3.5pt,-3.5pt)--++(-3.5pt,3.5pt);
\end{scriptsize}
\end{tikzpicture} \;\begin{tikzpicture}[line cap=round,line join=round,>=triangle 45,x=1cm,y=1cm,scale=0.6, every node/.style={scale=0.8}]
\clip(-4.9,-3.0) rectangle (6.5,3.44);
\draw [line width=1pt] (-4,0)-- (-4.44,-1.01);
\draw [line width=1pt] (-4,0)-- (-3.46,-1.03);
\draw (4.24,1.47) node[anchor=north west] {{\tiny $u_{3}$}};
\draw (-4.96,-1.13) node[anchor=north west] {{\tiny $3 \times P_0 \ast S_{1}$}};
\draw (-3.3,1.95) node[anchor=north west] {{\tiny $P_7 \ast S_{1}$}};
\draw [line width=1pt] (-4,0)-- (-3,0);
\draw [line width=1pt] (2,0)-- (3,0);
\draw [line width=1pt] (4,-1)-- (3.5301160671617895,-1.9877781701461203);
\draw [line width=1pt] (4,-1)-- (4.539533675750619,-1.9877781701461203);
\draw [line width=1pt] (1,0)-- (1,1);
\draw [line width=1pt] (-0.02,-0.01)-- (-0.02,-1.01);
\draw [line width=1pt] (4,0)-- (5,0);
\draw [line width=1pt] (-1.98,0)-- (-1.98,1);
\draw [line width=1pt] (-3,0)-- (-1.98,0);
\draw [line width=1pt] (-1.98,0)-- (-1,0);
\draw [line width=1pt] (-1,0)-- (-0.02,-0.01);
\draw [line width=1pt] (-0.02,-0.01)-- (1,0);
\draw [line width=1pt] (1,0)-- (2,0);
\draw [line width=1pt] (4.014083413745475,0.9782502824881779)-- (3.779013559690542,2.008409348787737);
\draw [line width=1pt] (4.014083413745475,0.9782502824881779)-- (4.249153267800409,2.008409348787737);
\draw [line width=1pt] (4,0)-- (4.014083413745475,0.9782502824881779);
\draw [line width=1pt] (4,0)-- (3,0);
\draw [line width=1pt] (4,-1)-- (4,0);
\draw [line width=1pt] (3.3503567670021344,1.9945817103139172)-- (4.014083413745475,0.9782502824881779);
\draw [line width=1pt] (4.014083413745475,0.9782502824881779)-- (4.719292975910273,1.9945817103139172);
\draw [line width=1pt] (-4,0)-- (-4,-1);
\draw (3.22,3.15) node[anchor=north west] {{\tiny $4 \times P_0 \ast S_{1}$}};
\draw (3.3,-2.03) node[anchor=north west] {{\tiny $2 \times P_0 \ast S_{1}$}};
\draw (4.64,0.03) node[anchor=north west] {{\tiny $P_0 \ast S_{1}$}};
\draw (1.26,1.79) node[anchor=north west] {{\tiny $P_1 \ast S_{0}$}};
\draw (-0.06,-0.99) node[anchor=north west] {{\tiny $P_1 \ast S_{1}$}};
\draw (3.2,-0.49) node[anchor=north west] {{\tiny $u_{1}$}};
\draw (-4.34,0.93) node[anchor=north west] {{\tiny $u_{2}$}};
\draw (-1.78,-1.05) node[anchor=north west] {{\tiny $P_1 \ast S_{2}$}};
\draw [line width=1pt] (-1,1)-- (-1,0);
\draw (-1.02,2.01) node[anchor=north west] {{\tiny $P_2 \ast S_{1}$}};
\draw [line width=1pt] (-1,0)-- (-1,-1);
\draw (-0.92,0.85) node[anchor=north west] {{\tiny $u_{4}$}};
\begin{scriptsize}
\draw [fill=black] (-4,0) circle (2.5pt);
\draw [fill=black] (-4.44,-1.01) ++(-3.5pt,0 pt) -- ++(3.5pt,3.5pt)--++(3.5pt,-3.5pt)--++(-3.5pt,-3.5pt)--++(-3.5pt,3.5pt);
\draw [fill=black] (-3.46,-1.03) ++(-3.5pt,0 pt) -- ++(3.5pt,3.5pt)--++(3.5pt,-3.5pt)--++(-3.5pt,-3.5pt)--++(-3.5pt,3.5pt);
\draw [fill=black] (-3,0) circle (2.5pt);
\draw [fill=black] (-1,0) circle (2.5pt);
\draw [fill=black] (2,0) circle (2.5pt);
\draw [fill=black] (3,0) circle (2.5pt);
\draw [fill=black] (-1,-1) ++(-3.5pt,0 pt) -- ++(3.5pt,3.5pt)--++(3.5pt,-3.5pt)--++(-3.5pt,-3.5pt)--++(-3.5pt,3.5pt);
\draw [fill=black] (-1,1) ++(-3.5pt,0 pt) -- ++(3.5pt,3.5pt)--++(3.5pt,-3.5pt)--++(-3.5pt,-3.5pt)--++(-3.5pt,3.5pt);
\draw [fill=black] (4,-1) circle (2.5pt);
\draw [fill=black] (3.5301160671617895,-1.9877781701461203) ++(-3.5pt,0 pt) -- ++(3.5pt,3.5pt)--++(3.5pt,-3.5pt)--++(-3.5pt,-3.5pt)--++(-3.5pt,3.5pt);
\draw [fill=black] (4.539533675750619,-1.9877781701461203) ++(-3.5pt,0 pt) -- ++(3.5pt,3.5pt)--++(3.5pt,-3.5pt)--++(-3.5pt,-3.5pt)--++(-3.5pt,3.5pt);
\draw [fill=black] (1,0) circle (2.5pt);
\draw [fill=black] (1,1) ++(-3.5pt,0 pt) -- ++(3.5pt,3.5pt)--++(3.5pt,-3.5pt)--++(-3.5pt,-3.5pt)--++(-3.5pt,3.5pt);
\draw [fill=black] (-0.02,-0.01) circle (2.5pt);
\draw [fill=black] (-0.02,-1.01) ++(-3.5pt,0 pt) -- ++(3.5pt,3.5pt)--++(3.5pt,-3.5pt)--++(-3.5pt,-3.5pt)--++(-3.5pt,3.5pt);
\draw [fill=black] (4,0) circle (2.5pt);
\draw [fill=black] (5,0) ++(-3.5pt,0 pt) -- ++(3.5pt,3.5pt)--++(3.5pt,-3.5pt)--++(-3.5pt,-3.5pt)--++(-3.5pt,3.5pt);
\draw [fill=black] (-1.98,0) circle (2.5pt);
\draw [fill=black] (-1.98,1) ++(-3.5pt,0 pt) -- ++(3.5pt,3.5pt)--++(3.5pt,-3.5pt)--++(-3.5pt,-3.5pt)--++(-3.5pt,3.5pt);
\draw [fill=black] (4.014083413745475,0.9782502824881779) circle (2.5pt);
\draw [fill=black] (3.779013559690542,2.008409348787737) ++(-3.5pt,0 pt) -- ++(3.5pt,3.5pt)--++(3.5pt,-3.5pt)--++(-3.5pt,-3.5pt)--++(-3.5pt,3.5pt);
\draw [fill=black] (4.249153267800409,2.008409348787737) ++(-3.5pt,0 pt) -- ++(3.5pt,3.5pt)--++(3.5pt,-3.5pt)--++(-3.5pt,-3.5pt)--++(-3.5pt,3.5pt);
\draw [fill=black] (3.3503567670021344,1.9945817103139172) ++(-3.5pt,0 pt) -- ++(3.5pt,3.5pt)--++(3.5pt,-3.5pt)--++(-3.5pt,-3.5pt)--++(-3.5pt,3.5pt);
\draw [fill=black] (4.719292975910273,1.9945817103139172) ++(-3.5pt,0 pt) -- ++(3.5pt,3.5pt)--++(3.5pt,-3.5pt)--++(-3.5pt,-3.5pt)--++(-3.5pt,3.5pt);
\draw [fill=black] (-4,-1) ++(-3.5pt,0 pt) -- ++(3.5pt,3.5pt)--++(3.5pt,-3.5pt)--++(-3.5pt,-3.5pt)--++(-3.5pt,3.5pt);
\end{scriptsize}
\end{tikzpicture} 
 \caption {Relabeling after applications of \texttt{ReduceStarVertex}$(T, u_1)$ with $k=4$.}
\label{fig:an_example:0203}
\end{figure}

Now, in Figure~\ref{fig:an_example:0203} (right), we have $k=4$ and then
after applying \texttt{Reduce\-Star\-Ver\-tex}$(T, u_1)$,
the gpp $P_1 \ast S_2$ is created.
The starlike vertex $u_1$
is eliminated while a new starlike vertex is created having weight equal 7.
Then it is relabeled as $u_2$ again and $u_{2}$ is relabeled as $u_{1}$. Only
$u_{3}$ and  $u_{4}$ remain unchanged, as Figure~\ref{fig:an_example:0405}
(left) illustrates.

We still have $k=4$ and apply
\texttt{ReduceStarVertex}$(T, u_1)$,
creating the gpp
$P_2 \ast S_3$.
The starlike vertex $u_1$ is eliminated
and a new starlike vertex is created. Counting the weights we label the new
starlike vertices as $u_{1}$, $u_{2}$, $u_{3}$ and $u_{4}$ with weights
$7,8,9$~and~$17$, as in Figure~\ref{fig:an_example:0405}
(right).
\begin{figure}[!ht]
  \centering
\begin{tikzpicture}[line cap=round,line join=round,>=triangle 45,x=1cm,y=1cm,scale=0.6, every node/.style={scale=0.8}]
\clip(-4.9,-2.0) rectangle (6.5,3.44);
\draw [line width=1pt] (-4,0)-- (-4.44,-1.01);
\draw [line width=1pt] (-4,0)-- (-3.46,-1.03);
\draw (4.04,1.06) node[anchor=north west] {{\tiny $u_{3}$}};
\draw (-4.96,-1.12) node[anchor=north west] {{\tiny $3 \times P_0 \ast S_{1}$}};
\draw (-3.3,1.96) node[anchor=north west] {{\tiny $P_7 \ast S_{1}$}};
\draw [line width=1pt] (-4,0)-- (-3,0);
\draw [line width=1pt] (2,0)-- (3,0);
\draw [line width=1pt] (1,0)-- (1,1);
\draw [line width=1pt] (-0.02,-0.01)-- (-0.02,-1.01);
\draw [line width=1pt] (4,0)-- (5,0);
\draw [line width=1pt] (-1.98,0)-- (-1.98,1);
\draw [line width=1pt] (-3,0)-- (-1.98,0);
\draw [line width=1pt] (-1.98,0)-- (-1,0);
\draw [line width=1pt] (-1,0)-- (-0.02,-0.01);
\draw [line width=1pt] (-0.02,-0.01)-- (1,0);
\draw [line width=1pt] (1,0)-- (2,0);
\draw [line width=1pt] (4.014083413745475,0.9782502824881779)-- (3.779013559690542,2.008409348787737);
\draw [line width=1pt] (4.014083413745475,0.9782502824881779)-- (4.249153267800409,2.008409348787737);
\draw [line width=1pt] (4,0)-- (4.014083413745475,0.9782502824881779);
\draw [line width=1pt] (4,0)-- (3,0);
\draw [line width=1pt] (3.3503567670021344,1.9945817103139172)-- (4.014083413745475,0.9782502824881779);
\draw [line width=1pt] (4.014083413745475,0.9782502824881779)-- (4.719292975910273,1.9945817103139172);
\draw [line width=1pt] (-4,0)-- (-4,-1);
\draw (3.22,3.14) node[anchor=north west] {{\tiny $4 \times P_0 \ast S_{1}$}};
\draw (4.64,0.02) node[anchor=north west] {{\tiny $P_0 \ast S_{1}$}};
\draw (1.26,1.8) node[anchor=north west] {{\tiny $P_1 \ast S_{0}$}};
\draw (-0.06,-0.98) node[anchor=north west] {{\tiny $P_1 \ast S_{1}$}};
\draw (-4.7,0.58) node[anchor=north west] {{\tiny $u_{1}$}};
\draw (3.2,0.56) node[anchor=north west] {{\tiny $u_{2}$}};
\draw (-1.78,-1.04) node[anchor=north west] {{\tiny $P_1 \ast S_{2}$}};
\draw [line width=1pt] (-1,1)-- (-1,0);
\draw (-1.02,2) node[anchor=north west] {{\tiny $P_2 \ast S_{1}$}};
\draw [line width=1pt] (-1,0)-- (-1,-1);
\draw (3.4,-0.96) node[anchor=north west] {{\tiny $P_1 \ast S_{2}$}};
\draw [line width=1pt] (4,0)-- (4,-1);
\draw (-1.02,0.54) node[anchor=north west] {{\tiny $u_{4}$}};
\begin{scriptsize}
\draw [fill=black] (-4,0) circle (2.5pt);
\draw [fill=black] (-4.44,-1.01) ++(-3.5pt,0 pt) -- ++(3.5pt,3.5pt)--++(3.5pt,-3.5pt)--++(-3.5pt,-3.5pt)--++(-3.5pt,3.5pt);
\draw [fill=black] (-3.46,-1.03) ++(-3.5pt,0 pt) -- ++(3.5pt,3.5pt)--++(3.5pt,-3.5pt)--++(-3.5pt,-3.5pt)--++(-3.5pt,3.5pt);
\draw [fill=black] (-3,0) circle (2.5pt);
\draw [fill=black] (-1,0) circle (2.5pt);
\draw [fill=black] (2,0) circle (2.5pt);
\draw [fill=black] (3,0) circle (2.5pt);
\draw [fill=black] (-1,-1) ++(-3.5pt,0 pt) -- ++(3.5pt,3.5pt)--++(3.5pt,-3.5pt)--++(-3.5pt,-3.5pt)--++(-3.5pt,3.5pt);
\draw [fill=black] (-1,1) ++(-3.5pt,0 pt) -- ++(3.5pt,3.5pt)--++(3.5pt,-3.5pt)--++(-3.5pt,-3.5pt)--++(-3.5pt,3.5pt);
\draw [fill=black] (4,-1) ++(-3.5pt,0 pt) -- ++(3.5pt,3.5pt)--++(3.5pt,-3.5pt)--++(-3.5pt,-3.5pt)--++(-3.5pt,3.5pt);
\draw [fill=black] (1,0) circle (2.5pt);
\draw [fill=black] (1,1) ++(-3.5pt,0 pt) -- ++(3.5pt,3.5pt)--++(3.5pt,-3.5pt)--++(-3.5pt,-3.5pt)--++(-3.5pt,3.5pt);
\draw [fill=black] (-0.02,-0.01) circle (2.5pt);
\draw [fill=black] (-0.02,-1.01) ++(-3.5pt,0 pt) -- ++(3.5pt,3.5pt)--++(3.5pt,-3.5pt)--++(-3.5pt,-3.5pt)--++(-3.5pt,3.5pt);
\draw [fill=black] (4,0) circle (2.5pt);
\draw [fill=black] (5,0) ++(-3.5pt,0 pt) -- ++(3.5pt,3.5pt)--++(3.5pt,-3.5pt)--++(-3.5pt,-3.5pt)--++(-3.5pt,3.5pt);
\draw [fill=black] (-1.98,0) circle (2.5pt);
\draw [fill=black] (-1.98,1) ++(-3.5pt,0 pt) -- ++(3.5pt,3.5pt)--++(3.5pt,-3.5pt)--++(-3.5pt,-3.5pt)--++(-3.5pt,3.5pt);
\draw [fill=black] (4.014083413745475,0.9782502824881779) circle (2.5pt);
\draw [fill=black] (3.779013559690542,2.008409348787737) ++(-3.5pt,0 pt) -- ++(3.5pt,3.5pt)--++(3.5pt,-3.5pt)--++(-3.5pt,-3.5pt)--++(-3.5pt,3.5pt);
\draw [fill=black] (4.249153267800409,2.008409348787737) ++(-3.5pt,0 pt) -- ++(3.5pt,3.5pt)--++(3.5pt,-3.5pt)--++(-3.5pt,-3.5pt)--++(-3.5pt,3.5pt);
\draw [fill=black] (3.3503567670021344,1.9945817103139172) ++(-3.5pt,0 pt) -- ++(3.5pt,3.5pt)--++(3.5pt,-3.5pt)--++(-3.5pt,-3.5pt)--++(-3.5pt,3.5pt);
\draw [fill=black] (4.719292975910273,1.9945817103139172) ++(-3.5pt,0 pt) -- ++(3.5pt,3.5pt)--++(3.5pt,-3.5pt)--++(-3.5pt,-3.5pt)--++(-3.5pt,3.5pt);
\draw [fill=black] (-4,-1) ++(-3.5pt,0 pt) -- ++(3.5pt,3.5pt)--++(3.5pt,-3.5pt)--++(-3.5pt,-3.5pt)--++(-3.5pt,3.5pt);
\end{scriptsize}
\end{tikzpicture} \;\begin{tikzpicture}[line cap=round,line join=round,>=triangle 45,x=1cm,y=1cm,scale=0.6, every node/.style={scale=0.8}]
\clip(-4.9,-2.0) rectangle (6.5,3.44);
\draw (-0.96,0.66) node[anchor=north west] {{\tiny $u_{3}$}};
\draw (-4.1,0.08) node[anchor=north west] {{\tiny $P_2 \ast S_{3}$}};
\draw (-3.3,1.96) node[anchor=north west] {{\tiny $P_7 \ast S_{1}$}};
\draw [line width=1pt] (2,0)-- (3,0);
\draw [line width=1pt] (1,0)-- (1,1);
\draw [line width=1pt] (-0.02,-0.01)-- (-0.02,-1.01);
\draw [line width=1pt] (4,0)-- (5,0);
\draw [line width=1pt] (-1.98,0)-- (-1.98,1);
\draw [line width=1pt] (-1.98,0)-- (-1,0);
\draw [line width=1pt] (-1,0)-- (-0.02,-0.01);
\draw [line width=1pt] (-0.02,-0.01)-- (1,0);
\draw [line width=1pt] (1,0)-- (2,0);
\draw [line width=1pt] (4.014083413745475,0.9782502824881779)-- (3.779013559690542,2.008409348787737);
\draw [line width=1pt] (4.014083413745475,0.9782502824881779)-- (4.249153267800409,2.008409348787737);
\draw [line width=1pt] (4,0)-- (4.014083413745475,0.9782502824881779);
\draw [line width=1pt] (4,0)-- (3,0);
\draw [line width=1pt] (3.3503567670021344,1.9945817103139172)-- (4.014083413745475,0.9782502824881779);
\draw [line width=1pt] (4.014083413745475,0.9782502824881779)-- (4.719292975910273,1.9945817103139172);
\draw (3.22,3.14) node[anchor=north west] {{\tiny $4 \times P_0 \ast S_{1}$}};
\draw (4.64,0.02) node[anchor=north west] {{\tiny $P_0 \ast S_{1}$}};
\draw (1.26,1.8) node[anchor=north west] {{\tiny $P_1 \ast S_{0}$}};
\draw (-0.06,-0.98) node[anchor=north west] {{\tiny $P_1 \ast S_{1}$}};
\draw (3.34,0.64) node[anchor=north west] {{\tiny $u_{1}$}};
\draw (4.18,1.26) node[anchor=north west] {{\tiny $u_{2}$}};
\draw (-1.78,-1.04) node[anchor=north west] {{\tiny $P_1 \ast S_{2}$}};
\draw [line width=1pt] (-1,1)-- (-1,0);
\draw (-1.02,2) node[anchor=north west] {{\tiny $P_2 \ast S_{1}$}};
\draw [line width=1pt] (-1,0)-- (-1,-1);
\draw (3.4,-0.96) node[anchor=north west] {{\tiny $P_1 \ast S_{2}$}};
\draw [line width=1pt] (4,0)-- (4,-1);
\draw [line width=1pt] (-3,0)-- (-1.98,0);
\draw (-2.38,-0.12) node[anchor=north west] {{\tiny $u_{4}$}};
\begin{scriptsize}
\draw [fill=black] (-3,0) ++(-3.5pt,0 pt) -- ++(3.5pt,3.5pt)--++(3.5pt,-3.5pt)--++(-3.5pt,-3.5pt)--++(-3.5pt,3.5pt);
\draw [fill=black] (-1,0) circle (2.5pt);
\draw [fill=black] (2,0) circle (2.5pt);
\draw [fill=black] (3,0) circle (2.5pt);
\draw [fill=black] (-1,-1) ++(-3.5pt,0 pt) -- ++(3.5pt,3.5pt)--++(3.5pt,-3.5pt)--++(-3.5pt,-3.5pt)--++(-3.5pt,3.5pt);
\draw [fill=black] (-1,1) ++(-3.5pt,0 pt) -- ++(3.5pt,3.5pt)--++(3.5pt,-3.5pt)--++(-3.5pt,-3.5pt)--++(-3.5pt,3.5pt);
\draw [fill=black] (4,-1) ++(-3.5pt,0 pt) -- ++(3.5pt,3.5pt)--++(3.5pt,-3.5pt)--++(-3.5pt,-3.5pt)--++(-3.5pt,3.5pt);
\draw [fill=black] (1,0) circle (2.5pt);
\draw [fill=black] (1,1) ++(-3.5pt,0 pt) -- ++(3.5pt,3.5pt)--++(3.5pt,-3.5pt)--++(-3.5pt,-3.5pt)--++(-3.5pt,3.5pt);
\draw [fill=black] (-0.02,-0.01) circle (2.5pt);
\draw [fill=black] (-0.02,-1.01) ++(-3.5pt,0 pt) -- ++(3.5pt,3.5pt)--++(3.5pt,-3.5pt)--++(-3.5pt,-3.5pt)--++(-3.5pt,3.5pt);
\draw [fill=black] (4,0) circle (2.5pt);
\draw [fill=black] (5,0) ++(-3.5pt,0 pt) -- ++(3.5pt,3.5pt)--++(3.5pt,-3.5pt)--++(-3.5pt,-3.5pt)--++(-3.5pt,3.5pt);
\draw [fill=black] (-1.98,0) circle (2.5pt);
\draw [fill=black] (-1.98,1) ++(-3.5pt,0 pt) -- ++(3.5pt,3.5pt)--++(3.5pt,-3.5pt)--++(-3.5pt,-3.5pt)--++(-3.5pt,3.5pt);
\draw [fill=black] (4.014083413745475,0.9782502824881779) circle (2.5pt);
\draw [fill=black] (3.779013559690542,2.008409348787737) ++(-3.5pt,0 pt) -- ++(3.5pt,3.5pt)--++(3.5pt,-3.5pt)--++(-3.5pt,-3.5pt)--++(-3.5pt,3.5pt);
\draw [fill=black] (4.249153267800409,2.008409348787737) ++(-3.5pt,0 pt) -- ++(3.5pt,3.5pt)--++(3.5pt,-3.5pt)--++(-3.5pt,-3.5pt)--++(-3.5pt,3.5pt);
\draw [fill=black] (3.3503567670021344,1.9945817103139172) ++(-3.5pt,0 pt) -- ++(3.5pt,3.5pt)--++(3.5pt,-3.5pt)--++(-3.5pt,-3.5pt)--++(-3.5pt,3.5pt);
\draw [fill=black] (4.719292975910273,1.9945817103139172) ++(-3.5pt,0 pt) -- ++(3.5pt,3.5pt)--++(3.5pt,-3.5pt)--++(-3.5pt,-3.5pt)--++(-3.5pt,3.5pt);
\end{scriptsize}
\end{tikzpicture} 
\caption{Relabeling after applications of \texttt{ReduceStarVertex}$(T, u_1)$ with $k=4$.}\label{fig:an_example:0405}
\end{figure}
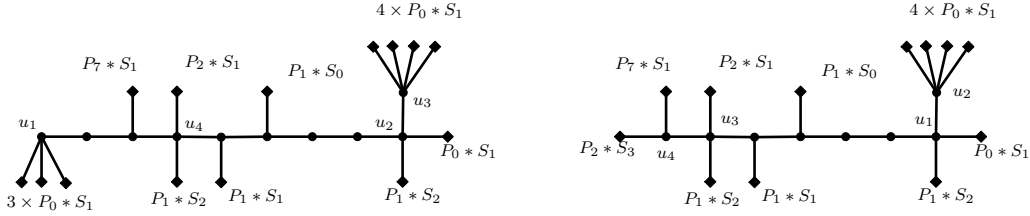

In Figure~\ref{fig:an_example:0405} (right), $k$ is four.
Applying  \texttt{ReduceStarVertex}$(T, u_1)$, eliminates
the starlike vertex $u_1$. The remaining three starlike vertices will be
labeled as $u_1$, $u_{2}$  and $u_{3}$  with weights $8,9$ and $17$, as in
Figure~\ref{fig:an_example:0607} (left).

We now have $k=3$ then we still
apply \texttt{ReduceStarVertex}$(T, u_1)$.
The starlike vertex $u_1$ is
eliminated
but a new starlike vertex is created as $P(u_2)=P_1 \ast S_3
\oplus P_1 \ast S_4$. Counting the weights we label new starlike vertices
$u_{1}$, $u_{2}$ and $u_{3}$ with weights
$9,16$~and~$17$.
This is
illustrated in Figure~\ref{fig:an_example:0607} (right).
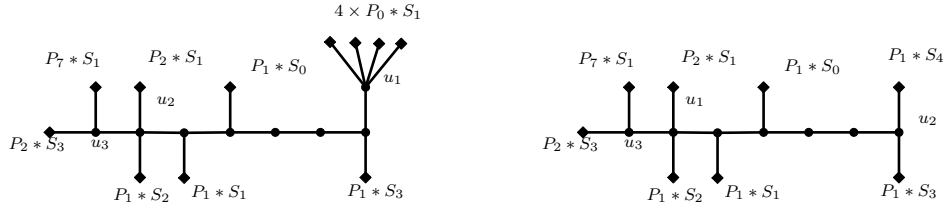
\begin{figure}[!h]
  \centering
\begin{tikzpicture}[line cap=round,line join=round,>=triangle 45,x=1cm,y=1cm,scale=0.6, every node/.style={scale=0.8}]
\clip(-4.9,-2.0) rectangle (6.5,3.44);
\draw (-2.28,0.04) node[anchor=north west] {{\tiny $u_{3}$}};
\draw (-4.1,0.08) node[anchor=north west] {{\tiny $P_2 \ast S_{3}$}};
\draw (-3.3,1.96) node[anchor=north west] {{\tiny $P_7 \ast S_{1}$}};
\draw [line width=1pt] (2,0)-- (3,0);
\draw [line width=1pt] (1,0)-- (1,1);
\draw [line width=1pt] (-0.02,-0.01)-- (-0.02,-1.01);
\draw [line width=1pt] (-1.98,0)-- (-1.98,1);
\draw [line width=1pt] (-1.98,0)-- (-1,0);
\draw [line width=1pt] (-1,0)-- (-0.02,-0.01);
\draw [line width=1pt] (-0.02,-0.01)-- (1,0);
\draw [line width=1pt] (1,0)-- (2,0);
\draw [line width=1pt] (4,1)-- (3.774868906623819,1.9849864155277592);
\draw [line width=1pt] (4,1)-- (4.300633049282667,1.9668566175050404);
\draw [line width=1pt] (4,0)-- (4,1);
\draw [line width=1pt] (4,0)-- (3,0);
\draw [line width=1pt] (3.2128451679195322,2.003116213550478)-- (4,1);
\draw [line width=1pt] (4,1)-- (4.790137595896079,1.9668566175050404);
\draw (3.12,3.02) node[anchor=north west] {{\tiny $4 \times P_0 \ast S_{1}$}};
\draw (1.26,1.8) node[anchor=north west] {{\tiny $P_1 \ast S_{0}$}};
\draw (-0.06,-0.98) node[anchor=north west] {{\tiny $P_1 \ast S_{1}$}};
\draw (4.2,1.48) node[anchor=north west] {{\tiny $u_{1}$}};
\draw (-0.82,0.98) node[anchor=north west] {{\tiny $u_{2}$}};
\draw (-1.78,-1.04) node[anchor=north west] {{\tiny $P_1 \ast S_{2}$}};
\draw [line width=1pt] (-1,1)-- (-1,0);
\draw (-1.02,2) node[anchor=north west] {{\tiny $P_2 \ast S_{1}$}};
\draw [line width=1pt] (-1,0)-- (-1,-1);
\draw (3.4,-0.96) node[anchor=north west] {{\tiny $P_1 \ast S_{3}$}};
\draw [line width=1pt] (4,0)-- (4,-1);
\draw [line width=1pt] (-3,0)-- (-1.98,0);
\begin{scriptsize}
\draw [fill=black] (-3,0) ++(-3.5pt,0 pt) -- ++(3.5pt,3.5pt)--++(3.5pt,-3.5pt)--++(-3.5pt,-3.5pt)--++(-3.5pt,3.5pt);
\draw [fill=black] (-1,0) circle (2.5pt);
\draw [fill=black] (2,0) circle (2.5pt);
\draw [fill=black] (3,0) circle (2.5pt);
\draw [fill=black] (-1,-1) ++(-3.5pt,0 pt) -- ++(3.5pt,3.5pt)--++(3.5pt,-3.5pt)--++(-3.5pt,-3.5pt)--++(-3.5pt,3.5pt);
\draw [fill=black] (-1,1) ++(-3.5pt,0 pt) -- ++(3.5pt,3.5pt)--++(3.5pt,-3.5pt)--++(-3.5pt,-3.5pt)--++(-3.5pt,3.5pt);
\draw [fill=black] (4,-1) ++(-3.5pt,0 pt) -- ++(3.5pt,3.5pt)--++(3.5pt,-3.5pt)--++(-3.5pt,-3.5pt)--++(-3.5pt,3.5pt);
\draw [fill=black] (1,0) circle (2.5pt);
\draw [fill=black] (1,1) ++(-3.5pt,0 pt) -- ++(3.5pt,3.5pt)--++(3.5pt,-3.5pt)--++(-3.5pt,-3.5pt)--++(-3.5pt,3.5pt);
\draw [fill=black] (-0.02,-0.01) circle (2.5pt);
\draw [fill=black] (-0.02,-1.01) ++(-3.5pt,0 pt) -- ++(3.5pt,3.5pt)--++(3.5pt,-3.5pt)--++(-3.5pt,-3.5pt)--++(-3.5pt,3.5pt);
\draw [fill=black] (4,0) circle (2.5pt);
\draw [fill=black] (-1.98,0) circle (2.5pt);
\draw [fill=black] (-1.98,1) ++(-3.5pt,0 pt) -- ++(3.5pt,3.5pt)--++(3.5pt,-3.5pt)--++(-3.5pt,-3.5pt)--++(-3.5pt,3.5pt);
\draw [fill=black] (4,1) circle (2.5pt);
\draw [fill=black] (3.774868906623819,1.9849864155277592) ++(-3.5pt,0 pt) -- ++(3.5pt,3.5pt)--++(3.5pt,-3.5pt)--++(-3.5pt,-3.5pt)--++(-3.5pt,3.5pt);
\draw [fill=black] (4.300633049282667,1.9668566175050404) ++(-3.5pt,0 pt) -- ++(3.5pt,3.5pt)--++(3.5pt,-3.5pt)--++(-3.5pt,-3.5pt)--++(-3.5pt,3.5pt);
\draw [fill=black] (3.2128451679195322,2.003116213550478) ++(-3.5pt,0 pt) -- ++(3.5pt,3.5pt)--++(3.5pt,-3.5pt)--++(-3.5pt,-3.5pt)--++(-3.5pt,3.5pt);
\draw [fill=black] (4.790137595896079,1.9668566175050404) ++(-3.5pt,0 pt) -- ++(3.5pt,3.5pt)--++(3.5pt,-3.5pt)--++(-3.5pt,-3.5pt)--++(-3.5pt,3.5pt);
\end{scriptsize}
\end{tikzpicture}\;\begin{tikzpicture}[line cap=round,line join=round,>=triangle 45,x=1cm,y=1cm,scale=0.6, every node/.style={scale=0.8}]
\clip(-4.9,-2.0) rectangle (6.5,3.44);
\draw (-2.28,0.04) node[anchor=north west] {{\tiny $u_{3}$}};
\draw (-4.1,0.08) node[anchor=north west] {{\tiny $P_2 \ast S_{3}$}};
\draw (-3.3,1.96) node[anchor=north west] {{\tiny $P_7 \ast S_{1}$}};
\draw [line width=1pt] (2,0)-- (3,0);
\draw [line width=1pt] (1,0)-- (1,1);
\draw [line width=1pt] (-0.02,-0.01)-- (-0.02,-1.01);
\draw [line width=1pt] (-1.98,0)-- (-1.98,1);
\draw [line width=1pt] (-1.98,0)-- (-1,0);
\draw [line width=1pt] (-1,0)-- (-0.02,-0.01);
\draw [line width=1pt] (-0.02,-0.01)-- (1,0);
\draw [line width=1pt] (1,0)-- (2,0);
\draw [line width=1pt] (4,0)-- (4,1);
\draw [line width=1pt] (4,0)-- (3,0);
\draw (1.26,1.8) node[anchor=north west] {{\tiny $P_1 \ast S_{0}$}};
\draw (-0.06,-0.98) node[anchor=north west] {{\tiny $P_1 \ast S_{1}$}};
\draw (-0.92,0.98) node[anchor=north west] {{\tiny $u_{1}$}};
\draw (4.24,0.56) node[anchor=north west] {{\tiny $u_{2}$}};
\draw (-1.78,-1.04) node[anchor=north west] {{\tiny $P_1 \ast S_{2}$}};
\draw [line width=1pt] (-1,1)-- (-1,0);
\draw (-1.02,2) node[anchor=north west] {{\tiny $P_2 \ast S_{1}$}};
\draw [line width=1pt] (-1,0)-- (-1,-1);
\draw (3.4,-0.96) node[anchor=north west] {{\tiny $P_1 \ast S_{3}$}};
\draw [line width=1pt] (4,0)-- (4,-1);
\draw [line width=1pt] (-3,0)-- (-1.98,0);
\draw (3.56,2.06) node[anchor=north west] {{\tiny $P_1 \ast S_{4}$}};
\begin{scriptsize}
\draw [fill=black] (-3,0) ++(-3.5pt,0 pt) -- ++(3.5pt,3.5pt)--++(3.5pt,-3.5pt)--++(-3.5pt,-3.5pt)--++(-3.5pt,3.5pt);
\draw [fill=black] (-1,0) circle (2.5pt);
\draw [fill=black] (2,0) circle (2.5pt);
\draw [fill=black] (3,0) circle (2.5pt);
\draw [fill=black] (-1,-1) ++(-3.5pt,0 pt) -- ++(3.5pt,3.5pt)--++(3.5pt,-3.5pt)--++(-3.5pt,-3.5pt)--++(-3.5pt,3.5pt);
\draw [fill=black] (-1,1) ++(-3.5pt,0 pt) -- ++(3.5pt,3.5pt)--++(3.5pt,-3.5pt)--++(-3.5pt,-3.5pt)--++(-3.5pt,3.5pt);
\draw [fill=black] (4,-1) ++(-3.5pt,0 pt) -- ++(3.5pt,3.5pt)--++(3.5pt,-3.5pt)--++(-3.5pt,-3.5pt)--++(-3.5pt,3.5pt);
\draw [fill=black] (1,0) circle (2.5pt);
\draw [fill=black] (1,1) ++(-3.5pt,0 pt) -- ++(3.5pt,3.5pt)--++(3.5pt,-3.5pt)--++(-3.5pt,-3.5pt)--++(-3.5pt,3.5pt);
\draw [fill=black] (-0.02,-0.01) circle (2.5pt);
\draw [fill=black] (-0.02,-1.01) ++(-3.5pt,0 pt) -- ++(3.5pt,3.5pt)--++(3.5pt,-3.5pt)--++(-3.5pt,-3.5pt)--++(-3.5pt,3.5pt);
\draw [fill=black] (4,0) circle (2.5pt);
\draw [fill=black] (-1.98,0) circle (2.5pt);
\draw [fill=black] (-1.98,1) ++(-3.5pt,0 pt) -- ++(3.5pt,3.5pt)--++(3.5pt,-3.5pt)--++(-3.5pt,-3.5pt)--++(-3.5pt,3.5pt);
\draw [fill=black] (4,1) ++(-3.5pt,0 pt) -- ++(3.5pt,3.5pt)--++(3.5pt,-3.5pt)--++(-3.5pt,-3.5pt)--++(-3.5pt,3.5pt);
\end{scriptsize}
\end{tikzpicture} 
\caption{Relabeling after applications of \texttt{ReduceStarVertex}$(T, u_1)$ with $k=4$ and $k=3$, respectively.} \label{fig:an_example:0607}
\end{figure}

In Figure~\ref{fig:an_example:0607} (right), we have $k=3$.
After \texttt{ReduceStarVertex}$(T, u_1)$, the starlike vertex $u_1$ is just
eliminated  and the remaining two starlike vertices  will be labeled as $u$
and $v$  with weights $16$ and $17$, as Figure~\ref{fig:an_example:0809}
(left) illustrates.

Since $k=2$,
we should apply
\texttt{ReduceStarVertex}$(T, u)$,
because $w(u)=16 \leq 2r=26$. After \texttt{ReduceStarVertex}$(T, u)$, the
starlike vertex $u$ is eliminated but a new starlike vertex is created.
Counting the weights, the remaining two starlike vertices  will be labeled as
$u$  and $v$
with weights $17$ and $20$,
Figure~\ref{fig:an_example:0809}
(right).
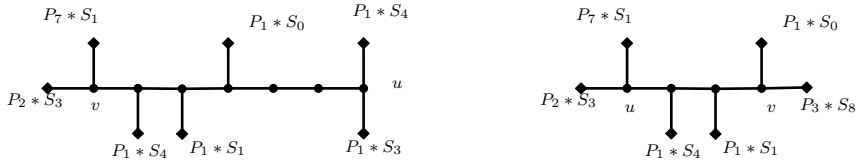
\begin{figure}[!ht]
  \centering
\begin{tikzpicture}[line cap=round,line join=round,>=triangle 45,x=1cm,y=1cm,scale=0.6, every node/.style={scale=0.8}]
\clip(-4.9,-2.0) rectangle (6.5,2.44);
\draw (-2.24,-0.12) node[anchor=north west] {{\tiny $v$}};
\draw (-4.1,0.08) node[anchor=north west] {{\tiny $P_2 \ast S_{3}$}};
\draw (-3.3,1.96) node[anchor=north west] {{\tiny $P_7 \ast S_{1}$}};
\draw [line width=1pt] (2,0)-- (3,0);
\draw [line width=1pt] (1,0)-- (1,1);
\draw [line width=1pt] (-0.02,-0.01)-- (-0.02,-1.01);
\draw [line width=1pt] (-1.98,0)-- (-1.98,1);
\draw [line width=1pt] (-1.98,0)-- (-1,0);
\draw [line width=1pt] (-1,0)-- (-0.02,-0.01);
\draw [line width=1pt] (-0.02,-0.01)-- (1,0);
\draw [line width=1pt] (1,0)-- (2,0);
\draw [line width=1pt] (4,0)-- (4,1);
\draw [line width=1pt] (4,0)-- (3,0);
\draw (1.26,1.8) node[anchor=north west] {{\tiny $P_1 \ast S_{0}$}};
\draw (-0.06,-0.98) node[anchor=north west] {{\tiny $P_1 \ast S_{1}$}};
\draw (4.44,0.36) node[anchor=north west] {{\tiny $u$}};
\draw (-1.78,-1.04) node[anchor=north west] {{\tiny $P_1 \ast S_{4}$}};
\draw [line width=1pt] (-1,0)-- (-1,-1);
\draw (3.4,-0.96) node[anchor=north west] {{\tiny $P_1 \ast S_{3}$}};
\draw [line width=1pt] (4,0)-- (4,-1);
\draw [line width=1pt] (-3,0)-- (-1.98,0);
\draw (3.56,2.06) node[anchor=north west] {{\tiny $P_1 \ast S_{4}$}};
\begin{scriptsize}
\draw [fill=black] (-3,0) ++(-3.5pt,0 pt) -- ++(3.5pt,3.5pt)--++(3.5pt,-3.5pt)--++(-3.5pt,-3.5pt)--++(-3.5pt,3.5pt);
\draw [fill=black] (-1,0) circle (2.5pt);
\draw [fill=black] (2,0) circle (2.5pt);
\draw [fill=black] (3,0) circle (2.5pt);
\draw [fill=black] (-1,-1) ++(-3.5pt,0 pt) -- ++(3.5pt,3.5pt)--++(3.5pt,-3.5pt)--++(-3.5pt,-3.5pt)--++(-3.5pt,3.5pt);
\draw [fill=black] (4,-1) ++(-3.5pt,0 pt) -- ++(3.5pt,3.5pt)--++(3.5pt,-3.5pt)--++(-3.5pt,-3.5pt)--++(-3.5pt,3.5pt);
\draw [fill=black] (1,0) circle (2.5pt);
\draw [fill=black] (1,1) ++(-3.5pt,0 pt) -- ++(3.5pt,3.5pt)--++(3.5pt,-3.5pt)--++(-3.5pt,-3.5pt)--++(-3.5pt,3.5pt);
\draw [fill=black] (-0.02,-0.01) circle (2.5pt);
\draw [fill=black] (-0.02,-1.01) ++(-3.5pt,0 pt) -- ++(3.5pt,3.5pt)--++(3.5pt,-3.5pt)--++(-3.5pt,-3.5pt)--++(-3.5pt,3.5pt);
\draw [fill=black] (4,0) circle (2.5pt);
\draw [fill=black] (-1.98,0) circle (2.5pt);
\draw [fill=black] (-1.98,1) ++(-3.5pt,0 pt) -- ++(3.5pt,3.5pt)--++(3.5pt,-3.5pt)--++(-3.5pt,-3.5pt)--++(-3.5pt,3.5pt);
\draw [fill=black] (4,1) ++(-3.5pt,0 pt) -- ++(3.5pt,3.5pt)--++(3.5pt,-3.5pt)--++(-3.5pt,-3.5pt)--++(-3.5pt,3.5pt);
\end{scriptsize}
\end{tikzpicture} \;\begin{tikzpicture}[line cap=round,line join=round,>=triangle 45,x=1cm,y=1cm,scale=0.6, every node/.style={scale=0.8}]
\clip(-4.9,-2.0) rectangle (6.5,2.44);
\draw (0.92,-0.18) node[anchor=north west] {{\tiny $v$}};
\draw (-4.1,0.08) node[anchor=north west] {{\tiny $P_2 \ast S_{3}$}};
\draw (-3.3,1.96) node[anchor=north west] {{\tiny $P_7 \ast S_{1}$}};
\draw [line width=1pt] (1,0)-- (1,1);
\draw [line width=1pt] (-0.02,-0.01)-- (-0.02,-1.01);
\draw [line width=1pt] (-1.98,0)-- (-1.98,1);
\draw [line width=1pt] (-1.98,0)-- (-1,0);
\draw [line width=1pt] (-1,0)-- (-0.02,-0.01);
\draw [line width=1pt] (-0.02,-0.01)-- (1,0);
\draw (1.26,1.8) node[anchor=north west] {{\tiny $P_1 \ast S_{0}$}};
\draw (-0.06,-0.98) node[anchor=north west] {{\tiny $P_1 \ast S_{1}$}};
\draw (-2.22,-0.16) node[anchor=north west] {{\tiny $u$}};
\draw (-1.78,-1.04) node[anchor=north west] {{\tiny $P_1 \ast S_{4}$}};
\draw [line width=1pt] (-1,0)-- (-1,-1);
\draw (1.66,-0.02) node[anchor=north west] {{\tiny $P_3 \ast S_{8}$}};
\draw [line width=1pt] (-3,0)-- (-1.98,0);
\draw [line width=1pt] (1,0)-- (2,0.02);
\begin{scriptsize}
\draw [fill=black] (-3,0) ++(-3.5pt,0 pt) -- ++(3.5pt,3.5pt)--++(3.5pt,-3.5pt)--++(-3.5pt,-3.5pt)--++(-3.5pt,3.5pt);
\draw [fill=black] (-1,0) circle (2.5pt);
\draw [fill=black] (-1,-1) ++(-3.5pt,0 pt) -- ++(3.5pt,3.5pt)--++(3.5pt,-3.5pt)--++(-3.5pt,-3.5pt)--++(-3.5pt,3.5pt);
\draw [fill=black] (1,0) circle (2.5pt);
\draw [fill=black] (1,1) ++(-3.5pt,0 pt) -- ++(3.5pt,3.5pt)--++(3.5pt,-3.5pt)--++(-3.5pt,-3.5pt)--++(-3.5pt,3.5pt);
\draw [fill=black] (-0.02,-0.01) circle (2.5pt);
\draw [fill=black] (-0.02,-1.01) ++(-3.5pt,0 pt) -- ++(3.5pt,3.5pt)--++(3.5pt,-3.5pt)--++(-3.5pt,-3.5pt)--++(-3.5pt,3.5pt);
\draw [fill=black] (2,0.02) ++(-3.5pt,0 pt) -- ++(3.5pt,3.5pt)--++(3.5pt,-3.5pt)--++(-3.5pt,-3.5pt)--++(-3.5pt,3.5pt);
\draw [fill=black] (-1.98,0) circle (2.5pt);
\draw [fill=black] (-1.98,1) ++(-3.5pt,0 pt) -- ++(3.5pt,3.5pt)--++(3.5pt,-3.5pt)--++(-3.5pt,-3.5pt)--++(-3.5pt,3.5pt);
\end{scriptsize}
\end{tikzpicture}
\caption{Relabeling after applications of \texttt{ReduceStarVertex}$(T, u_1)$
 with $k=3$ and $k=2$, respectively.}\label{fig:an_example:0809}
\end{figure}

In Figure~\ref{fig:an_example:0809} (right), we have $k=2$, and because $w(u)=17 \leq 2r=26$, we should
apply \texttt{ReduceStarVertex}$(T, u)$.
After \texttt{Re\-du\-ce\-Star\-Ver\-tex}$(T,
u)$, the starlike vertex $u$ is eliminated but a new starlike vertex is created.
Counting the weights, the remaining two starlike vertices  will be
labeled as $u$  and $v$  with weights $20$ and $27$, respectively, as in
Figure~\ref{fig:an_example:1011} (left).

\begin{figure}[!ht]
  \centering
\begin{tikzpicture}[line cap=round,line join=round,>=triangle 45,x=1cm,y=1cm,scale=0.6, every node/.style={scale=0.8}]
\clip(-4.9,-2.0) rectangle (6.5,2.44);
\draw (-0.94,0.86) node[anchor=north west] {{\tiny $v$}};
\draw (-3.06,-0.04) node[anchor=north west] {{\tiny $P_2 \ast S_{8}$}};
\draw [line width=1pt] (1,0)-- (1,1);
\draw [line width=1pt] (-0.02,-0.01)-- (-0.02,-1.01);
\draw [line width=1pt] (-1,0)-- (-0.02,-0.01);
\draw [line width=1pt] (-0.02,-0.01)-- (1,0);
\draw (1.26,1.8) node[anchor=north west] {{\tiny $P_1 \ast S_{0}$}};
\draw (-0.06,-0.98) node[anchor=north west] {{\tiny $P_1 \ast S_{1}$}};
\draw (1.12,-0.06) node[anchor=north west] {{\tiny $u$}};
\draw (-1.78,-1.04) node[anchor=north west] {{\tiny $P_1 \ast S_{4}$}};
\draw [line width=1pt] (-1,0)-- (-1,-1);
\draw (1.7,-0.04) node[anchor=north west] {{\tiny $P_3 \ast S_{8}$}};
\draw [line width=1pt] (-2,0)-- (-1,0);
\draw [line width=1pt] (1,0)-- (2,0);
\begin{scriptsize}
\draw [fill=black] (-2,0) ++(-3.5pt,0 pt) -- ++(3.5pt,3.5pt)--++(3.5pt,-3.5pt)--++(-3.5pt,-3.5pt)--++(-3.5pt,3.5pt);
\draw [fill=black] (-1,0) circle (2.5pt);
\draw [fill=black] (-1,-1) ++(-3.5pt,0 pt) -- ++(3.5pt,3.5pt)--++(3.5pt,-3.5pt)--++(-3.5pt,-3.5pt)--++(-3.5pt,3.5pt);
\draw [fill=black] (1,0) circle (2.5pt);
\draw [fill=black] (1,1) ++(-3.5pt,0 pt) -- ++(3.5pt,3.5pt)--++(3.5pt,-3.5pt)--++(-3.5pt,-3.5pt)--++(-3.5pt,3.5pt);
\draw [fill=black] (-0.02,-0.01) circle (2.5pt);
\draw [fill=black] (-0.02,-1.01) ++(-3.5pt,0 pt) -- ++(3.5pt,3.5pt)--++(3.5pt,-3.5pt)--++(-3.5pt,-3.5pt)--++(-3.5pt,3.5pt);
\draw [fill=black] (2,0) ++(-3.5pt,0 pt) -- ++(3.5pt,3.5pt)--++(3.5pt,-3.5pt)--++(-3.5pt,-3.5pt)--++(-3.5pt,3.5pt);
\end{scriptsize}
\end{tikzpicture}\;\begin{tikzpicture}[line cap=round,line join=round,>=triangle 45,x=1cm,y=1cm,scale=0.6, every node/.style={scale=0.8}]
\clip(-4.9,-2.0) rectangle (6.5,2.44);
\draw (-0.94,0.86) node[anchor=north west] {{\tiny $v$}};
\draw (-3.06,-0.04) node[anchor=north west] {{\tiny $P_2 \ast S_{8}$}};
\draw [line width=1pt] (-0.02,-0.01)-- (-0.02,-1.01);
\draw [line width=1pt] (-1,0)-- (-0.02,-0.01);
\draw (-0.06,-0.98) node[anchor=north west] {{\tiny $P_1 \ast S_{1}$}};
\draw (-0.02,0.86) node[anchor=north west] {{\tiny $u$}};
\draw (-1.78,-1.04) node[anchor=north west] {{\tiny $P_1 \ast S_{4}$}};
\draw [line width=1pt] (-1,0)-- (-1,-1);
\draw (0.98,-0.06) node[anchor=north west] {{\tiny $P_1 \ast S_{10}$}};
\draw [line width=1pt] (-2,0)-- (-1,0);
\draw [line width=1pt] (-0.02,-0.01)-- (1,0);
\begin{scriptsize}
\draw [fill=black] (-2,0) ++(-3.5pt,0 pt) -- ++(3.5pt,3.5pt)--++(3.5pt,-3.5pt)--++(-3.5pt,-3.5pt)--++(-3.5pt,3.5pt);
\draw [fill=black] (-1,0) circle (2.5pt);
\draw [fill=black] (-1,-1) ++(-3.5pt,0 pt) -- ++(3.5pt,3.5pt)--++(3.5pt,-3.5pt)--++(-3.5pt,-3.5pt)--++(-3.5pt,3.5pt);
\draw [fill=black] (-0.02,-0.01) circle (2.5pt);
\draw [fill=black] (-0.02,-1.01) ++(-3.5pt,0 pt) -- ++(3.5pt,3.5pt)--++(3.5pt,-3.5pt)--++(-3.5pt,-3.5pt)--++(-3.5pt,3.5pt);
\draw [fill=black] (1,0) ++(-3.5pt,0 pt) -- ++(3.5pt,3.5pt)--++(3.5pt,-3.5pt)--++(-3.5pt,-3.5pt)--++(-3.5pt,3.5pt);
\end{scriptsize}
\end{tikzpicture} 
\caption{Relabeling after applications of \texttt{ReduceStarVertex}$(T, u_1)$
 with $k=2$.}\label{fig:an_example:1011}
\end{figure}
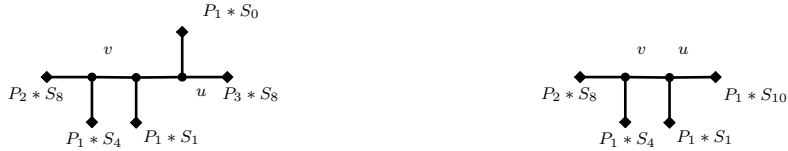
We still have $k=2$ so we should
apply \texttt{ReduceStarVertex}$(T, u)$,
because $w(u) = 20 \leq 26$.
After \texttt{ReduceStarVertex}$(T, u)$, the
starlike vertex $u$ is eliminated but a new starlike vertex is created.
The remaining two starlike vertices  will be labeled as
$u$  and $v$  with weights $24$ and $27$, as illustrated in
Figure~\ref{fig:an_example:1011} (right).

In Figure~\ref{fig:an_example:1011} (right), $k$ remains at two, so
we should apply \texttt{Re\-du\-ce\-Star\-Ver\-tex}$(T, u)$, because $w(u)=24 \leq 2r=26$.
After this transformation
the starlike vertex
$u$ is eliminated.
There remains  only one starlike vertex  which will be labeled $u$ and $T= u+
P_{2} \ast S_{8} \oplus P_{1}\ast S_{4} \oplus P_{1}\ast S_{12}$, which is
illustrated in Figure~\ref{fig:an_example:1213}~(left).
\begin{figure}[!ht]
  \centering
\begin{tikzpicture}[line cap=round,line join=round,>=triangle 45,x=1cm,y=1cm,scale=0.6, every node/.style={scale=0.8}]
\clip(-2.9,-1.4) rectangle (3.0,1.4);
\draw (-3.06,-0.04) node[anchor=north west] {{\tiny $P_2 \ast S_{8}$}};
\draw (-0.18,-0.16) node[anchor=north west] {{\tiny $P_1 \ast S_{12}$}};
\draw (-0.44,1.44) node[anchor=north west] {{\tiny $u$}};
\draw (-1.46,-0.58) node[anchor=north west] {{\tiny $P_1 \ast S_{4}$}};
\draw [line width=1pt] (-1,1)-- (-1.0111987877942101,-0.44598695065040167);
\draw [line width=1pt] (-2,0)-- (-1,1);
\draw [line width=1pt] (-1,1)-- (0,0);
\begin{scriptsize}
\draw [fill=black] (-2,0) ++(-3.5pt,0 pt) -- ++(3.5pt,3.5pt)--++(3.5pt,-3.5pt)--++(-3.5pt,-3.5pt)--++(-3.5pt,3.5pt);
\draw [fill=black] (-1,1) circle (2.5pt);
\draw [fill=black] (-1.0111987877942101,-0.44598695065040167) ++(-3.5pt,0 pt) -- ++(3.5pt,3.5pt)--++(3.5pt,-3.5pt)--++(-3.5pt,-3.5pt)--++(-3.5pt,3.5pt);
\draw [fill=black] (0,0) ++(-3.5pt,0 pt) -- ++(3.5pt,3.5pt)--++(3.5pt,-3.5pt)--++(-3.5pt,-3.5pt)--++(-3.5pt,3.5pt);
\end{scriptsize}
\end{tikzpicture} \;\begin{tikzpicture}[line cap=round,line join=round,>=triangle 45,x=1cm,y=1cm,scale=0.6, every node/.style={scale=0.8}]
\clip(-2.9,-1.4) rectangle (3.0,1.4);
\draw (-0.18,-0.16) node[anchor=north west] {{\tiny $P_1 \ast S_{12}$}};
\draw (-0.44,1.44) node[anchor=north west] {{\tiny $u$}};
\draw (-2.3,-0.24) node[anchor=north west] {{\tiny $P_1 \ast S_{13}$}};
\draw [line width=1pt] (-1,1)-- (-1.7,-0.16);
\draw [line width=1pt] (-1,1)-- (0,0);
\begin{scriptsize}
\draw [fill=black] (-1,1) circle (2.5pt);
\draw [fill=black] (-1.7,-0.16) ++(-3.5pt,0 pt) -- ++(3.5pt,3.5pt)--++(3.5pt,-3.5pt)--++(-3.5pt,-3.5pt)--++(-3.5pt,3.5pt);
\draw [fill=black] (0,0) ++(-3.5pt,0 pt) -- ++(3.5pt,3.5pt)--++(3.5pt,-3.5pt)--++(-3.5pt,-3.5pt)--++(-3.5pt,3.5pt);
\end{scriptsize}
\end{tikzpicture} 
\caption{Relabeling after application of \texttt{ReduceStarVertex}$(T, u_1)$
 with $k=2$, and a single starlike remaining.}\label{fig:an_example:1213}
\end{figure}
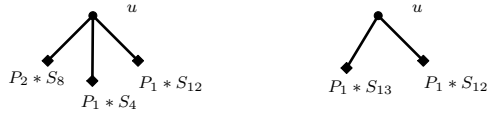

We now have $k=1$. Then, by Theorem~\ref{th:starlike}, we should apply the transformation
Star-up (Proposition \ref{star-up}) where it is possible.
In this case  $P_{2} \ast S_{8} \to P_{0} \ast
S_{9}$.
Therefore, $T$ is transformed to $u+ P_{0} \ast S_{9} \oplus P_{1}\ast S_{4} \oplus P_{1}\ast S_{12}$.
Here, $L=3$ and $\ell_0=3$ because $r=13$.
As prescribed by
Case 3 of
Theorem~\ref{th:starlike},
we perform
\texttt{Procedure ReduceStarVertex}
on two summands (e.g. case $\alpha=1$, (ii) $q_1=0, s_1=9, q_2=1, s_2=4, q'=1, s'=12$),
$P_{0} \ast S_{9} \oplus P_{1} \ast S_{4} \to  P_{1} \ast S_{13}$,
as Figure~\ref{fig:an_example:1213} (right) illustrates.

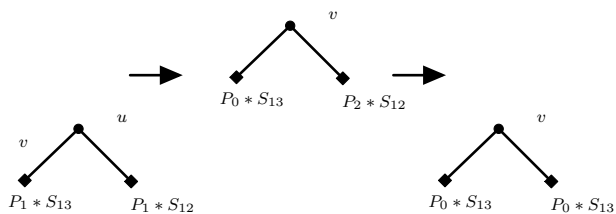
\begin{figure}[!h]
  \centering
\begin{tikzpicture}[line cap=round,line join=round,>=triangle 45,x=1cm,y=1cm,scale=0.7, every node/.style={scale=0.8}]
\clip(-9.9,-1.0) rectangle (6.5,3.4);
\draw (-5.16,-0.18) node[anchor=north west] {{\tiny $P_1 \ast S_{12}$}};
\draw (-5.42,1.42) node[anchor=north west] {{\tiny $u$}};
\draw (-7.48,-0.12) node[anchor=north west] {{\tiny $P_1 \ast S_{13}$}};
\draw [line width=1pt] (-5.98,0.98)-- (-7,0);
\draw [line width=1pt] (-5.98,0.98)-- (-4.98,-0.02);
\draw (-1.14,1.78) node[anchor=north west] {{\tiny $P_2 \ast S_{12}$}};
\draw (-1.4,3.38) node[anchor=north west] {{\tiny $v$}};
\draw (-3.46,1.84) node[anchor=north west] {{\tiny $P_0 \ast S_{13}$}};
\draw [line width=1pt] (-1.96,2.94)-- (-2.98,1.96);
\draw [line width=1pt] (-1.96,2.94)-- (-0.96,1.94);
\draw (-7.28,0.94) node[anchor=north west] {{\tiny $v$}};
\draw (2.82,-0.18) node[anchor=north west] {{\tiny $P_0 \ast S_{13}$}};
\draw (2.56,1.42) node[anchor=north west] {{\tiny $v$}};
\draw (0.5,-0.12) node[anchor=north west] {{\tiny $P_0 \ast S_{13}$}};
\draw [line width=1pt] (2,0.98)-- (0.98,0);
\draw [line width=1pt] (2,0.98)-- (3,-0.02);
\draw [->,line width=1pt] (-5,2) -- (-4,2);
\draw [->,line width=1pt] (0,2) -- (1,2);
%\draw (-4.96,3.12) node[anchor=north west] {\parbox{2.12 cm}{change of root}};
%\draw (0.18,2.66) node[anchor=north west] {star-up};
\begin{scriptsize}
\draw [fill=black] (-5.98,0.98) circle (2.5pt);
\draw [fill=black] (-7,0) ++(-3.5pt,0 pt) -- ++(3.5pt,3.5pt)--++(3.5pt,-3.5pt)--++(-3.5pt,-3.5pt)--++(-3.5pt,3.5pt);
\draw [fill=black] (-4.98,-0.02) ++(-3.5pt,0 pt) -- ++(3.5pt,3.5pt)--++(3.5pt,-3.5pt)--++(-3.5pt,-3.5pt)--++(-3.5pt,3.5pt);
\draw [fill=black] (-1.96,2.94) circle (2.5pt);
\draw [fill=black] (-2.98,1.96) ++(-3.5pt,0 pt) -- ++(3.5pt,3.5pt)--++(3.5pt,-3.5pt)--++(-3.5pt,-3.5pt)--++(-3.5pt,3.5pt);
\draw [fill=black] (-0.96,1.94) ++(-3.5pt,0 pt) -- ++(3.5pt,3.5pt)--++(3.5pt,-3.5pt)--++(-3.5pt,-3.5pt)--++(-3.5pt,3.5pt);
\draw [fill=black] (2,0.98) circle (2.5pt);
\draw [fill=black] (0.98,0) ++(-3.5pt,0 pt) -- ++(3.5pt,3.5pt)--++(3.5pt,-3.5pt)--++(-3.5pt,-3.5pt)--++(-3.5pt,3.5pt);
\draw [fill=black] (3,-0.02) ++(-3.5pt,0 pt) -- ++(3.5pt,3.5pt)--++(3.5pt,-3.5pt)--++(-3.5pt,-3.5pt)--++(-3.5pt,3.5pt);
\end{scriptsize}
\end{tikzpicture} 
\caption{Applications of \texttt{Star-up} and \texttt{ReduceStarVertex}.}\label{fig:an_example:end}
\end{figure}

Finally, we have $T=u+ P_{1} \ast S_{13} \oplus  P_{1}\ast S_{12}$,
and $k=0$ meaning that there are no starlike vertices anymore.
As $r_2 < r$ we are in Case 2 of Theorem~\ref{th:path}.
Recall that we change the root from $u$ to the vertex $v$ obtaining
$T'=v+ P_{0} \ast S_{13} \oplus P_{2}\ast S_{12}$.
Then we apply Star-up in $P_{2}\ast S_{12} \to P_{0} \ast S_{13}$.
We reached our goal.  That is, $T_1 = v + P_{0} \ast S_{13} \oplus  P_{0}\ast S_{13}$
as expected.
These final transformations are illustrated in
Figure~\ref{fig:an_example:end}.

\section*{Acknowledgments}

It just came to our knowledge, after acceptance of the manuscript, that the paper \cite{SIN2020} by Changhyon Sin, also proves the conjecture. We are grateful to the referees whose careful reading improved the paper, and pointed out some errors in the original submission. Vilmar Trevisan acknowledges partial support of CNPq grants  409746/2016-9 and 303334/2016-9,  CAPES-PRINT 88887.467572/2019-00, and FAPERGS PqG 17/2551-0001.
\bibliographystyle{acm}
\bibliography{refs}
\end{document}